\documentclass[a4paper,UKenglish,cleveref, autoref, nameinlink, thm-restate]{lipics-v2021}

\graphicspath{{./figures/}}

\usepackage{amsmath}
\usepackage{amsfonts}
\usepackage{amssymb}
\usepackage{amsthm}
\usepackage{graphicx}
\usepackage{todonotes}
\usepackage{boxedminipage}
\usepackage{bm}
\usepackage{subcaption}
\usepackage{xspace}
\usepackage{alphalph}

\usepackage{xcolor,soul}
\Crefname{property}{Property}{Properties}
\Crefname{observation}{Observation}{Observations}
\Crefname{theorem}{Theorem}{Theorems}
\Crefname{section}{Section}{Sections}
\Crefname{figure}{Figure}{Figures}
\Crefname{enumi}{Case}{Cases} 

\definecolor{defblue}{rgb}{0.121,0.47,0.705}
\definecolor{orangered}{rgb}{1,0.271,0}
\DeclareTextFontCommand{\emph}{\color{defblue}\em}

\definecolor{lipicsblue}{rgb}{0.08235294118,0.3098039216,0.537254902}

\hypersetup{colorlinks=true,
    linkcolor=lipicsblue,
    anchorcolor=lipicsblue,
    citecolor=lipicsblue,
    filecolor=lipicsblue,
    menucolor=lipicsblue,
    urlcolor=lipicsblue,
    bookmarksopen=true,
    bookmarksopenlevel=2,
    bookmarksnumbered=true,
    plainpages=false,
    }

\hideLIPIcs

\title{On Large Induced Outerplanar Subgraphs in~$2$-Outerplanar Graphs} 

\author{Marco {D'Elia}}{Roma Tre University, Rome, Italy}{marco.delia@uniroma3.it}{https://orcid.org/0009-0008-6266-3324}{}

\author{Fabrizio Frati}{Roma Tre University, Rome, Italy}{fabrizio.frati@uniroma3.it}{https://orcid.org/0000-0001-5987-8713}{}

\authorrunning{M. {D'Elia} and F. Frati}

\Copyright{Marco D'Elia and Fabrizio Frati}

\ccsdesc[500]{Theory of computation~Design and analysis of algorithms}
\ccsdesc[500]{Mathematics of computing~Combinatorics}
\ccsdesc[500]{Mathematics of computing~Graph theory}
\keywords{Induced graphs, planar graphs, outerplanar graphs} 

\category{} 

\relatedversion{} 

\acknowledgements{}

\nolinenumbers 

\begin{document}

\maketitle

\begin{abstract} 
Borradaile, Le and Sherman-Bennett [{\em Graphs and Combinatorics}, 2017] proved that every~$n$-vertex~$2$-outerplane graph has a set of at least~$2n/3$ vertices that induces an outerplane graph. We identify a major flaw in their proof and recover their result with a different, and unfortunately much more complex, proof.  
\end{abstract}

\section{Introduction}\label{se:introduction}

A famous conjecture of Albertson and Berman~\cite{albertson1979conjecture} states that every~$n$-vertex planar graph has a set with at least~$n/2$ vertices that induces a forest. The conjecture is the strongest possible, as there exist~$n$-vertex planar graphs that have no set with more than~$n/2$ vertices inducing a forest (for example, a planar graph composed of~$n/4$ vertex-disjoint~$K_4$'s). The conjecture has been proved for many classes of planar graphs, including outerplanar graphs~\cite{hosono1990induced}, cubic planar graphs~\cite{bhs-lb-87}, planar graphs with girth larger than~$3$~\cite{DBLP:journals/dam/DrossMP19,DBLP:journals/ipl/FertinGR02, DBLP:journals/jgt/KellyL18,DBLP:journals/dmtcs/KowalikLS10,DBLP:journals/gc/Le18,DBLP:journals/gc/Salavatipour06},~$n$-vertex planar graphs with at most~$\lfloor 7n/4 \rfloor$ edges~\cite{DBLP:journals/jgt/ShiX17}, bipartite planar graphs~\cite{DBLP:journals/gc/AkiyamaW87,wxy-ifbpg-17}, and~$2$-outerplanar graphs~\cite{DBLP:journals/gc/BorradaileLS17}, often with bounds that are even stronger than~$n/2$. A positive solution to Albertson and Berman's conjecture would imply, without using the~$4$-color theorem, the existence of an independent set with~$n/4$ vertices in any~$n$-vertex planar graph. Over the years, the conjecture has gained a prominent place in graph theory, many related conjectures have been stated, and many related results have been discovered, see, e.g., \cite{DBLP:journals/gc/AkiyamaW87,DBLP:journals/jgt/AngeliniEFG16,DBLP:journals/gc/BorradaileLS17,DBLP:journals/combinatorics/ChappellP13,DBLP:journals/dm/DrossMP19,DBLP:journals/corr/abs-2506-10471,DBLP:journals/jgt/GuKOQZ22,DBLP:journals/siamdm/KangMM12,DBLP:journals/corr/abs-2303-13655,DBLP:journals/combinatorics/LukotkaMZ15,DBLP:journals/gc/Pelsmajer04,DBLP:journals/jgt/Poh90}.

The best known lower bound for the size of a set of vertices that induces a forest in any~$n$-vertex planar graph is~$2n/5$. This is a consequence of a result of Borodin~\cite{borodin1976proof}, which asserts that every~$n$-vertex planar graph has an acyclic~$5$-coloring; in such a coloring, the two largest color classes induce a forest with at least~$2n/5$ vertices. Since there exist planar graphs with no acyclic~$4$-coloring, this line of attack to the Albertson and Berman's conjecture cannot yield bounds better than~$2n/5$. A promising alternative direction is the one of finding, in any~$n$-vertex planar graph~$G$, a large induced subgraph which belongs to a restricted planar graph class, for which it is known that a large induced forest can always be found. In particular, it is known that every~$n$-vertex outerplanar graph has an induced forest with at least~$2n/3$ vertices~\cite{hosono1990induced}. Thus, every~$n$-vertex planar graph has an induced forest with at least~$2cn/3$ vertices, where~$c$ is the maximum constant such that every~$n$-vertex planar graph has an induced outerplanar graph with at least~$cn$ vertices. Hence, any bound for~$c$ larger than~$3/5$ would improve the currently best known bound of~$2n/5$ for the Albertson and Berman's conjecture. It is not hard to observe that~$c\leq 2/3$ (as one of the two~$6$-vertex maximal planar graphs has a largest induced outerplanar graph with~$4$ vertices only) and that~$c\geq 1/2$ (this comes from a decomposition in {\em outerplanar layers} of the given graph). The best known lower bound for~$c$ is~$11/21$, due to Bose, Dujmovi\'c, Houdrouge, Morin, and Odak~\cite{DBLP:journals/corr/abs-2312-03399}. In fact, the~$11n/21$ lower bound~\cite{DBLP:journals/corr/abs-2312-03399}, as well as the previous~$n/2$ lower bound~\cite{DBLP:journals/jgt/AlbertsonBHT90,DBLP:journals/jgt/AngeliniEFG16} holds for the size of an {\em outerplane graph} that one is guaranteed to find in an~$n$-vertex planar graph. That is, fix a plane embedding of the~$n$-vertex planar graph; then the found subset induces an embedded planar graph in which every vertex is incident to the outer face. This property allows the result to be used for graph drawing applications, see~\cite{DBLP:journals/jgt/AngeliniEFG16,DBLP:journals/corr/abs-2312-03399,DBLP:journals/corr/abs-2506-10471}.  

A \emph{$2$-outerplanar graph} is a planar graph that admits a plane embedding such that, by removing the vertices incident to the outer face, one gets an outerplanar graph; the~$2$-outerplanar graph, equipped with such an embedding, is called \emph{$2$-outerplane graph}. The class of~$2$-outerplanar graphs has been studied by Borradaile, Le and Sherman-Bennett~\cite{DBLP:journals/gc/BorradaileLS17} as an interesting benchmark for the maximum size of an induced forest or induced outerplane graph that one is guaranteed to find in a planar graph. They proved the best possible results for both problems: Every~$n$-vertex~$2$-outerplane graph has a set of~$n/2$ vertices that induces a forest (thus proving Albertson and Berman's conjecture for the~$2$-outerplanar graphs) and has a set of~$2n/3$ vertices that induces an outerplane graph. Unfortunately, we point out a major flaw in their proof of the second result. What makes the flaw major is that the general strategy they use, that is, the set of vertices inducing an outerplane graph should contain all the vertices that are not incident to the outer face of the given~$2$-outerplane graph, plus, possibly, some additional vertices, cannot yield the desired bound. Indeed, we show that, for arbitrarily large values of~$n$, there exists an~$n$-vertex~$2$-outerplane graph such that any set of vertices that contains all the vertices that are not incident to the outer face and induces an outerplanar graph has at most~$7n/11$ vertices. 

The aim of this paper is to recover the following result.

\begin{theorem} \label{th:main}
Let~$G$ be an~$n$-vertex~$2$-outerplane graph. There exists a set~$\mathcal I\subseteq V(G)$ with~$|\mathcal I|\geq 2n/3$, such that the subgraph of~$G$ induced by~$\mathcal I$ is an outerplane graph. 
\end{theorem}

The rest of the paper is organized as follows. In~\cref{se:preliminaries} we present some definitions and preliminaries. In~\cref{se:flaw} we discuss more in detail the flaw in the algorithm by Borradaile, Le and Sherman-Bennett~\cite{DBLP:journals/gc/BorradaileLS17}. In~\cref{se:algorithm} we show a new proof for~\cref{th:main}. Finally, in~\cref{se:conclusions} we conclude and present some open problems.

\section{Preliminaries} 
\label{se:preliminaries}

We consider simple and finite graphs, and we use standard terminology from Graph Theory~\cite{Diestelbook}. We denote by~$V(G)$ the vertex set of a graph~$G$. Given a set~$\mathcal I \subseteq V(G)$, we denote by~$G[\mathcal I]$ the subgraph of~$G$ \emph{induced by~$\mathcal I$}, that is, the graph whose vertex set is~$\mathcal I$ and whose edge set consists of all edges~$uv$ in~$G$ such that~$u,v \in \mathcal I$. We denote by~$\delta_G(v)$ the \emph{degree} of a vertex~$v$ in~$G$, that is, the number of edges incident to~$v$. 

A \emph{drawing} of a graph maps each vertex to a point in the plane and each edge to a Jordan curve connecting its endpoints. A drawing is \emph{planar} if no two edges cross, except possibly at common endpoints. A planar drawing partitions the plane into connected regions, called \emph{faces}. The unbounded face is called the \emph{outer face}, while the other faces are \emph{internal}. 
A planar drawing is \emph{outerplanar} if every vertex is incident to the outer face.
An \emph{outerplanar graph} is a graph that admits an outerplanar drawing.
We call \emph{outerplane embedding} the topological information associated with such a drawing, and we call \emph{outerplane graph} a graph together with an outerplane embedding. 
A planar drawing is \emph{$2$-outerplanar} if removing all vertices incident to the outer face results in an outerplanar drawing.
A \emph{$2$-outerplanar graph} is a graph that admits a~$2$-outerplanar drawing. We call \emph{$2$-outerplane embedding} the topological information associated with such a drawing, and we call \emph{$2$-outerplane graph} a graph together with a~$2$-outerplane embedding. An induced subgraph~$G[\mathcal I]$ of a~$2$-outerplane graph~$G$ inherits the embedding of~$G$, thus we say that~$G[\mathcal I]$ is outerplane if restricting the embedding of~$G$ to the vertices and edges of~$G[\mathcal I]$ results in an outerplane graph. A~$2$-outerplane graph is \emph{internally-triangulated} if all its internal faces are~$3$-cycles. Given a~$2$-outerplane graph~$G$, we denote by~$L_1$ the set of vertices of~$G$ that lie on the outer face, and by~$L_2$ the set of vertices of~$G$ that do not lie on the outer face, i.e.,~$L_2 = V(G) - L_1$. In our proofs we will exploit several times the following observation. We say that two outerplane subgraphs of a~$2$-outerplane graph are \emph{each in the outer face of the other one} if every vertex and edge of each graph is in the outer face or on the boundary of the outer face of the other graph.

\begin{observation}\label{obs:union}
Let~$G_1$ and~$G_2$ be two outerplane subgraphs of a~$2$-outerplane graph such that: (i) each of~$G_1$ and~$G_2$ is in the outer face of the other one; and (ii)~$G_1$ and~$G_2$ do not share any vertex, or share a single vertex, or share a single edge. Then the subgraph of~$G$ composed of~$G_1$ and~$G_2$ is outerplane.
\end{observation}

A \emph{cutvertex} of a connected graph~$G$ is a vertex whose removal disconnects~$G$. A connected graph is \emph{biconnected} if it contains no cutvertex. A \emph{biconnected component} (or \emph{block}) of a graph~$G$ is a maximal (in terms of vertices and edges) biconnected subgraph of~$G$. A biconnected graph or a block is \emph{trivial} if it is a single edge, and \emph{non-trivial} otherwise. For a connected graph~$G$, the \emph{block-cutvertex tree}~$T_G$ is a tree that describes the arrangement of biconnected components of~$G$. It contains a B-node for each block of~$G$ and a C-node for each cutvertex of~$G$. A B-node~$B$ and a C-node~$c$ are adjacent in~$T_G$ if the cutvertex corresponding to~$c$ is a vertex of the block corresponding to~$B$. Note that the leaves of~$T_G$ are all B-nodes. We often identify a cutvertex with its corresponding C-node and a block with its corresponding B-node. 

We are going to exploit the following lemma.

\begin{lemma} \label{le:augmentation}
Let~$G$ be a~$2$-outerplane graph with~$n\geq 3$ vertices. It is possible to add edges to~$G$ so that it becomes an internally-triangulated~$2$-outerplane non-trivial biconnected graph.
\end{lemma}

\begin{proof}
Throughout the proof, let~$f_G$ denote the outer face of~$G$ and~$W_G$ the boundary of~$f_G$. We add edges to~$G$ in~$4$ steps.

{\bf (Step 1)} First, we augment~$G$ so that~$W_G$ is connected. This is done as follows. While there exist two connected components~$C_1$ and~$C_2$ that contain vertices~$v_1$ and~$v_2$ in~$L_1$, respectively, we add the edge~$v_1v_2$ to~$G$ and embed it in~$f_G$. This augmentation maintains the embedding~$2$-outerplane, since the edge~$v_1v_2$ does not prevent any vertex from being incident to~$f_G$ and is incident to two vertices in~$L_1$. Once Step 1 is concluded,~$W_G$ is connected.

{\bf (Step 2)} Second, we augment~$G$ so that~$W_G$ is a cycle. This is done as follows. We traverse~$W_G$ in clockwise direction. Whenever we encounter two consecutive edges~$v_1u$ and~$uv_2$ that belong to distinct blocks~$B_1$ and~$B_2$ of~$G$, respectively, we add the edge~$v_1v_2$ to~$G$ and embed it in~$f_G$, so that the path~$v_1uv_2$ is replaced by the edge~$v_1v_2$ in~$W_G$. This augmentation maintains the embedding~$2$-outerplane, since the edge~$v_1v_2$ does not prevent any vertex from being incident to~$f_G$ and is incident to two vertices in~$L_1$. In particular, an incidence of~$u$ on~$f_G$ is turned into an incidence of~$u$ on the internal face delimited by the cycle~$v_1uv_2$; however, since the edges~$v_1u$ and~$uv_2$ are incident to~$f_G$, the edges~$uw_1$ and~$uw_2$ that respectively follow and precede~$v_1u$ and~$uv_2$ in clockwise direction along the boundaries of~$B_1$ and~$B_2$ (such edges are again~$uw_1$ and/or~$uw_2$ if~$B_1$ and/or~$B_2$ are trivial), are also incident to~$f_G$, hence~$u$ remains incident to~$f_G$. Once Step 2 is concluded,~$W_G$ is a cycle.

{\bf (Step 3)} Third, while~$G$ contains an internal face~$f$ whose boundary contains two vertices~$v_1$ and~$v_2$ such that at least one of them is in~$L_1$ and such that the edge~$v_1v_2$ does not belong to~$G$, we add to~$G$ the edge~$v_1v_2$ and embed it inside~$f$. This augmentation maintains the embedding~$2$-outerplane, since the edge~$v_1v_2$ does not prevent any vertex from being incident to~$f_G$, given that it is embedded inside an internal face of~$G$, and since it is incident to at least one vertex in~$L_1$. Once Step 3 is concluded, every vertex in~$L_1$ that is on the boundary of an internal face~$f$ of~$G$ is adjacent to all the vertices on the boundary of~$f$.

\begin{figure}[tb]
\centering
\includegraphics[scale=.8]{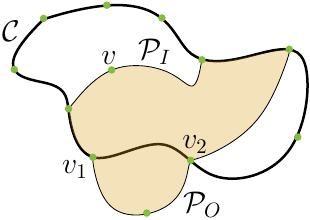}
\caption{Illustration for the correctness of Step 4 in the proof of \cref{le:augmentation}. The cycle~$\mathcal C$ is represented by a thick line and the face~$f$ is shaded orange.}
\label{fig:augmentation}
\end{figure}

{\bf (Step 4)} Finally, for every internal face~$f$ whose boundary contains at least~$4$ distinct vertices, we add an edge between two vertices~$v_1$ and~$v_2$ that are incident to~$f$ and that are not already adjacent in~$G$ (note that~$v_1$ and~$v_2$ are in~$L_2$, as otherwise Step 3 would not be concluded). The existence of such a pair of vertices is a standard argument: The vertices incident to~$f$ cannot induce a clique, since they are at least~$4$ and since they induce an outerplanar subgraph of~$G$. We prove that this augmentation maintains the embedding~$2$-outerplane. Suppose, for a contradiction, that this is not the case; refer to~\cref{fig:augmentation}. The edge~$v_1v_2$ does not prevent any vertex from being incident to~$f_G$, since it is embedded inside an internal face of~$G$. Hence, if the embedding is not~$2$-outerplane, there exists a cycle~$\mathcal C$ in~$G[L_2]$ such that the bounded region~$\mathcal R$ delimited by~$\mathcal C$ contains a vertex~$v$ in its interior. Suppose that~$\mathcal C$ is {\em minimal}, i.e., no cycle~$\mathcal C'\neq \mathcal C$ is such that the bounded region delimited by~$\mathcal C'$ contains~$v$ in its interior and all the edges of~$\mathcal C'$ are in the closure of~$\mathcal R$. Indeed, if such a cycle~$\mathcal C'$ exists, we can use it in place of~$\mathcal C$; the repetition of this argument eventually leads to a minimal cycle~$\mathcal C$.  Note that~$v_1v_2$ is an edge of~$\mathcal C$, as otherwise~$G$ would not be~$2$-outerplane even before the augmentation. Also,~$f$ is externally delimited by a cycle which consists of two paths between~$v_1$ and~$v_2$, one, say~$\mathcal P_I$, whose vertices and edges are all in the closure of~$\mathcal R$ and one, say~$\mathcal P_O$, whose vertices and edges are all in the closure of the complement of~$\mathcal R$. We have that~$v$ is a vertex of~$\mathcal P_I$, as otherwise~$\mathcal P_I$, together with the edge~$v_1v_2$ or with the path~$\mathcal P$ obtained from~$\mathcal C$ by removing the edge~$v_1v_2$, would create a cycle~$\mathcal C'$ such that the bounded region delimited by~$\mathcal C'$ contains~$v$ in its interior and all the edges of~$\mathcal C'$ are in the closure of~$\mathcal R$, contradicting the minimality of~$\mathcal C$. This implies that~$\mathcal P_O$ consists only of vertices in~$L_2$, as if it contained a vertex~$u$ in~$L_1$, then the edge~$uv$ would not belong to~$G$, since~$v$ and~$u$ are one inside and one outside~$\mathcal C$, hence we would be in Step 3 and the edge~$uv$, rather than~$v_1v_2$, would be added to~$G$. Then the paths~$\mathcal P$ and~$\mathcal P_O$ define a walk in~$G[L_2]$ which bounds a region which contains~$v$ in its interior. Hence,~$G$ is not~$2$-outerplane before the augmentation, a contradiction. 

Once Step 4 is concluded,~$G$ is internally-triangulated. Since~$n\geq 3$ and~$W_G$ is a cycle, we have that~$G$ is a non-trivial biconnected graph and concludes the proof of the lemma.  
\end{proof}

\section{On the Proof of~\cref{th:main} by Borradaile, Le and Sherman-Bennett} \label{se:flaw}

In this section, we discuss the flaw in the proof of~\cref{th:main} presented by Borradaile, Le and Sherman-Bennett \cite{DBLP:journals/gc/BorradaileLS17}. There are three distinct levels for this discussion.

First, at the highest level, their general strategy, which they describe as ``{\em To find the vertices inducing a large
outerplanar graph in~$G$, we delete vertices in~$L_1$ until all vertices in~$L_2$ are “exposed” to the external face}'' does not, unfortunately, allow one to find a set of~$2n/3$ vertices that induces an outerplane graph. Indeed, consider the~$2$-outerplane graph~$G$ in~\cref{fig:flaw1}. Each vertex in~$L_{1}$ prevents a distinct vertex in~$L_2$ from being incident to the outer face. Thus, if a set of vertices in~$G$ induces an outerplane graph and contains all the vertices in~$L_2$, it does not contain any of the vertices in~$L_1$, and thus it contains at most~$7$ of the~$11$ vertices of~$G$. Note that~$7/11<2/3$. Furthermore, this counter-example holds even if one wants to find an induced {\em outerplanar}, rather than {\em outerplane}, graph in~$G$. In fact, any vertex of~$L_1$ induces, together with~$4$ vertices in~$L_2$, a subgraph of~$G$ containing~$K_{2,3}$. Replicating this example as shown in~\cref{fig:flaw2} allows us to state the following.

\begin{theorem} \label{th:flaw}
For every positive~$n$ multiple of~$11$, there exists an~$n$-vertex~$2$-outerplane graph~$G$ such that every set of vertices in~$G$ which induces an outerplanar graph and includes all the vertices not incident to the outer face of~$G$ contains at most~$\frac{7n}{11}$ vertices.    
\end{theorem}

\begin{figure}[htb]
\centering
    \begin{subfigure}{0.48\textwidth}
		\centering
		\includegraphics[scale=1, page=1]{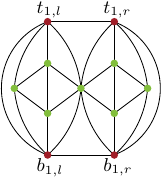}
		\subcaption{}
        \label{fig:flaw1}
	\end{subfigure}
    \begin{subfigure}{0.48\textwidth}
		\centering
		\includegraphics[scale=0.85, page=3]{figures/counterexample.pdf}
		\subcaption{}
        \label{fig:flaw2}
	\end{subfigure}
\caption{An~$n$-vertex~$2$-outerplane graph on which the algorithm by Borradaile et al.\ fails to find a set of~$2n/3$ vertices that induces an outerplanar graph. In this and the following figures, vertices depicted in red and green belong to~$L_1$ and~$L_2$, respectively.}
\label{fig:flaw}
\end{figure}

Second, at an intermediate level, the algorithm by Borradaile et al.\ relies on defining triples of vertices, so that each vertex in~$L_2$ occurs in exactly one triple; each triple contains either a single vertex in~$L_2$ and two of its neighbors in~$L_1$, or two vertices in~$L_2$ and a common neighbor in~$L_1$. Once a vertex in~$L_1$ is deleted, the vertices in~$L_2$ in the same triple as that vertex are exposed on the outer face. In order to find these triples, they rely on a structural lemma (\cite[Lemma 3]{DBLP:journals/gc/BorradaileLS17}), which asserts that there exists a matching on the boundary of~$G[L_2]$ which spans every vertex of~$L_2$ that has a unique neighbor in~$L_1$. The graph~$G$ shown in~\cref{fig:flaw1} is a counter-example to this statement. Indeed,~$L_2$ consists of~$7$ vertices; also,~$4$ of such vertices have a unique neighbor in~$L_1$ and no two of such~$4$ vertices are adjacent on the boundary of~$G[L_2]$. Hence, no matching exists on the boundary of~$G[L_2]$ which spans all such~$4$ vertices. 

Third, at a lowest level, their proof of \cite[Lemma 3]{DBLP:journals/gc/BorradaileLS17} works by induction on the number of vertices with a unique neighbor in~$L_1$. If there is at least one of such vertices, then its neighbors on the boundary of~$G[L_2]$ can be contracted into it, resulting in a graph with one less vertex with a unique neighbor in~$L_1$. The contraction might result in parallel edges and loops, which the authors remove from the graph. However, such a removal might result in a~$2$-outerplane graph which is not internally-triangulated any longer. Hence, the tools they use to find the next vertices to be contracted \cite[Observation 16, Lemma 1, Lemma 2]{DBLP:journals/gc/BorradaileLS17} are not applicable to the obtained graph and the induction breaks down. 

{\bf Remark.} The described flaw in the proof of \cite[Lemma 3]{DBLP:journals/gc/BorradaileLS17} does not occur if~$G[L_2]$ is biconnected or, more in general, a set of pairwise-disjoint biconnected components, since in this case the contractions they use preserve the property that the graph is internally-triangulated. Thus, the algorithm by Borradaile et al.\ successfully finds a set of~$2n/3$ vertices that induces an outerplane graph if the input~$n$-vertex graph~$G$ is such that~$G[L_2]$ is a set of pairwise-disjoint biconnected components. Unfortunately, we cannot use this result as a black box in our algorithm for general~$2$-outerplane graphs, since our induction needs to ensure stronger properties on the large set of vertices inducing an outerplane graph. 

\section{A New Proof of~\cref{th:main}} \label{se:algorithm}

In this section, we show an algorithm that, given an~$n$-vertex~$2$-outerplane graph~$G$, computes a set~$\mathcal I\subseteq V(G)$ with~$|\mathcal I|\geq 2n/3$, such that~$G[\mathcal I]$ is an outerplane graph. This proves~\cref{th:main}. We will not explicitly discuss the running time of our algorithm, however, it is easy to see that it can be implemented to run in polynomial time.

The section is organized as follows. In~\cref{se:definitions}, we introduce some definitions. In~\cref{se:base}, we present the case distinction used by our inductive algorithm. Finally, in~\cref{se:induction}, we discuss the inductive cases of our algorithm.

\subsection{Definitions} \label{se:definitions}

Our algorithm works by induction on~$n$. In the base case, we have~$n\leq 2$ and then the entire vertex set of~$G$ induces an outerplane graph. If~$n \geq 3$, by~\cref{le:augmentation} we can assume that~$G$ is an internally-triangulated non-trivial biconnected graph. Given a set of vertices~$\mathcal I \subseteq V(G)$, we say that~$\mathcal I$ is a \emph{good set} if~$|\mathcal I| \geq \frac{2}{3}n$ and if~$G[\mathcal I]$ is outerplane.

\begin{figure}[tb]
\centering
    \begin{subfigure}{0.48\textwidth}
		\centering
		\includegraphics[scale=.7, page=1]{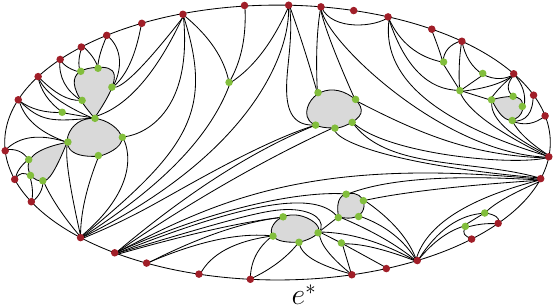}
		\subcaption{}
        \label{fig:setting-1}
	\end{subfigure}
    \hfill
    \begin{subfigure}{0.48\textwidth}
		\centering
		\includegraphics[scale=.7, page=2]{Setting.pdf}
		\subcaption{}
        \label{fig:setting-2}
	\end{subfigure}
    \hfill
    \begin{subfigure}{0.48\textwidth}
		\centering
		\includegraphics[scale=1, page=3]{Setting.pdf}
		\subcaption{}
        \label{fig:setting-3}
	\end{subfigure}
    \hfill
    \begin{subfigure}{0.48\textwidth}
		\centering
		\includegraphics[scale=1, page=4]{Setting.pdf}
		\subcaption{}
        \label{fig:setting-4}
	\end{subfigure}
\caption{(a) An internally-triangulated~$2$-outerplane graph~$G$ with a distinguished edge~$e^*$ along its outer face. (b) The outerplane graph~$G[L_1]$, together with its rooted weak dual tree~$T^*$, whose nodes are empty squares and whose edges are thick and red. (c)-(d) Two terminal components, say~$K_c$ and~$K_d$, of~$G[L_2]$, together with their rooted block-cutvertex trees~$T_{K_c}$ and~$T_{K_d}$, respectively. The extremal leaves of~$T_{K_c}$ are~$B_2$,~$B_3$, and~$B_4$, while~$B_8$ is the only extremal leaf of~$T_{K_d}$.}
\label{fig:setting}
\end{figure}

We introduce some definitions. Let~$G$ be an internally-triangulated~$2$-outerplane graph with a distinguished edge~$e^*$ incident to its outer face, as the one in \cref{fig:setting-1}. The \emph{weak dual}~$T$ of~$G[L_1]$ is the tree that has a node for each internal face of~$G[L_1]$ and an edge between two nodes if the corresponding faces share an edge on their boundaries; see \cref{fig:setting-2}. We root~$T$ at the node~$r^*$ corresponding to the unique internal face of~$G[L_1]$ that has~$e^*$ on its boundary; we denote by~$T^*$ the tree~$T$ once rooted at~$r^*$. Consider a leaf~$l$ of~$T^*$, which corresponds to a face~$f$ of~$G[L_1]$. The face~$f$ might be also a face of~$G$, or it might contain a connected component~$K$ of~$G[L_2]$. In the latter case, we say that~$K$ is a \emph{terminal component} of~$G[L_2]$. Consider a terminal component~$K$ of~$G[L_2]$ embedded in~$G$ inside a face~$f$ of~$G[L_1]$ corresponding to a leaf~$l$ of~$T^*$; see the examples in \cref{fig:setting-3,fig:setting-4}. The rooting of~$T^*$ induces a rooting of the block-cutvertex tree~$T_K$ of~$K$ as follows. Let~$xy$ be the edge of~$G[L_1]$ dual to the edge of~$T^*$ from~$l$ to its parent (or let~$e^*=xy$ if~$l$ is the root of~$T^*$), and let~$z$ be the vertex of~$K$ such that the~$3$-cycle~$xyz$ bounds an internal face of~$G$. If~$z$ is a cutvertex of~$K$, then~$T_K$ is rooted at the corresponding C-node, otherwise~$T_K$ is rooted at the unique B-node containing~$z$. A leaf~$B$ in the rooted tree~$T_K$ is \emph{extremal} if it has maximum depth (i.e., distance from the root) among the leaves of~$T_K$. Note that the children of the parent of an extremal leaf are also extremal leaves.

Consider a terminal component~$K$ of~$G[L_2]$ and suppose that~$K$ is not biconnected. Let~$B$ be an extremal leaf of the rooted block-cutvertex tree~$T_K$ of~$K$. The unique neighbor of~$B$ in~$T_K$ is called \emph{link-vertex} of~$B$ and denoted by~$c_B$. 
If~$B$ is a trivial block, it corresponds to an edge~$c_Bd_B$. If~$B$ is a non-trivial block, the neighbors of~$c_B$ on the boundary of the outer face of~$B$ are called  \emph{leftmost} and \emph{rightmost} neighbors of~$c_B$ in~$B$ and denoted by~$u_B$ and~$v_B$, where~$u_B$,~$v_B$, and~$c_B$ appear in this clockwise order along the outer face of~$B$. Let~$u'_B$ be the vertex that follows~$u_B$ in clockwise order along the outer face of~$B$ and let~$v'_B$ be the vertex that follows~$v_B$ in counter-clockwise order along the outer face of~$B$. 

If~$B$ is a trivial (non-trivial) block, let~$c_B\ell_B$ be the edge of~$G$ that precedes~$c_Bd_B$ (resp.~$c_Bu_B$) in clockwise direction around~$c_B$ and note that~$\ell_B\in L_1$. We call~$\ell_B$ the \emph{left cage vertex} of~$B$. Symmetrically, let~$c_Br_B$ be the edge of~$G$ that follows~$c_Bd_B$ (resp.~$c_Bv_B$) in clockwise direction around~$c_B$ and note that~$r_B\in L_1$. We call~$r_B$ the \emph{right cage vertex} of~$B$. The path obtained by traversing in clockwise direction the outer face of~$G$ from~$\ell_B$ to~$r_B$ is called the \emph{cage path} of~$B$ and denoted by~$\mathcal P_B$. Let~$\ell'_B$ be the vertex that precedes~$\ell_B$ in clockwise order along the outer face of~$G$, and let~$r'_B$ be the vertex that follows~$r_B$ in clockwise order along the outer face of~$G$. Finally, the subgraph of~$G$ whose vertices and edges are inside or on the boundary of the cycle~$\mathcal P_B\cup r_Bc_B\ell_B$ is called the \emph{cage graph} of~$B$ and denoted by~$G_B$. The fact that~$K$ is a terminal component of~$G[L_2]$ and that~$B$ is a leaf of~$T_K$ implies that~$G_B$ does not contain edges between two non-consecutive vertices of~$\mathcal P_B$; hence, every internal face of~$G$ that has an edge of~$\mathcal P_B$ on its boundary has a vertex of~$B$ as its third incident vertex. 

If~$B$ is non-trivial, then we say that~$B$ is \emph{pesky} (see~\cref{fig:pesky-component}) if it satisfies the following properties (the first two properties actually imply the other two, but we state all four of them, so to better describe the structure of~$G_B$):

\begin{itemize}
\item first,~$\delta_{G_B}(\ell_B)=\delta_{G_B}(r_B)=4$;
\item second, for each internal vertex~$w$ of~$\mathcal P_B$, we have~$\delta_{G_B}(w)=5$;
\item third,~$u_B$ and~$v_B$ have only~$\ell_B$ and~$r_B$, respectively, as neighbors in~$L_1$; and
\item finally, traversing~the boundary of the outer face of~$B$ in clockwise direction from~$u_B$ to~$v_B$, the encountered vertices have alternately~$1$ and~$2$ neighbors in~$L_1$. 
\end{itemize}

\begin{figure}[tb]
\centering
    \begin{subfigure}{0.48\textwidth}
		\centering
		\includegraphics[scale=1, page=1]{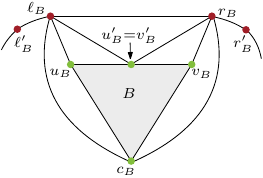}
		\subcaption{}
        \label{fig:pesky1}
	\end{subfigure}
    \hfill
    \begin{subfigure}{0.48\textwidth}
		\centering
		\includegraphics[scale=1, page=2]{figures/bad-component-revised.pdf}
		\subcaption{}
        \label{fig:pesky2}
	\end{subfigure}
\caption{Two pesky blocks.}
\label{fig:pesky-component}
\end{figure}

We observe the following. 

\begin{observation} \label{obs:pesky}
If~$B$ is pesky, the number of vertices of~$B$ is exactly twice the number of vertices of~$\mathcal P_B$. 
\end{observation}

If~$B$ is non-trivial and not pesky, then we say that it is \emph{cushy}.

\subsection{Case~distinction} \label{se:base}

We now discuss the case distinction employed by our inductive algorithm. Suppose that~$G$ is a non-trivial biconnected graph. In all the inductive cases, the basic strategy is the one of removing some vertices of the input~$2$-outerplane graph~$G$, obtaining a smaller~$2$-outerplane graph~$H$, of then applying induction on~$H$, so to get a good set~$\mathcal I_H$ for~$H$, and of then defining a good set~$\mathcal I$ for~$G$ by adding to~$\mathcal I_H$ some of the initially removed vertices. Two initial inductive cases are used to deal with the situation in which~$G$ contains a vertex~$v\in L_1$ whose degree in~$G$ is~$2$ ({\bf Case~1}) or~$3$ ({\bf Case~2}). This allows us to assume that every vertex of~$L_1$ has degree at least~$4$ in~$G$. 

Next, consider a leaf~$l$ of the rooted weak dual tree~$T^*$ of~$G[L_1]$, and let~$f$ be the face of~$G[L_1]$ corresponding to~$l$. Since no vertex of~$L_1$ has degree~$2$ in~$G$, it follows that~$f$ contains a terminal component~$K$. Let~$xy$ be the edge of~$G[L_1]$ dual to the edge of~$T^*$ from the node corresponding to~$f$ to its parent (or let~$xy$ be the edge~$e^*$ if the node corresponding to~$f$ is the root of~$T^*$). Note that~$K$ is not a single vertex, otherwise all its neighbors in~$L_1$ different from~$x$ and~$y$ (there exists at least one such a neighbor) would have degree~$3$ in~$G$, which is excluded. If~$K$ is a single edge~$uv$ ({\bf Case~3}), then~$f$ is delimited by a~$3$-cycle~$xyz$ such that~$z$ is adjacent to both~$u$ and~$v$, or by a~$4$-cycle~$xytz$, such that each of~$t$ and~$z$ is adjacent to both~$u$ and~$v$. Indeed, if~$f$ was delimited by a cycle with more than~$4$ vertices, at least one of them would have degree~$3$ in~$G$; likewise, if one of the vertices on the boundary of~$f$ was not adjacent to~$u$ or~$v$, then it would have degree~$3$ in~$G$.  

Having ruled out that~$K$ is a vertex or an edge, it might be that~$K$ is a non-trivial biconnected graph ({\bf Case~4}). Unfortunately, we cannot handle this case via the algorithm by Borradaile et al.\ \cite{DBLP:journals/gc/BorradaileLS17}. Indeed, the natural choice for the graph~$H$ on which induction should be applied is the following: Remove~$K$ and the vertices in the cycle~$\mathcal C_f$ delimiting~$f$, except for~$x$ and~$y$, from~$G$, and let the resulting graph be~$H$. However, with this approach, the induction would tell us whether~$x$ and~$y$ belong to the set of vertices inducing an outerplane graph, and the algorithm by Borradaile et al.\ does not guarantee this freedom of choice when applied to the~$2$-outerplane graph induced by the vertices inside or on the boundary of~$\mathcal C_f$. Hence, we need to come up with a new algorithm providing guarantees on~$x$ and~$y$.

In the remaining cases,~$K$ is not biconnected. Then recall that~$T_K$ denotes the rooted block-cutvertex tree of~$K$. Let~$B$ be an extremal leaf of~$K$. Suppose first that~$B$ is non-trivial. If~$B$ is cushy ({\bf Case~5}), then we can apply a natural strategy for defining~$H$: We remove from~$G$ all the cage graph~$G_B$ of~$B$, except for the link-vertex~$c_B$ and, possibly, for the left and/or right cage vertices~$\ell_B$ and~$r_B$. If, on the other hand,~$B$ is pesky, then this simple strategy does not work (which justifies the attribute {\em pesky}) and we need to explore the structure of~$G$ in further depth. In particular, the next cases handle the situations in which: The edge~$c_B \ell'_B$ exists and the~$3$-cycle~$c_B \ell_B \ell'_B$ bounds an internal face of~$G$ ({\bf Case~6}); the edge~$c_B r'_B$ exists and the~$3$-cycle~$c_B r_B r'_B$ bounds an internal face of~$G$ ({\bf Case~6'}); the edge~$c_B \ell'_B$ exists and the~$3$-cycle~$c_B \ell_B \ell'_B$ contains in its interior a unique vertex ({\bf Case~7}); and the edge~$c_B r'_B$ exists and the~$3$-cycle~$c_B r_B r'_B$ contains in its interior a unique vertex ({\bf Case~7'}). Having ruled out that there is a face (Cases 6 and 6') or a~$K_4$ (Cases 7 and 7') next to~$G_B$, we next handle the situation in which there is a different pesky block~$B^*$ next to~$G_B$. Formally, this happens if there exists a pesky non-trivial extremal leaf~$B^*\neq B$ in~$K$ such that the link-vertex~$c_{B^*}$ of~$B^*$ is~$c_B$ and the right cage vertex~$r_{B^*}$ of~$B^*$ is~$\ell_B$ ({\bf Case~8}) or if there exists a pesky non-trivial extremal leaf~$B^*\neq B$ in~$K$ such that the link-vertex~$c_{B^*}$ of~$B^*$ is~$c_B$ and the left cage vertex~$\ell_{B^*}$ of~$B^*$ is~$r_B$ ({\bf Case~8'}). 

The next cases assume that~$B$ is trivial. Indeed, similarly to the case in which~$B$ is non-trivial, we can successfully apply induction if: The edge~$c_B \ell'_B$ exists and the~$3$-cycle~$c_B \ell_B \ell'_B$ bounds an internal face of~$G$ ({\bf Case~9}); the edge~$c_B r'_B$ exists and the~$3$-cycle~$c_B r_B r'_B$ bounds an internal face of~$G$ ({\bf Case~9'}); the edge~$c_B \ell'_B$ exists and the~$3$-cycle~$c_B \ell_B \ell'_B$ contains in its interior a unique vertex ({\bf Case~10}); and the edge~$c_B r'_B$ exists and the~$3$-cycle~$c_B r_B r'_B$ contains in its interior a unique vertex ({\bf Case~10'}). 

The described case distinction greatly simplifies the structure of~$G$ that can be assumed in the proximity of an extremal leaf~$B$ of the rooted block-cutvertex tree~$T_K$ of a terminal component~$K$ of~$G[L_2]$, as stated in the following.

\begin{lemma} \label{le:case-distinction-blocks}
If neither of Cases~5, 6, 6', 7, 7', 8, 8', 9, 9' 10, and 10' applies, then~$B$ is the only child of the cutvertex~$c_B$ in~$T_K$. 
\end{lemma} 

\begin{proof}
Suppose, for a contradiction, that~$c_B$ has a child~$B' \neq B$ in~$T_K$. Since~$B$ is an extremal leaf of~$T_K$, then so is~$B'$. The edges incident to~$c_B$ that belong to the cage graph~$G_B$ of~$B$ appear consecutively in the circular order around~$c_B$, between the edges~$c_B\ell_B$ and~$c_B r_B$, and likewise for the edges incident to~$c_B$ that belong to the cage graph~$G_{B'}$ of~$B'$. We define a linear order~$\sigma$ of the edges that are incident to~$c_B$ and that belong to blocks of~$K$ that are children of~$c_B$ in~$T_K$, as follows. Start from any edge incident to~$c_B$ in the parent block of~$c_B$ in~$T_K$ (or start from the edge~$c_Bx$ if~$c_B$ is the root of~$T_K$). Then visit in clockwise order the edges incident to~$c_B$ in~$G$ and append to~$\sigma$ those that belong to blocks of~$K$ that are children of~$c_B$ in~$T_K$. We can assume that, in~$\sigma$, the edges that belong to~$B$ come right after the edges that belong to~$B'$, as otherwise a different choice for~$B$ and~$B'$ can be made, so to guarantee this property. 
If~$B$ or~$B'$ is non-trivial and cushy, then Case~5 applies, a contradiction. Otherwise, consider the face incident to the edge~$c_B\ell_B$ and outside the cage graph~$G_B$; let~$z$ be the third vertex incident to the face. If~$z \in L_1$, then~$z=\ell'_B$, as otherwise~$K$ would not be a terminal component of~$G[L_2]$, hence we would be in Case~6 or Case~9, depending on whether~$B$ is non-trivial or trivial, respectively, a contradiction. Hence, we have~$z\in L_2$, which implies that~$z\in V(B')$. If~$B'$ is trivial, then Case~7 or Case~10 applies, depending on whether~$B$ is non-trivial or trivial, respectively, a contradiction. If~$B'$ is non-trivial and pesky, and~$B$ is trivial, then Case~7' applies, a contradiction. Finally, if both~$B$ and~$B'$ are non-trivial and pesky, then Case~8 applies, a contradiction.
Since in every case we get a contradiction, the lemma follows.
\end{proof}

By~\cref{le:case-distinction-blocks} and since~$K$ is not biconnected,~$c_B$ has a parent B-node~$B_P$ in~$T_K$. We can handle directly the case in which~$B_P$ is a trivial block, both if~$B$ is non-trivial ({\bf Case~11}) and if it is trivial ({\bf Case~12}). In order to discuss Case~11 and Case~12, it will be useful to exploit the following property.

\begin{lemma} \label{le:case11}
Suppose that~$B_P$ is a trivial block, corresponding to an edge~$c_Bp$. Then the edges~$p \ell_B$ and~$p r_B$ belong to~$G$, and the cycles~$p \ell_B c_B$ and~$p r_B c_B$ delimit faces of~$G$. 
\end{lemma}

\begin{proof}
Consider the face incident to the edge~$c_B\ell_B$ and outside the cage graph~$G_B$; let~$z$ be the third vertex incident to the face. We prove that~$z \in L_2$. Suppose, for a contradiction, that~$z \in L_1$.  If~$z=\ell'_B$, then Case~6 or Case~9 applies, depending on whether~$B$ is non-trivial or trivial, respectively, a contradiction. If~$z\neq \ell'_B$, then~$\ell_b\ell'_B$ is an internal edge of~$G[L_1]$, hence it is dual to an edge~$e$ of~$T^*$; one of the end-vertices of~$e$ is the node~$l$ of~$T^*$ corresponding to the face~$f$ of~$G[L_1]$ containing~$K$. If~$l$ is the parent of the other end-vertex of~$e$ in~$T^*$, then~$K$ is not a terminal component of~$G[L_2]$, a contradiction. If~$l$ is a child of the other end-vertex of~$e$ in~$T^*$, then~$\ell_b\ell'_B$ is the edge~$xy$. However, this implies that~$T_K$ is rooted at~$c_B$, a contradiction to the fact that~$c_B$ has a parent~$B_P$ in~$T_K$. This proves that~$z \in L_2$. It follows that~$z=p$, as otherwise the edge~$c_Bz$ would be an edge between two vertices of different connected components of~$G[L_2]$ or an edge between two different blocks of~$K$, which is a contradiction in both cases. It can be symmetrically proved that the edge~$p r_B$ belongs to~$G$ and the cycle~$p r_B c_B$ delimits an internal face of~$G$.
\end{proof}

We are now ready to discuss the final inductive case ({\bf Case~13}), which occurs when neither of Cases 1-12 applies in the entire graph~$G$ (in particular, Cases 5-12 do not apply to any extremal leaf in any terminal component of~$G[L_2]$). The next lemma describes the structure that we are guaranteed to encounter in Case~13. 

\begin{lemma} \label{le:case13}
If neither of Cases 1--12 applies, then there exists a terminal component~$K$ of~$G[L_2]$ that contains a biconnected component~$B_P$ that satisfies the following properties.
\begin{itemize}
\item First,~$B_P$ is non-trivial and is not a leaf of~$T_K$. 
\item Second, each child~$c_B$ of~$B_P$ in~$T_K$ has a unique child~$B$, which is an extremal leaf of~$T_K$ and is either a trivial block or a non-trivial pesky block.
\item Third, for each child~$B$ of a child~$c_B$ of~$B_P$, let~$c^l_B$ and~$c^r_B$ be the vertices that precede and follow~$c_B$ in clockwise order along the boundary of~$B_P$; then the edges~$c^l_B \ell_B$ and~$c^r_B r_B$ belong to~$G$, and the cycles~$c^l_B \ell_B c_B$ and~$c^r_B r_B c_B$ delimit internal faces of~$G$. 
\end{itemize}
\end{lemma}

\begin{proof}
Since Case~1 does not apply, we have that~$G[L_2]$ contains a terminal component. Let then~$K$ be any terminal component of~$G[L_2]$. Since Cases~2,~3, and~4 do not apply to~$K$, it follows that~$K$ is not biconnected. Let then~$B$ be an extremal leaf of~$T_K$. Since Case~5 does not apply,~$B$ is either a trivial block or a non-trivial pesky block. Since Cases~5, 6, 6', 7, 7', 8, 8', 9, 9', 10, and 10' do not apply, by~\cref{le:case-distinction-blocks}, we have that~$B$ is the only child of its parent~$c_B$ in~$T_K$. We can now define~$B_P$ as the parent of~$c_B$ in~$T_K$. By construction,~$B_P$ is not a leaf of~$T_K$. Also, since Cases~11 and~12 do not apply,~$B_P$ is non-trivial. Since~$B$ is an extremal leaf, every child of a child of~$B_P$ is also an extremal leaf, and then by the same arguments as for~$B$, it is either a trivial block or a non-trivial pesky block and it is the only child of its parent in~$T_K$. It remains to prove the third item. This proof is very similar to the proof of~\cref{le:case11}. Consider any child~$B$ of a child~$c_B$ of~$B_P$, and let~$c^l_B$ be the vertex that precedes~$c_B$ in clockwise order along the boundary of~$B_P$. Consider the face incident to the edge~$c_B\ell_B$ and outside the cage graph~$G_B$; let~$z$ be the third vertex incident to the face. We prove that~$z \in L_2$. Suppose, for a contradiction, that~$z \in L_1$.  If~$z=\ell'_B$, then Case~6 or Case~9 applies, depending on whether~$B$ is non-trivial or trivial, respectively, a contradiction.  If~$z\neq \ell'_B$, then~$\ell_b\ell'_B$ is an internal edge of~$G[L_1]$, hence it is dual to an edge~$e$ of~$T^*$; one of the end-vertices of~$e$ is the node~$l$ of~$T^*$ corresponding to the face~$f$ of~$G[L_1]$ containing~$K$. If~$l$ is the parent of the other end-vertex of~$e$ in~$T^*$, then~$K$ is not a terminal component of~$G[L_2]$, a contradiction. If~$l$ is a child of the other end-vertex of~$e$ in~$T^*$, then~$\ell_b\ell'_B$ is the edge~$xy$. However, this implies that~$T_K$ is rooted at~$c_B$, a contradiction to the fact that~$c_B$ has a parent~$B_P$ in~$T_K$. This proves that~$z \in L_2$. It follows that~$z=c^l_B$, as otherwise the edge~$c_Bz$ would be an edge between two vertices of different connected components of~$G[L_2]$, or an edge between two different blocks of~$K$, or an edge between two vertices of the same block~$B$ of~$K$ not belonging to~$B$, which is a contradiction in every case. It can be symmetrically proved that the edge~$c^r_B r_B$ belongs to~$G$ and the cycle~$c^r_B r_B c_B$ delimits an internal face of~$G$, where~$c^r_B$ is the vertex that follows~$c_B$ in clockwise order along the boundary of~$B_P$.
\end{proof}

Finally, it remains to observe that one of the inductive cases always applies; indeed, if neither of Cases 1--12 applies, then Case~13 does. Note that the graph~$H$ constructed in order to apply induction might be not biconnected, however by~\cref{le:augmentation} it can be augmented to a non-trivial biconnected graph by only adding edges to it, as long as it has at least three vertices (while it encounters the base case if it has at most two vertices). 


\subsection{Inductive cases} \label{se:induction}

We now discuss the inductive cases. Each inductive case assumes that neither of the previous cases applies. 

\paragraph*{Case~1: There exists a vertex~$v\in L_1$ such that~$\delta_G(v)= 2$} Let~$u$ and~$z$ be the neighbors of~$v$ in~$G$, and observe that they belong to~$L_1$ (see \cref{fig:algorithm-case1}). We construct an~$(n-1)$-vertex~$2$-outerplane graph~$H$ from~$G$ by removing~$v$ and its incident edges. We apply induction on~$H$, so to find a good set~$\mathcal I_H$ for~$H$. We define~$\mathcal I=\mathcal I_H\cup \{v\}$. We prove that~$\mathcal I$ is a good set. First, we have~$|\mathcal I|=|\mathcal I_H|+1\geq 2(n-1)/3+1>2n/3$. Also,~$G[\mathcal I]$ is outerplane by~\cref{obs:union}, since~$G[\mathcal I_H]$ and~$G[v \cup (\{u,z\}\cap \mathcal I_H)]$ are outerplane, are each in the outer face of the other one, and share either no vertex, or a single vertex, or a single edge.     

\begin{figure}[tb]
\centering
    \begin{subfigure}{0.48\textwidth}
		\centering
		\includegraphics[scale=1.2, page=1]{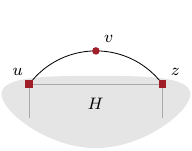}
		\subcaption{}
        \label{fig:algorithm-case1}
	\end{subfigure}
    \hfill
    \begin{subfigure}{0.48\textwidth}
		\centering
		\includegraphics[scale=1.2, page=2]{algorithm.pdf}
		\subcaption{}
        \label{fig:algorithm-case2.1}
	\end{subfigure}
    
    \bigskip
    
    \begin{subfigure}{0.48\textwidth}
		\centering
		\includegraphics[scale=1.2, page=3]{algorithm.pdf}
		\subcaption{}
        \label{fig:algorithm-case2.2}
	\end{subfigure}
    \hfill
    \begin{subfigure}{0.48\textwidth}
		\centering
		\includegraphics[scale=1.2, page=4]{algorithm.pdf}
		\subcaption{}
        \label{fig:algorithm-case2.3}
	\end{subfigure}

\caption{Illustrations for (a) Case~1, (b) Case~2.1, (c) Case~2.2, and (d) Case~2.3. In this and the following figures, squared vertices belong to the graph~$H$ (shown with gray edges and a shaded area which hides most of its vertices and edges) on which induction is applied, while rounded vertices do not belong to~$H$. A cross on a vertex not belonging to~$H$ indicates that it is not selected to be in the set~$\mathcal I$ inducing an outerplane graph. In (b), the dotted curve represents an edge added to~$H$ for the induction and the vertex~$v'$ might belong to~$L_1$ or~$L_2$.}
\label{fig:algorithm-cases1-2}
\end{figure}


\paragraph*{Case~2: There exists a vertex~$v\in L_1$ such that~$\delta_G(v)= 3$} Let~$u$ and~$z$ be the neighbors of~$v$ along the outer face of~$G$; hence,~$u$ and~$z$ belong to~$L_1$. The third neighbor,~$v'$, of~$v$ might belong to~$L_1$ or~$L_2$.  We further distinguish three cases.

{\bf Case~2.1:} The edge~$uz$ does not belong to~$G$ (see \cref{fig:algorithm-case2.1}). We construct an~$(n-1)$-vertex~$2$-outerplane graph~$H$ from~$G$ by removing~$v$ and its incident edges, and by adding the edge~$uz$ in the outer face of~$H$ in place of the path~$uvz$. We apply induction on~$H$, so to find a good set~$\mathcal I_H$ for~$H$. We define~$\mathcal I=\mathcal I_H\cup \{v\}$. We prove that~$\mathcal I$ is a good set. First, we have~$|\mathcal I|=|\mathcal I_H|+1\geq 2(n-1)/3+1>2n/3$. Suppose, for a contradiction, that~$G[\mathcal I]$ contains a cycle~$\mathcal C$ which contains a vertex~$w\in \mathcal I$ in its interior. Note that~$w\neq v$, since~$v$ is incident to the outer face of~$G$, hence~$w\in \mathcal I_H$. Also,~$v$ is one of the vertices of~$\mathcal C$, as otherwise~$G[\mathcal I_H]$ would also contain~$\mathcal C$, and~$\mathcal I_H$ would not be a good set. The neighbors of~$v$ in~$\mathcal C$ are two vertices among~$u$,~$z$, and~$v'$. Note that~$H$ contains all three edges~$uz$,~$uv'$, and~$v'z$, the first one by construction and the other two since~$G$ is internally-triangulated. Hence, the length-$2$ path in~$\mathcal C$ that has~$v$ as middle vertex can be replaced by the edge between its end-vertices, resulting in a cycle which also contains~$w$ in its interior and whose vertices all belong to~$\mathcal I_H$, which implies that~$\mathcal I_H$ is not a good set, a contradiction.

{\bf Case~2.2:} The edge~$uz$ belongs to~$G$ and the subgraph~$G_{uvz}$ of~$G$ induced by the vertices inside or on the boundary of the~$3$-cycle~$uvz$ has at least~$6$ vertices (see \cref{fig:algorithm-case2.2}). Note that the assumption that~$uz$ belongs to~$G$ implies that~$v'\in L_2$. Let~$H$ be the subgraph of~$G$ induced by~$V(G)-V(G_{uvz})$.  We apply induction on~$H$, so to find a good set~$\mathcal I_H$ for~$H$. We define~$\mathcal I$ as~$\mathcal I_H$ plus all the vertices of~$G_{uvz}$ different from~$u$ and~$z$. We prove that~$\mathcal I$ is a good set. First, we have~$|\mathcal I|=|\mathcal I_H|+|V(G_{uvz})|-2$. Note that~$|V(H)|=n-|V(G_{uvz})|$, hence, by induction, we have~$|\mathcal I_H|\geq 2(n-|V(G_{uvz})|)/3$. It follows that~$|\mathcal I|\geq 2(n-|V(G_{uvz})|)/3+|V(G_{uvz})|-2=2n/3 +|V(G_{uvz})|/3-2\geq 2n/3$, where the last inequality comes from the fact that~$|V(G_{uvz})|\geq 6$. By~\cref{obs:pesky}, we have that~$G[\mathcal I]$ is an outerplane graph, since it is composed of the outerplane graphs~$G[\mathcal I_H]$ and~$G_{uvz}-\{u,z\}$, which are each in the outer face of the other one and do not share any vertex. Note that~$G_{uvz}-\{u,z\}$ is indeed an outerplane graph since it is composed of the outerplane graph~$G[V(G_{uvz})\cap L_2]$ augmented with the edge~$vv'$ in its outer face.  

{\bf Case~2.3:} The edge~$uz$ belongs to~$G$ and~$G_{uvz}$ has at most~$5$ vertices (see \cref{fig:algorithm-case2.3}). Let~$H$ be the subgraph of~$G$ induced by~$\{u\} \cup (V(G)-V(G_{uvz}))$.  We apply induction on~$H$, so to find a good set~$\mathcal I_H$ for~$H$. We define~$\mathcal I$ as~$\mathcal I_H$ plus all the vertices of~$G_{uvz}$ different from~$u$ and~$z$. Note that~$u$ is in~$\mathcal I$ if and only if it belongs to~$\mathcal I_H$, while~$z$ is not in~$\mathcal I$. We prove that~$\mathcal I$ is a good set. First, we have~$|\mathcal I|=|\mathcal I_H|+|V(G_{uvz})|-2$. Note that~$|V(H)|=n-|V(G_{uvz})|+1$, hence, by induction, we have~$|\mathcal I_H|\geq 2(n-|V(G_{uvz})|+1)/3$. It follows that~$|\mathcal I|\geq 2(n-|V(G_{uvz})|+1)/3+|V(G_{uvz})|-2=2n/3 +|V(G_{uvz})|/3-4/3\geq 2n/3$, where the last inequality comes from the fact that~$|V(G_{uvz})|\geq 4$, given that~$v$ has degree~$3$. Note that~$G[\mathcal I_H]$ is outerplane by induction and that the subgraph of~$G_{uvz}$ induced by the vertices in~$\mathcal I$ is either a path in the outer face of~$G[\mathcal I_H]$, if~$u\notin \mathcal I$, or a biconnected outerplane graph whose vertex set intersects with the one of~$G[\mathcal I_H]$ only at~$u$, if~$u \in \mathcal I$. Hence, by~\cref{obs:union}, we have that~$G[\mathcal I]$ is an outerplane graph. 


In the remaining cases,~$l$ denotes a leaf of the rooted weak dual tree~$T^*$ of~$G[L_1]$, where~$l$ corresponds to a face~$f$ of~$G[L_1]$ and contains a terminal component~$K$ of~$G[L_2]$. Let~$xy$ be the edge of~$G[L_1]$ dual to the edge of~$T^*$ from the node corresponding to~$f$ to its parent (or let~$xy$ be the edge~$e^*$ if the node corresponding to~$f$ is the root of~$T^*$). 

\begin{figure}[tb]
\centering
    \begin{subfigure}{0.48\textwidth}
		\centering
		\includegraphics[scale=1.2, page=5]{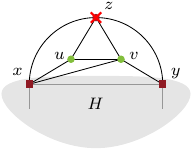}
		\subcaption{}
        \label{fig:algorithm-case3.1}
	\end{subfigure}
    \hfill
    \begin{subfigure}{0.48\textwidth}
		\centering
		\includegraphics[scale=1.2, page=6]{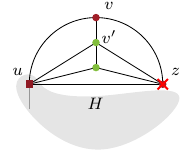}
		\subcaption{}
        \label{fig:algorithm-case3.2}
	\end{subfigure}
\caption{Illustrations for (a) Case~3.1 and (b) Case~3.2.}
\label{fig:algorithm-cases3}
\end{figure}

\paragraph*{Case~3:~$K$ is a single edge~$uv$} Recall that~$f$ is delimited by a~$3$-cycle~$xyz$ such that~$z$ is adjacent to both~$u$ and~$v$, or by a~$4$-cycle~$xytz$, such that each of~$t$ and~$z$ is adjacent to both~$u$ and~$v$. We distinguish~these~cases.

{\bf Case~3.1:} The face~$f$ of~$G[L_1]$ is delimited by a~$3$-cycle~$xyz$ such that~$z$ is adjacent to both~$u$ and~$v$ (see \cref{fig:algorithm-case3.1}). Let~$H$ be the subgraph of~$G$ induced by~$V(G)-\{u,v,z\}$.  We apply induction on~$H$, so to find a good set~$\mathcal I_H$ for~$H$. We define~$\mathcal I=\mathcal I_H\cup \{u,v\}$. We prove that~$\mathcal I$ is a good set. First, we have~$|\mathcal I|=|\mathcal I_H|+2$. Note that~$|V(H)|=n-3$, hence, by induction, we have~$|\mathcal I_H|\geq 2(n-3)/3$. It follows that~$|\mathcal I|\geq 2(n-3)/3+2=2n/3$. Also,~$G[\mathcal I_H]$ is outerplane by induction and the subgraph of~$G$ induced by~$u$,~$v$, and by~$\{x,y\}\cap \mathcal I_H$ is outerplane, as well, given that~$z$ is a neighbor of all of~$u,v,x,y$. Since~$G[\mathcal I_H]$ and~$G[\{u,v\} \cup (\{x,y\}\cap \mathcal I_H)]$ are each in the outer face of the other one, and share either no vertex, or a single vertex, or a single edge, by~\cref{obs:union}, we have that~$G[\mathcal I]$ is an outerplane graph. 

{\bf Case~3.2:} The face~$f$ of~$G[L_1]$ is delimited by a~$4$-cycle~$xytz$, such that each of~$t$ and~$z$ is adjacent to both~$u$ and~$v$ (see \cref{fig:algorithm-case3.2}). Let~$H$ be the subgraph of~$G$ induced by~$V(G)-\{u,v,t,z\}$. We apply induction on~$H$, so to find a good set~$\mathcal I_H$ for~$H$. We define~$\mathcal I=\mathcal I_H\cup \{u,v,t\}$. We prove that~$\mathcal I$ is a good set. First, we have~$|\mathcal I|=|\mathcal I_H|+3$. Note that~$|V(H)|=n-4$, hence, by induction, we have~$|\mathcal I_H|\geq 2(n-4)/3$. It follows that~$|\mathcal I|\geq 2(n-4)/3+3>2n/3$. Also,~$G[\mathcal I_H]$ is outerplane by induction; also, the subgraph of~$G$ induced by~$u$,~$v$,~$t$, and by~$\{x,y\}\cap \mathcal I_H$ is outerplane, as well, given that~$z$ is adjacent to all of~$u,v,t$, and~$x$, and given that~$y$ belong to~$L_1$. Since~$G[\mathcal I_H]$ and~$G[\{u,v,t\} \cup (\{x,y\}\cap \mathcal I_H)]$ are each in the outer face of the other one, and share either no vertex, or a single vertex, or a single edge, by~\cref{obs:union}, we have that~$G[\mathcal I]$ is an outerplane graph. 

\begin{figure}[htb]
\centering
\includegraphics[scale=1.2, page=1]{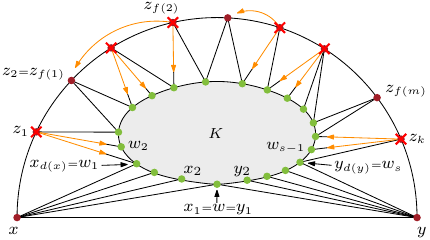}
\caption{Illustration for Case~4. The result of applying the basic algorithm to~$L$ to compute~$\mathcal I_L$. In this and the following figures, orange arrows indicate the charging of vertices that are not in~$\mathcal I_L$.}
\label{fig:algorithm-case4}
\end{figure}

\paragraph*{Case~4:~$K$ is a non-trivial biconnected graph}

Recall that~$\mathcal C_f$ denotes the cycle delimiting the boundary of the face~$f$ of~$G[L_1]$ containing~$K$. Since~$K$ is a terminal component, all the edges of~$\mathcal C_f$, except possibly for~$xy$, are incident to the outer face of~$G$. Denote by~$L$ the graph~$G[V(\mathcal C_f)\cup V(K)]$ and let~$n_L=|V(L)|$. Before describing how to apply induction, we need to prove the following lemma. 

\begin{lemma} \label{le:biconnectedK}
There exists a set~$\mathcal I_L\subset V(L)$ such that~$\{x,y\}\cap \mathcal I_L=\emptyset$ and such that (at least) one of the following properties is satisfied:
\begin{itemize}
    \item {\em Property A:}~$|\mathcal I_L|\geq 2n_L/3$ and~$G[\mathcal I_L]$ is an outerplane graph;
    \item {\em Property B:}~$|\mathcal I_L|\geq 2(n_L-1)/3$ and~$G[\mathcal I_L\cup \{x\}]$ is an outerplane graph;
    \item {\em Property C:}~$|\mathcal I_L|\geq 2(n_L-1)/3$ and~$G[\mathcal I_L\cup \{y\}]$ is an outerplane graph; and
    \item {\em Property D:}~$|\mathcal I_L|\geq 2(n_L-2)/3$ and~$G[\mathcal I_L\cup \{x,y\}]$ is an outerplane graph.
\end{itemize}
\end{lemma}

\begin{proof}
Refer to \cref{fig:algorithm-case4}.
Suppose, w.l.o.g., that~$x$ comes right after~$y$ in clockwise order along the boundary of the outer face of~$L$. Then let~$y,x,z_1,\dots,z_k$ be the clockwise order of the vertices along the outer face of~$L$. The \emph{internal degree}~$d(z)$ of a vertex~$z$ on the boundary of~$L$ is the number of its neighbors in~$K$. Note that, since Cases 1 and 2 do not apply, the internal degree of every vertex among~$z_1,\dots,z_k$ is at least two. Let also~$z_{f(1)},z_{f(2)},\dots,z_{f(m)}$ be the sequence~$z_1,\dots,z_k$ restricted to the~$m\geq 0$ vertices with internal degree~$2$. Let~$xyw$ be the internal face of~$L$ incident to the edge~$xy$, and let~$x_1=y_1=w,x_2,\dots,x_{d(x)}=w_1,w_2,\dots,w_s=y_{d(y)},y_{d(y)-1},\dots,y_2$ be the clockwise order of the vertices along the outer face of~$K$, where~$x_1,x_2,\dots,x_{d(x)}$ are the neighbors of~$x$ and~$y_{d(y)},y_{d(y)-1},\dots,y_1$ are the neighbors of~$y$. 

We start by defining a \emph{basic algorithm} for the construction of the required set~$\mathcal I_L$. The algorithm is as follows. First, all the vertices of~$K$ are in~$\mathcal I_L$. Second,~$x$ and~$y$ are not in~$\mathcal I_L$. Third, each vertex~$z_j$ with~$d(z_j)\geq 3$ is not in~$\mathcal I_L$. Finally, a vertex~$z_{f(j)}$ with internal degree~$2$ is in~$\mathcal I_L$ if~$j$ is odd and not in~$\mathcal I_L$ if~$j$ is even. Note that~$\{x,y\}\cap \mathcal I_L=\emptyset$, as required.

Our actual algorithm for the definition of~$\mathcal I_L$ starts from the set constructed by the basic algorithm and then it might remove~$z_{f(m)}$ from it. We distinguish several cases, based on the values of~$d(x)$ and~$d(y)$, however it is convenient to first factor out some parts of the proof that are common to the distinct cases.

First, in order to prove that~$\mathcal I_L$ satisfies Property~A,~B,~C, or~D, we need to show that~$G[\mathcal I_L]$, or~$G[\mathcal I_L\cup \{x\}]$, or~$G[\mathcal I_L\cup \{y\}]$, or~$G[\mathcal I_L\cup \{x,z\}]$ is outerplane, respectively. For that, it suffices to show that, for each vertex in~$K$, there exists a neighbor among~$y,x,z_1,\dots,z_k$ which is not in~$\mathcal I_L$, or not in~$\mathcal I_L\cup \{x\}$, or not in~$\mathcal I_L\cup \{y\}$, or not in~$\mathcal I_L\cup \{x,y\}$, respectively. This proof strategy is reminiscent of~\cite{DBLP:journals/gc/BorradaileLS17}. Consider a vertex~$w_i$, for some~$2\leq i \leq s-1$. If~$w_i$ is a neighbor of a vertex~$z_j$ with~$d(z_j)\geq 3$, then by the construction of the basic algorithm~$z_j$ is not in~$\mathcal I_L$, hence it is the desired neighbor of~$w_i$. Each remaining vertex~$w_i$ with~$2\leq i \leq s-1$ is only neighbor of vertices~$z_j$ with~$d(z_j)=2$. Since~$L$ is internally-triangulated, it follows that~$w_i$ is a neighbor of two vertices~$z_{f(j)}$ and~$z_{f(j+1)}$ with~$d(z_{f(j)})=d(z_{f(j+1)})=2$ that are consecutive along the boundary of~$L$. By the construction of the basic algorithm, one of them is not in~$\mathcal I_L$, hence it is the desired neighbor of~$w_i$. The proof that the vertices~$x_i$ and~$y_i$ have a neighbor among~$y,x,z_1,\dots,z_k$ which is  not in~$\mathcal I_L$, or not in~$\mathcal I_L\cup \{x\}$, or not in~$\mathcal I_L\cup \{y\}$, or not in~$\mathcal I_L\cup \{x,y\}$ will be slightly different on a case-by-case basis. 

Second, in order to prove that~$\mathcal I_L$ satisfies Property~A,~B,~C, or~D, we need to show that~$|\mathcal I_L|\geq 2n_L/3$, or~$|\mathcal I_L|\geq 2(n_L-1)/3$, or~$|\mathcal I_L|\geq 2(n_L-1)/3$, or~$|\mathcal I_L|\geq 2(n_L-2)/3$, respectively. This is done by {\em charging} each vertex of~$L$ not in~$\mathcal I_L$, or each vertex of~$L$ not in~$\mathcal I_L$ and different from~$x$, or each vertex of~$L$ not in~$\mathcal I_L$ and different from~$y$, or each vertex of~$L$ not in~$\mathcal I_L$ and different from~$x$ and~$y$, respectively, to two vertices in~$\mathcal I_L$, so that each vertex in~$\mathcal I_L$ is charged with at most one vertex. We charge each vertex~$z_j$ with~$d(z_j)\geq 3$ to its first two neighbors in the sequence~$w_1,w_2,\dots,w_s$. We charge each vertex~$z_{f(j)}$ with~$d(z_{f(j)})=2$ and with~$j$ even to~$z_{f(j-1)}$ and to its first neighbor in the sequence~$w_1,w_2,\dots,w_s$; note that~$z_{f(j-1)}\in \mathcal I_L$. The charge of~$x$ and~$y$ will be slightly different on a case-by-case basis; also, whenever the algorithm removes~$z_{f(m)}$ from~$\mathcal I_L$, we need to charge that vertex to two vertices in~$\mathcal I_L$, as well. Note that each vertex~$z_{f(j)}\in \mathcal I_L$ is charged with at most one vertex, namely~$z_{f(j+1)}$, if such a vertex exists. Also, if a vertex~$w_i$ is charged as the second neighbor in the sequence~$w_1,w_2,\dots,w_s$ of a vertex~$z_j$ with~$d(z_j)\geq 3$, then it is not charged with any other vertex, since~$z_j$ is its only neighbor in~$z_1,\dots,z_k$. Furthermore, if a vertex~$w_i$ is charged as the first neighbor in the sequence~$w_1,w_2,\dots,w_s$ of a vertex~$z_j$, then it is not charged as the first neighbor in the sequence~$w_1,w_2,\dots,w_s$ of a different vertex~$z_{j'}$. Indeed, if~$w_i$ were charged as the first neighbor in the sequence~$w_1,w_2,\dots,w_s$ of two distinct vertices~$z_j$ and~$z_{j'}$ with~$j<j'$, it would be the only neighbor of~$z_j$ in~$K$, hence the degree of~$z_j$ in~$G$ would be~$3$, contradicting the fact that Case~2 does not apply. Hence, each vertex in the sequence~$w_1,w_2,\dots,w_s$ is charged with at most one vertex. Finally, observe that the vertices~$x_{d(x)-1},x_{d(x)-2},\dots,x_1=y_1,y_2,\dots,w_s=y_{d(y)}$ have not yet been charged with any vertex, while~$w_1=x_{d(x)}$ might have been charged with~$z_1$, if~$d(z_1)\geq 3$; also, for each~$j$ odd with~$1\leq j\leq m$, the first neighbor in the sequence~$w_1,w_2,\dots,w_s$ of~$z_{f(j)}$ has not yet been charged with any vertex.  
We now present our case distinction.

\begin{figure}[tb]
\centering
    \begin{subfigure}{0.48\textwidth}
		\centering
		\includegraphics[scale=1.2, page=2]{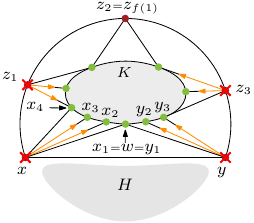}
		\subcaption{}
        \label{fig:algorithm-case4-degree-at-least-6}
	\end{subfigure}
    \hfill
    \begin{subfigure}{0.48\textwidth}
		\centering
		\includegraphics[scale=1.2, page=3]{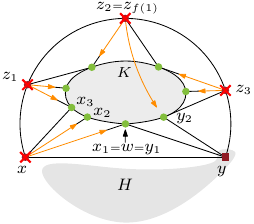}
		\subcaption{}
        \label{fig:algorithm-case4-degree-5}
	\end{subfigure}
\caption{Illustrations for Case~4 (a) if~$d(x)+d(y)\geq 6$ and (b) if~$d(x)+d(y)=5$.}
\label{fig:algorithm-case4-degree-at-least-5}
\end{figure}

\begin{itemize}
    \item {\em Suppose first that~$d(x)+d(y)\geq 6$} (see \cref{fig:algorithm-case4-degree-at-least-6}). This implies that~$x$ and~$y$ have a total of at least~$5$ distinct neighbors in~$K$. Indeed,~$w$ contributes to both~$d(x)$ and~$d(y)$; also,~$x$ and~$y$ do not have any common neighbor~$w'$ other than~$w$, as otherwise, by planarity,~$K$ would lie inside the cycle~$xwyw'$, thus~$z_1,\dots,z_k$ would not have any neighbor in~$K$ other than~$w'$, and hence they would have degree~$3$ in~$G$, contradicting the fact that Case~2 does not apply. We define a set~$\mathcal I_L$ by means of the basic algorithm and prove that it satisfies Property~A. In order to prove that~$G[\mathcal I_L]$ is outerplane, it suffices to observe that the neighbor~$x$ of the vertices~$x_i$ is not in~$\mathcal I_L$ and the neighbor~$y$ of the vertices~$y_i$ is not in~$\mathcal I_L$. In order to prove that~$|\mathcal I_L|\geq 2n_L/3$, we charge~$x$ and~$y$ to the first and to the last two vertices in the sequence~$x_{d(x)-1},x_{d(x)-2},\dots,x_1=y_1,y_2,\dots,y_{d(y)}$, respectively. Since such sequence contains at least~$4$ vertices, each of~$x_{d(x)-1},x_{d(x)-2}$,~$\dots,x_1=y_1,y_2,\dots,y_{d(y)}$ is charged with at most one of~$x$ and~$y$.

\item {\em Suppose next that~$d(x)+d(y)=5$} (see \cref{fig:algorithm-case4-degree-5}). Assume that~$d(x)>d(y)$, as the case~$d(x)< d(y)$ can be handled symmetrically by mirroring the embedding of~$G$. We define a set~$\mathcal I_L$ by means of the basic algorithm, however if~$m$ is odd, we remove~$z_{f(m)}$ from~$\mathcal I_L$. We prove that~$\mathcal I_L$ satisfies Property~C (if~$d(x)<d(y)$, rather than~$d(x)>d(y)$, we can define a set~$\mathcal I_L$ that satisfies Property~B). In order to prove that~$G[\mathcal I_L\cup \{y\}]$ is outerplane, we observe that the neighbor~$x$ of the vertices~$x_i$ is not in~$\mathcal I_L\cup \{y\}$. For the vertices~$y_i$, we cannot rely on~$y$, which is in~$\mathcal I_L\cup \{y\}$. However,~$y_1$ is also a neighbor of~$x$, which is not in~$\mathcal I_L\cup \{y\}$. Furthermore,~$y_2$, if it exists, is also a neighbor of~$z_k$, which is not in~$\mathcal I_L\cup \{y\}$; this is obvious if~$d(z_k)\geq 3$ and it comes from the choice of excluding from~$\mathcal I_L$ the vertex~$z_{f(m)}$ if~$d(z_k)=2$.  In order to prove that~$|\mathcal I_L|\geq 2(n_L-1)/3$, we charge~$x$ to the~$2$ vertices in the sequence~$x_{d(x)-1},x_{d(x)-2},\dots,x_1=y_1,y_2,\dots,y_{d(y)-1}$ (more precisely, this sequence is~$x_3,x_2$ if~$d(x)=4$ and~$d(y)=1$, and it is~$x_2,x_1=y_1$ if~$d(x)=3$ and~$d(y)=2$). Also, we charge~$z_{f(m)}$ to its first neighbor in the sequence~$w_1,w_2,\dots,w_s$ and to~$y_{d(y)}$. Note that both such vertices were previously uncharged.

\item {\em Suppose next that~$d(x)+d(y)=4$.} Assume that~$d(x)\geq d(y)$, as the case~$d(x)< d(y)$ can be handled symmetrically by mirroring the embedding of~$G$. We need to further brake down this case into three subcases.

\begin{itemize}
    \item {\em First, suppose that~$d(x)=3$} (see \cref{fig:algorithm-case4-degree-4-x-has-degree-3}). We define a set~$\mathcal I_L$ by means of the basic algorithm and prove that it satisfies Property~C (if~$d(y)=3$, rather than~$d(x)=3$, we can define a set~$\mathcal I_L$ that satisfies Property~B). In order to prove that~$G[\mathcal I_L\cup \{y\}]$ is outerplane, we observe that the neighbor~$x$ of the vertices~$x_1=y_1,x_2,x_3$ is not in~$\mathcal I_L\cup \{y\}$. In order to prove that~$|\mathcal I_L|\geq 2(n_L-1)/3$, we charge~$x$ to~$x_1$ and~$x_2$. Note that both such vertices were previously uncharged.
    
    \item {\em Second, suppose that~$d(x)=2$ and~$m$ is even} (see \cref{fig:algorithm-case4-degree-4-x-has-degree-2-m-is-even}). We define a set~$\mathcal I_L$ by means of the basic algorithm and prove that~it satisfies Property~C (we can similarly define a set~$\mathcal I_L$ that satisfies Property~B). In order to prove that~$G[\mathcal I_L\cup \{y\}]$ is outerplane, we observe that the neighbor~$x$ of the vertices~$x_1=y_1,x_2$ is not in~$\mathcal I_L\cup \{y\}$. Also, the neighbor~$z_k$ of~$y_2$ is not in~$\mathcal I_L\cup \{y\}$; this is true either because~$d(z_k)\geq 3$, or because~$z_k$ coincides with~$z_{f(m)}$ and~$m$ is even, by assumption. In order to prove that~$|\mathcal I_L|\geq 2(n_L-1)/3$, we charge~$x$ to~$y_1$ and~$y_2$. Note that both such vertices were previously uncharged. 
        
     \item {\em Third, suppose that~$d(x)=2$ and~$m$ is odd} (see \cref{fig:algorithm-case4-degree-4-x-has-degree-2-m-is-odd}). We define a set~$\mathcal I_L$ by means of the basic algorithm and prove that~it satisfies Property~A. In order to prove that~$G[\mathcal I_L]$ is outerplane, we observe that the neighbor~$x$ of the vertices~$x_1,x_2$ is not in~$\mathcal I_L$ and the neighbor~$y$ of the vertices~$y_1,y_2$ is not in~$\mathcal I_L$. In order to prove that~$|\mathcal I_L|\geq 2n_L/3$, we charge~$x$ to~$y_1$ and~$y_2$ and~$y$ to~$z_{f(m)}$ and to the first neighbor of~$z_{f(m)}$ in the sequence~$w_1,w_2,\dots,w_s$. Note that all four charged vertices were previously uncharged. 
\end{itemize}

\begin{figure}[tb]
\centering
    \begin{subfigure}{0.32\textwidth}
		\centering
		\includegraphics[scale=1, page=4]{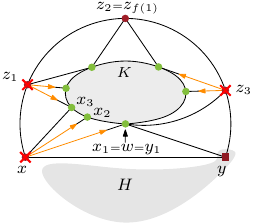}
		\subcaption{}
        \label{fig:algorithm-case4-degree-4-x-has-degree-3}
	\end{subfigure}
    \hfill
    \begin{subfigure}{0.32\textwidth}
		\centering
		\includegraphics[scale=1, page=5]{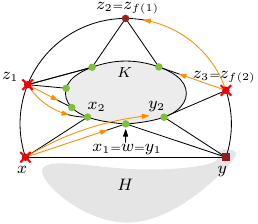}
		\subcaption{}
        \label{fig:algorithm-case4-degree-4-x-has-degree-2-m-is-even}
	\end{subfigure}
    \hfill
    \begin{subfigure}{0.32\textwidth}
		\centering
		\includegraphics[scale=1, page=6]{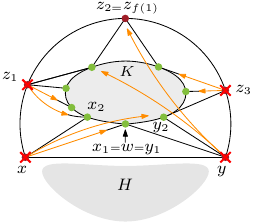}
		\subcaption{}
        \label{fig:algorithm-case4-degree-4-x-has-degree-2-m-is-odd}
	\end{subfigure}
\caption{Illustrations for Case~4 when~$d(x)+d(y)=4$ (a) if~$d(x) = 3$, (b) if~$d(x) = 2$ and~$m$ is even, and (c) if~$d(x) = 2$ and~$m$ is odd.}
\label{fig:algorithm-case4-degree-4}
\end{figure}

\item {\em Suppose next that~$d(x)+d(y)=3$.} Assume that~$d(x)>d(y)$, as the case~$d(x)< d(y)$ can be handled symmetrically by mirroring the embedding of~$G$. We need to further brake down this case into three subcases.

\begin{itemize}
    \item {\em First, suppose that~$d(z_1)=2$} (see \cref{fig:algorithm-case4-degree-3-z1-has-degree-2}). We define a set~$\mathcal I_L$ by means of the basic algorithm. We prove that~$\mathcal I_L$ satisfies Property~C (if~$d(x)<d(y)$ and~$d(z_k)=2$, rather than~$d(x)>d(y)$ and~$d(z_1)=2$, we can define a set~$\mathcal I_L$ that satisfies Property~B). In order to prove that~$G[\mathcal I_L\cup \{y\}]$ is outerplane, we observe that the neighbor~$x$ of the vertices~$x_1=y_1,x_2$ is not in~$\mathcal I_L$. In order to prove that~$|\mathcal I_L|\geq 2(n_L-1)/3$, we charge~$x$ to~$x_1=y_1=y_{d(y)}$ and to~$x_2=x_{d(x)}$. Note that, while the basic algorithm always guarantees that~$y_{d(y)}$ is not charged with any vertex~$z_j$, the fact that~$x_{d(x)}$ is not charged with any vertex~$z_j$ relies on the assumption~$d(z_1)=2$. 

    \item {\em Second, suppose that~$d(z_1)\geq 3$ and~$m$ is even} (see \cref{fig:algorithm-case4-degree-3-z1-has-degree-at-least-3-m-is-even}). We define a set~$\mathcal I_L$ by means of the basic algorithm. We prove that~$\mathcal I_L$ satisfies Property~D. In order to prove that~$G[\mathcal I_L\cup \{x,y\}]$ is outerplane, we observe that the neighbor~$z_1$ of the vertex~$x_2$ is not in~$\mathcal I_L$, given that~$d(z_1)\geq 3$. Also, the neighbor~$z_k$ of the vertex~$x_1=y_1$ is not in~$\mathcal I_L\cup \{x,y\}$; this is true either because~$d(z_k)\geq 3$, or because~$z_k$ coincides with~$z_{f(m)}$ and~$m$ is even, by assumption. Since in this case we do not need to charge~$x$ and~$y$, the bound~$|\mathcal I_L|\geq 2(n_L-2)/3$ directly follows.
    
    \item {\em Third, suppose that~$d(z_1)\geq 3$ and~$m$ is odd} (see \cref{fig:algorithm-case4-degree-3-z1-has-degree-at-least-3-m-is-odd}). We define a set~$\mathcal I_L$ by means of the basic algorithm. We prove that~$\mathcal I_L$ satisfies Property~B (if~$d(x)=1<d(y)$,~$d(z_k)\geq 3$, and~$m$ is odd, rather than~$d(x)>d(y)$,~$d(z_1)\geq 3$, and~$m$ is odd, we can define a set~$\mathcal I_L$ that satisfies Property~C). In order to prove that~$G[\mathcal I_L\cup \{x\}]$ is outerplane, we observe that the neighbor~$z_1$ of the vertex~$x_2$ is not in~$\mathcal I_L$, given that~$d(z_1)\geq 3$, and that the neighbor~$y$ of the vertex~$x_1=y_1$ is not in~$\mathcal I_L\cup \{x\}$. In order to prove that~$|\mathcal I_L|\geq 2(n_L-1)/3$, we charge~$y$ to~$x_1=y_1=y_{d(y)}$ and to~$z_{f(m)}$. Note that both such vertices were previously uncharged. 
\end{itemize}

\begin{figure}[tb]
\centering
    \begin{subfigure}{0.32\textwidth}
		\centering
		\includegraphics[scale=1, page=7]{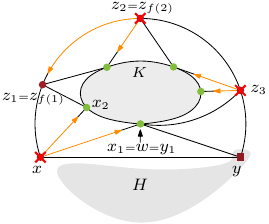}
		\subcaption{}
        \label{fig:algorithm-case4-degree-3-z1-has-degree-2}
	\end{subfigure}
    \hfill
    \begin{subfigure}{0.32\textwidth}
		\centering
		\includegraphics[scale=1, page=8]{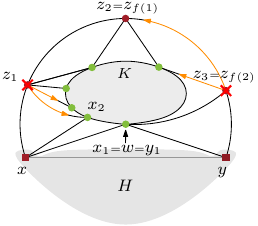}
		\subcaption{}
        \label{fig:algorithm-case4-degree-3-z1-has-degree-at-least-3-m-is-even}
	\end{subfigure}
    \hfill
    \begin{subfigure}{0.32\textwidth}
		\centering
		\includegraphics[scale=1, page=9]{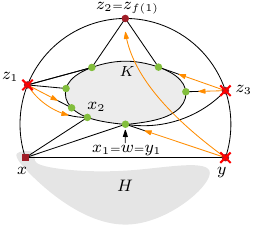}
		\subcaption{}
        \label{fig:algorithm-case4-degree-3-z1-has-degree-at-least-3-m-is-odd}
	\end{subfigure}
\caption{Illustrations for Case~4 when~$d(x)+d(y)=3$ (a) if~$d(z_1) = 2$, (b) if~$d(z_1) \geq 3$ and~$m$ is even, and (c) if~$d(z_1) \geq 3$ and~$m$ is odd.}
\label{fig:algorithm-case4-degree-3}
\end{figure}

    \item {\em Suppose finally that~$d(x)+d(y)=2$} (see \cref{fig:algorithm-case4-degree-2}). Hence,~$x$ and~$y$ are neighbors of~$w$ and of no other vertex of~$K$. We define a set~$\mathcal I_L$ by means of the basic algorithm, however if~$m$ is odd and greater than~$1$, we remove~$z_{f(m)}$ from~$\mathcal I_L$. We prove that~$\mathcal I_L$ satisfies Property~D. In order to prove that~$G[\mathcal I_L\cup \{x,y\}]$ is outerplane, we need to show that there exists a neighbor of~$x_1=y_1=w$ among~$y,x,z_1,\dots,z_k$ that is not in~$\mathcal I_L\cup \{x,y\}$. Clearly, such a neighbor cannot be~$x$ or~$y$, since these vertices are in~$\mathcal I_L\cup \{x,y\}$. Note that~$z_1$ and~$z_k$ are neighbors of~$w$ and in this case they are guaranteed to be distinct vertices, given that~$L$ is internally-triangulated, that~$d(x)=d(y)=1$, and that~$K$ is not a single vertex. If~$z_1$ or~$z_k$ has internal degree at least~$3$, then that vertex is the desired neighbor of~$w$ not in~$\mathcal I_L$. Otherwise,~$z_1$ and~$z_k$ have both internal degree~$2$, thus~$m\geq 2$, and then~$z_k=z_{f(m)}$ is not in~$\mathcal I_L$.  In order to prove that~$|\mathcal I_L|\geq 2(n_L-2)/3$, we do not need to charge~$x$ and~$y$; however, if~$m$ is odd and greater than~$1$, then we need to charge~$z_{f(m)}$. We charge~it to its first neighbor in the sequence~$w_1,w_2,\dots,w_s$ and to the first neighbor of~$z_{f(1)}$ in the sequence~$w_1,w_2,\dots,w_s$. Note that both such vertices were previously uncharged.
\end{itemize}
This completes the proof of the lemma.
\end{proof}

\begin{figure}[htb]
\centering
\includegraphics[scale=1.2, page=10]{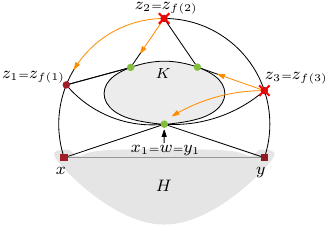}
\caption{Illustration for Case~4 when~$d(x)+d(y)=2$.}
\label{fig:algorithm-case4-degree-2}
\end{figure}

We can now apply induction as follows. Let~$\mathcal I_L$ be a set as in~\cref{le:biconnectedK}. If~$\mathcal I_L$ satisfies Property~A, let~$H$ be the subgraph of~$G$ induced by~$V(G)-V(L)$; otherwise, if it satisfies Property~B, let~$H$ be the subgraph of~$G$ induced by~$V(G)-(V(L)-\{x\})$; otherwise, if it satisfies Property~C, let~$H$ be the subgraph of~$G$ induced by~$V(G)-(V(L)-\{y\})$; otherwise, if it satisfies Property~D, let~$H$ be the subgraph of~$G$ induced by~$V(G)-(V(L)-\{x,y\})$. We apply induction on~$H$, so to find a good set~$\mathcal I_H$ for~$H$. We define~$\mathcal I=\mathcal I_H\cup \mathcal I_L$. 
We prove that~$\mathcal I$ is a good set. Note that~$|\mathcal I|=|\mathcal I_H|+|\mathcal I_L|$. If~$\mathcal I_L$ satisfies Property~A, then~$|V(H)|=n-n_L$, hence, by induction, we have~$|\mathcal I_H|\geq 2(n-n_L)/3$ and, by~\cref{le:biconnectedK}, we have~$|\mathcal I_L|\geq 2n_L/3$, thus~$|\mathcal I|\geq 2n/3$. If~$\mathcal I_L$ satisfies Property~B or Property~C, then~$|V(H)|=n-n_L+1$, hence, by induction, we have~$|\mathcal I_H|\geq 2(n-n_L+1)/3$ and, by~\cref{le:biconnectedK}, we have~$|\mathcal I_L|\geq 2(n_L-1)/3$, thus~$|\mathcal I|\geq 2n/3$. Finally, if~$\mathcal I_L$ satisfies Property~D, then~$|V(H)|=n-n_L+2$, hence, by induction, we have~$|\mathcal I_H|\geq 2(n-n_L+2)/3$ and, by~\cref{le:biconnectedK}, we have~$|\mathcal I_L|\geq 2(n_L-2)/3$, thus~$|\mathcal I|\geq 2n/3$.   
In order to prove that~$G[\mathcal I]$ is outerplane, note that~$G[\mathcal I_H]$ is outerplane by induction. If~$\mathcal I_L$ satisfies Property~A, then~$G[\mathcal I_L]$ is outerplane by~\cref{le:biconnectedK}. Since~$G[\mathcal I_H]$ and~$G[\mathcal I_L]$ are each in the outer face of the other one and share no vertex, by~\cref{obs:union} we have that~$G[\mathcal I]$ is an outerplane graph. If~$\mathcal I_L$ satisfies Property~B, then~$G[\mathcal I_L \cup (\{x\}\cap \mathcal I_H)]$ is outerplane by~\cref{le:biconnectedK}. Since~$G[\mathcal I_H]$ and~$G[\mathcal I_L \cup (\{x\}\cap \mathcal I_H)]$ are each in the outer face of the other one and share either no vertex or a single vertex, by~\cref{obs:union} we have that~$G[\mathcal I]$ is an outerplane graph. Similarly, if~$\mathcal I_L$ satisfies Property~C, then~$G[\mathcal I_L \cup (\{y\}\cap \mathcal I_H)]$ is outerplane by~\cref{le:biconnectedK}. Since~$G[\mathcal I_H]$ and~$G[\mathcal I_L \cup (\{y\}\cap \mathcal I_H)]$ are each in the outer face of the other one and share either no vertex or a single vertex, by~\cref{obs:union} we have that~$G[\mathcal I]$ is an outerplane graph. Finally, if~$\mathcal I_L$ satisfies Property~D, then~$G[\mathcal I_L \cup (\{x,y\}\cap \mathcal I_H)]$ is outerplane by~\cref{le:biconnectedK}. Since~$G[\mathcal I_H]$ and~$G[\mathcal I_L \cup (\{x,y\}\cap \mathcal I_H)]$ are each in the outer face of the other one and share either no vertex, or a single vertex, or a single edge, by~\cref{obs:union} we have that~$G[\mathcal I]$ is an outerplane graph.

In the remaining cases,~$K$ is not biconnected. Let~$B$ be an extremal leaf of~$K$. {\em Suppose first that~$B$ is a non-trivial block.}

\paragraph*{Case~5:~$B$ is non-trivial and cushy} 

Let~$n_B=|V(G_B)|$, where~$G_B$ is the cage graph of~$B$. Let~$c_B,w_1=u_B,w_2,\dots,w_{s-1},w_s=v_B$ be the clockwise order of the vertices along the outer face of~$B$, and let~$c_B,z_1=\ell_B,z_2,\dots,z_{k-1},z_k=r_B$ be the clockwise order of the vertices along the outer face of~$G_B$. The \emph{internal degree}~$d(z_i)$ of a vertex~$z_i\neq c_B$ on the boundary of~$G_B$ is the number of its neighbors in~$B$. Note that, since Cases 1 and 2 do not apply, the internal degree of every vertex among~$z_2,\dots,z_{k-1}$ is at least two. The internal degree of~$z_1$ and~$z_k$ is also at least two. Indeed,~$z_1$ and~$z_k$ are both neighbors of~$c_B$. Also,~$c_B\ell_B$ is the edge of~$G$ that precedes~$c_Bu_B$ in clockwise direction, hence the edge~$\ell_Bu_B$ exists, since~$G$ is internally-triangulated. Likewise, the edge~$r_Bv_B$ exists. We further distinguish two cases.

In {\bf Case~5.1},~$d(z_1)=d(z_k)=2$ and either (i)~$k=3$ and~$d(z_2)=2$ (see \cref{fig:algorithm-case5.1-edge-l_b-r_b-not-in-G-i-holds}), or (ii)~$k\geq 3$ and~$d(z_j)=3$, for~$j=2,\dots,k-1$ (see \cref{fig:algorithm-case5.1-edge-l_b-r_b-not-in-G-ii-holds}). Note that, if (i) holds, then~$k=3$,~$s=2$, and~$n_B=6$, while if (ii) holds, then~$s=2(k-2)+1=2k-3$ and~$n_B=3(k-2)+4=3k-2$.

\begin{figure}[tb]
\centering
    \begin{subfigure}{0.48\textwidth}
		\centering
		\includegraphics[scale=1.2, page=2]{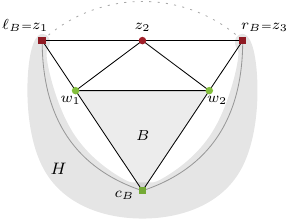}
		\subcaption{}
        \label{fig:algorithm-case5.1-edge-l_b-r_b-not-in-G-i-holds}
	\end{subfigure}
    \hfill
    \begin{subfigure}{0.48\textwidth}
		\centering
		\includegraphics[scale=1.2, page=3]{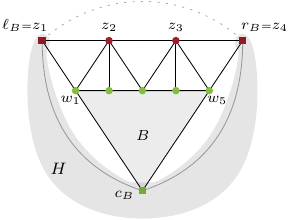}
		\subcaption{}
        \label{fig:algorithm-case5.1-edge-l_b-r_b-not-in-G-ii-holds}
	\end{subfigure}
\caption{Illustrations for Case~5.1 when the edge~$\ell_B r_B$ does not belong to~$G$ (a) if~$k=3$ and~$d(z_2) = 2$, and (b) if~$k \ge 3$ and~$d(z_j) = 3$, for~$j=2,\dots,k-1$.}
\label{fig:algorithm-case5.1-edge-l_b-r_b-not-in-G}
\end{figure}

Suppose first that the edge~$\ell_Br_B$ does not belong to~$G$ (refer to \cref{fig:algorithm-case5.1-edge-l_b-r_b-not-in-G}). We construct a~$2$-outerplane graph~$H$ from~$G$ by removing the vertices in~$V(G_B)-\{c_B,\ell_B,r_B\}$ and their incident edges, and by adding the edge~$\ell_Br_B$ in the outer face. We apply induction on~$H$, so to find a good set~$\mathcal I_H$ for~$H$. We define~$\mathcal I=\mathcal I_H\cup (V(B)-\{c_B\})$. Note that~$c_B$,~$\ell_B$, and~$r_B$ belong to~$\mathcal I$ if and only if they belong to~$\mathcal I_H$. We prove that~$\mathcal I$ is a good set. First, we have~$|\mathcal I|=|\mathcal I_H|+s$ and~$|V(H)|=n-n_B+3$. If (i) holds, then~$|\mathcal I|=|\mathcal I_H|+2\geq 2(n-3)/3+2=2n/3$. If (ii) holds, then~$|\mathcal I|=|\mathcal I_H|+2(k-2)+1\geq 2(n-3(k-2)-1)/3+2(k-2)+1>2n/3$. Suppose, for a contradiction, that~$G[\mathcal I]$ contains a cycle~$\mathcal C$ which contains a vertex~$w\in \mathcal I$ in its interior. Note that~$w\notin \{w_1,\dots,w_s\}$, since~$w_1,\dots,w_s$ are incident to the outer face of~$G[\mathcal I]$, hence~$w\in \mathcal I_H$. Also,~$\mathcal C$ uses vertices among~$w_1,\dots,w_s$, as otherwise~$G[\mathcal I_H]$ would also contain~$\mathcal C$, and~$\mathcal I_H$ would not be a good set. Since any cycle in~$G$ on the vertices~$V(B)\cup \{\ell_B,r_B\}$ does not include any vertex in its interior, it follows that~$\mathcal C$ also contains vertices of~$V(H)- \{\ell_B,r_B\}$. Let then~$\mathcal P$ be the maximal path in~$\mathcal C$ on the vertices in~$V(B)\cup \{\ell_B,r_B\}$ and note that the end-vertices of~$\mathcal P$ are two among~$c_B,\ell_B,r_B$. Thus,~$\mathcal P$ can be replaced by the edge between its end-vertices, resulting in a cycle of~$H$ which also contains~$w$ in its interior and whose vertices all belong to~$\mathcal I_H$, which implies that~$\mathcal I_H$ is not a good set, a contradiction.

Suppose next that the edge~$\ell_Br_B$ belongs to~$G$. Let~$b$ be the number of vertices inside the cycle~$c_B\ell_Br_B$. Note that~$b\geq 1$, given that~$K$ is not biconnected. Let~$G'_B$ be the subgraph of~$G$ induced by the vertices inside or on the boundary of the cycle~$C'_B:=z_1,z_2,\dots,z_{k-1},z_k$. Observe that~$G'_B$ consists of~$C'_B$ and of~$K$, which is inside~$C'_B$. Hence,~$|V(G'_B)|=n_B+b$.
\begin{itemize}
    \item If~$b\geq 3$ (refer to \cref{fig:algorithm-case5.1-edge-l_b-r_b-in-G-and-b-is-at-least-3}), let~$H$ be the subgraph of~$G$ induced by~$V(G)-V(G'_B)$. We apply induction on~$H$, so to find a good set~$\mathcal I_H$ for~$H$. We define~$\mathcal I=\mathcal I_H\cup V(K)$. We prove that~$\mathcal I$ is a good set. First, we have~$|\mathcal I|=|\mathcal I_H|+|V(K)|$ and~$|V(H)|=n-|V(G'_B)|$. If (i) holds (see \cref{fig:algorithm-case5.1-edge-l_b-r_b-in-G-and-b-is-at-least-3-i-holds}), then we have~$n_B=6$ and, by the assumption~$b\geq 3$,  we have~$|V(G'_B)|\geq 9$. It follows that~$|\mathcal I|=|\mathcal I_H|+|V(G'_B)|-3\geq 2(n-|V(G'_B)|)/3 +|V(G'_B))|-3=2n/3+|V(G'_B)|/3-3\geq 2n/3$. If (ii) holds (see \cref{fig:algorithm-case5.1-edge-l_b-r_b-in-G-and-b-is-at-least-3-ii-holds}), then, by the assumption~$b\geq 3$,  we have~$|V(G'_B)|\geq 3k+1$. It follows that~$|\mathcal I|=|\mathcal I_H|+|V(G'_B)|-k \geq 2(n-|V(G'_B)|)/3 +|V(G'_B))|-k=2n/3+|V(G'_B)|/3-k> 2n/3$. Also,~$G[\mathcal I_H]$ is outerplane by induction and~$K=G[V(K)]$ is outerplane since it is a component of~$G[L_2]$. Since~$G[\mathcal I_H]$ and~$K$ are each in the outer face of the other one, and share no vertex, by~\cref{obs:union}, we have that~$G[\mathcal I]$ is an outerplane graph. 
    \item If~$b=1$ or~$b=2$ (refer to \cref{fig:algorithm-case5.1-edge-l_b-r_b-in-G-and-b-is-at-most-2}), let~$U$ be the set of vertices inside~$c_B\ell_Br_B$; we have~$U=\{u_1\}$ if~$b=1$ and~$U=\{u_1,u_2\}$ if~$b=2$. Since~$G$ is internally-triangulated, there exists one of~$\ell_B$ and~$r_B$, possibly both, that is adjacent to all the vertices in~$U$. Assume that~$\ell_B$ is adjacent to all the vertices in~$U$, as the other case is similar. Let~$H$ be the subgraph of~$G$ induced by~$V(G)-(V(G'_B)-\{r_B\})$. We apply induction on~$H$, so to find a good set~$\mathcal I_H$ for~$H$. We define~$\mathcal I=\mathcal I_H\cup V(K)$. We prove that~$\mathcal I$ is a good set. First, we have~$|\mathcal I|=|\mathcal I_H|+|V(K)|$ and~$|V(H)|=n-|V(G'_B)|+1$. If (i) holds (see \cref{fig:algorithm-case5.1-edge-l_b-r_b-in-G-and-b-is-at-most-2-i-holds}), then we have~$n_B=6$ and~$|V(G'_B)|\geq 7$. It follows that~$|\mathcal I|=|\mathcal I_H|+|V(G'_B)|-3\geq 2(n-|V(G'_B)|+1)/3 +|V(G'_B))|-3=2n/3+|V(G'_B)|/3-7/3\geq 2n/3$. If (ii) holds (see \cref{fig:algorithm-case5.1-edge-l_b-r_b-in-G-and-b-is-at-most-2-ii-holds}), then we have~$|V(G'_B)|\geq 3k-1$. It follows that~$|\mathcal I|=|\mathcal I_H|+|V(G'_B)|-k \geq 2(n-|V(G'_B)|+1)/3 +|V(G'_B))|-k=2n/3+|V(G'_B)|/3-k+2/3> 2n/3$. Also,~$G[\mathcal I_H]$ is outerplane by induction and~$G[V(K)\cup (\mathcal I_H \cap \{r_B\})]$ is outerplane since each vertex of~$K$ is adjacent to a vertex among~$z_1=\ell_B,z_2,\dots,z_{k-1}$; note that all such vertices do not belong to~$\mathcal I$. Since~$G[\mathcal I_H]$ and~$G[V(K)\cup (\mathcal I_H \cap \{r_B\})]$ are each in the outer face of the other one, and share either no vertex or a single vertex, by~\cref{obs:union}, we have that~$G[\mathcal I]$ is an outerplane graph.  
\end{itemize}

\begin{figure}[tb]
\centering
    \begin{subfigure}{0.48\textwidth}
		\centering
		\includegraphics[scale=1, page=4]{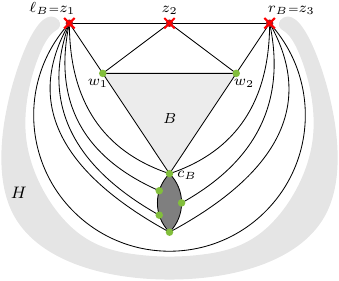}
		\subcaption{}
        \label{fig:algorithm-case5.1-edge-l_b-r_b-in-G-and-b-is-at-least-3-i-holds}
	\end{subfigure}
    \hfill
    \begin{subfigure}{0.48\textwidth}
		\centering
		\includegraphics[scale=1, page=5]{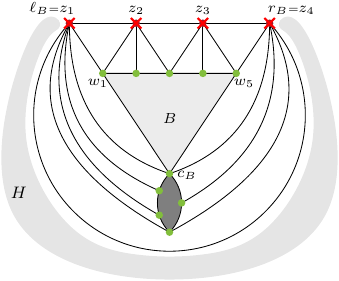}
		\subcaption{}
        \label{fig:algorithm-case5.1-edge-l_b-r_b-in-G-and-b-is-at-least-3-ii-holds}
	\end{subfigure}
\caption{Illustrations for Case~5.1 when the edge~$\ell_B r_B$ belongs to~$G$ and~$b \geq 3$ (a) if~$k=3$ and~$d(z_2) = 2$, and (b) if~$k \ge 3$ and~$d(z_j) = 3$, for~$j=2,\dots,k-1$.}
\label{fig:algorithm-case5.1-edge-l_b-r_b-in-G-and-b-is-at-least-3}
\end{figure}

\begin{figure}[tb]
\centering
    \begin{subfigure}{0.48\textwidth}
		\centering
		\includegraphics[scale=1.1, page=6]{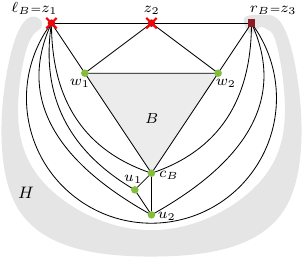}
		\subcaption{}
        \label{fig:algorithm-case5.1-edge-l_b-r_b-in-G-and-b-is-at-most-2-i-holds}
	\end{subfigure}
    \hfill
    \begin{subfigure}{0.48\textwidth}
		\centering
		\includegraphics[scale=1.1, page=7]{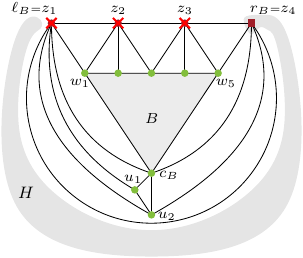}
		\subcaption{}
        \label{fig:algorithm-case5.1-edge-l_b-r_b-in-G-and-b-is-at-most-2-ii-holds}
	\end{subfigure}
\caption{Illustrations for Case~5.1 when the edge~$\ell_B r_B$ belongs to~$G$ and~$b=1$ or~$b=2$ (a) if~$k=3$ and~$d(z_2) = 2$, and (b) if~$k \ge 3$ and~$d(z_j) = 3$, for~$j=2,\dots,k-1$.}
\label{fig:algorithm-case5.1-edge-l_b-r_b-in-G-and-b-is-at-most-2}
\end{figure}

In {\bf Case~5.2}, we have that Case~5.1 does not apply. Although the setting is different, the general strategy we employ to handle this case is similar to the one of Case~4. We find a set~$\mathcal I_B$ of vertices in~$G_B$ so that excluding or disregarding the vertices that~$G_B$ shares with the rest of the graph (these vertices are~$c_B$,~$\ell_B$, and~$r_B$, which play a role similar to the one that in Case~4 is played by~$x$ and~$y$) we can ensure, on one hand, that the number of vertices in~$\mathcal I_B$ is at least twice the number of vertices of~$G_B$ that are excluded from~$\mathcal I_B$ and, on the other hand, that~$\mathcal I_B$, together with the set of vertices that are inductively selected in the rest of the graph, induces an outerplane graph. Formally, we prove the following.

\begin{lemma} \label{le:cushy}
There exists a set~$\mathcal I_B\subset V(G_B)$ such that~$\{c_B,\ell_B,r_B\} \cap \mathcal I_B=\emptyset$ and such that (at least) one of the following properties is satisfied:
\begin{itemize}
    \item {\em Property~A:}~$|\mathcal I_B|\geq 2(n_B-1)/3$ and~$G[\mathcal I_B\cup\{c_B\}]$ is an outerplane graph;
    \item {\em Property~B:}~$|\mathcal I_B|\geq 2(n_B-2)/3$ and~$G[\mathcal I_B\cup\{c_B,\ell_B\}]$ is an outerplane graph; and
    \item {\em Property~C:}~$|\mathcal I_B|\geq 2(n_B-2)/3$ and~$G[\mathcal I_B\cup\{c_B,r_B\}]$ is an outerplane graph.
\end{itemize}
\end{lemma}

\begin{proof}
Let~$z_{f(1)},z_{f(2)},\dots,z_{f(m)}$ be the sequence~$z_2,\dots,z_{k-1}$ restricted to the~$m\geq 0$ vertices with internal degree~$2$; observe that, differently from the proof of~\cref{le:biconnectedK}, the vertices~$z_1$ and~$z_k$ are not in the sequence~$z_{f(1)},z_{f(2)},\dots,z_{f(m)}$. 

The set~$\mathcal I_B$ is defined according to the following rules. First, all the vertices of~$B$, except for~$c_B$, are in~$\mathcal I_B$. Second,~$\ell_B$ and~$r_B$ are not in~$\mathcal I_B$. Third, each vertex~$z_j$ with~$d(z_j)\geq 3$ is not in~$\mathcal I_B$. Finally, a vertex~$z_{f(j)}$ with internal degree~$2$ is in~$\mathcal I_B$ if~$j$ is odd and not in~$\mathcal I_B$ if~$j$ is even. Note that~$\{c_B,\ell_B,r_B\} \cap \mathcal I_B=\emptyset$, as required. We now distinguish three cases.

\begin{figure}[tb]
\centering
    \begin{subfigure}{0.32\textwidth}
		\centering
		\includegraphics[scale=1, page=9]{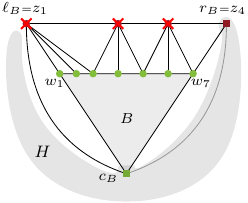}
		\subcaption{}
        \label{fig:algorithm-case5.2.1-i}
	\end{subfigure}
    \hfill
    \begin{subfigure}{0.32\textwidth}
		\centering
		\includegraphics[scale=1, page=10]{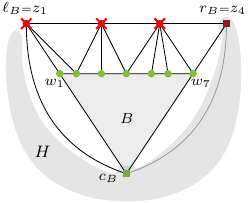}
		\subcaption{}
        \label{fig:algorithm-case5.2.1-ii}
    \end{subfigure}
     \hfill
    \begin{subfigure}{0.32\textwidth}
		\centering
		\includegraphics[scale=1, page=11]{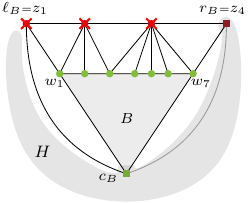}
		\subcaption{}
        \label{fig:algorithm-case5.2.1-iii}
	\end{subfigure}

\caption{Illustrations for Case~5.2.1 (a) if (i) holds, (b) if (ii) holds, and (c) if (iii) holds.}
\label{fig:algorithm-case5.2.1-i-ii-iii}
\end{figure}

\begin{figure}[tb]
\centering
    \begin{subfigure}{0.48\textwidth}
		\centering
		\includegraphics[scale=1.2, page=12]{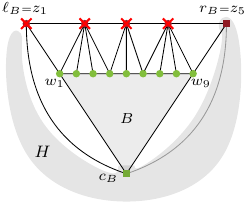}
		\subcaption{}
        \label{fig:algorithm-case5.2.1-iv}
	\end{subfigure}
    \hfill
    \begin{subfigure}{0.48\textwidth}
		\centering
		\includegraphics[scale=1.2, page=13]{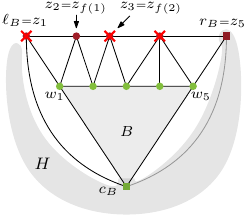}
		\subcaption{}
        \label{fig:algorithm-case5.2.1-v}
	\end{subfigure}

\caption{Illustrations for Case~5.2.1 (a) if (iv) holds and (b) if (v) holds.}
\label{fig:algorithm-case5.2.1-iv-v}
\end{figure}

\begin{figure}[tb]
\centering
    \begin{subfigure}{0.48\textwidth}
		\centering
		\includegraphics[scale=1.2, page=14]{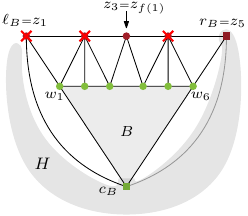}
		\subcaption{}
        \label{fig:algorithm-case5.2.1-vi}
	\end{subfigure}
    \hfill
    \begin{subfigure}{0.48\textwidth}
		\centering
		\includegraphics[scale=1.2, page=15]{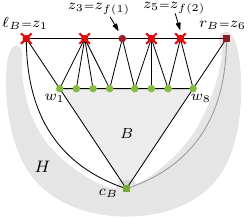}
		\subcaption{}
        \label{fig:algorithm-case5.2.1-vii}
	\end{subfigure}

\caption{Illustrations for Case~5.2.1 (a) if (vi) holds and (b) if (vii) holds.}
\label{fig:algorithm-case5.2.1-vi-vii}
\end{figure}

In {\bf Case~5.2.1}, we have~$d(z_k)=2$ and at least one of the following conditions is satisfied: 
\begin{enumerate}[(i)]
\item~$3\leq d(z_1)\leq 4$ and~$d(z_2)=d(z_3)=\dots=d(z_{k-1})=3$ (see \cref{fig:algorithm-case5.2.1-i}); 
\item~$d(z_1)\leq 3$,~$d(z_i)=4$, and~$d(z_j)=3$, for some~$i\in \{2,\dots,k-1\}$ and for~$j=2,\dots,k-1$ with~$j\neq i$ (see \cref{fig:algorithm-case5.2.1-ii}); 
\item~$d(z_1)=2$,~$d(z_i)=5$, and~$d(z_j)=3$, for some~$i\in \{2,\dots,k-1\}$ and for~$j=2,\dots,k-1$ with~$j\neq i$ (see \cref{fig:algorithm-case5.2.1-iii}); 
\item~$d(z_1)=2$,~$d(z_i)=4$,~$d(z_{i'})=4$, and~$d(z_j)=3$, for some distinct~$i,i'\in \{2,\dots,k-1\}$ and for~$j=2,\dots,k-1$ with~$j\neq i$ and~$j\neq i'$ (see \cref{fig:algorithm-case5.2.1-iv}); 
\item~$m=2$ and~$d(z_j)\leq 3$ for~$j=1,\dots,k-1$, or~$m=4$,~$d(z_1)=2$, and~$d(z_j)\leq 3$ for~$j=2,\dots,k-1$ (see \cref{fig:algorithm-case5.2.1-v}); 
\item~$m=1$,~$d(z_1)=2$,~$d(z_{k-1})=3$, and~$d(z_j)\leq 3$ for~$j=2,\dots,k-2$ (see \cref{fig:algorithm-case5.2.1-vi}); or 
\item~$d(z_1)=2$,~$m=2$,~$d(z_i)=4$, and~$d(z_j)\leq 3$, for some~$i\in \{2,\dots,k-1\}$ and for~$j=2,\dots,k-1$ with~$j\neq i$ (see \cref{fig:algorithm-case5.2.1-vii}).
\end{enumerate}
In this case, we have that~$\mathcal I_B$ satisfies Property~C. 

In order to prove that~$G[\mathcal I_B\cup\{c_B,r_B\}]$ is an outerplane graph, it suffices to show that, for each vertex in~$V(B)-\{c_B\}$, there exists a neighbor among~$z_1,\dots,z_k$ which is not in~$\mathcal I_B\cup \{c_B,r_B\}$. This is done similarly to the proof of~\cref{le:biconnectedK}. Namely, consider a vertex~$w_i$, for some~$2\leq i \leq s-1$. If~$w_i$ is a neighbor of a vertex~$z_j$ with~$d(z_j)\geq 3$, then~$z_j$ is not in~$\mathcal I_B\cup \{c_B,r_B\}$, hence it is the desired neighbor of~$w_i$. Each remaining vertex~$w_i$ with~$2\leq i \leq s-1$ is only neighbor of vertices~$z_j$ with~$d(z_j)=2$ and with~$2\leq j\leq k-1$; indeed,~$w_i$ is not a neighbor of~$z_k$, given that~$d(z_k)=2$ and~$i\leq s-1$, and if~$w_i$ were a neighbor of~$z_1$, then~$d(z_1)\geq 3$, given that~$i\geq 2$. Since~$G_B$ is internally-triangulated, it follows that~$w_i$ is a neighbor of two vertices~$z_{f(j)}$ and~$z_{f(j+1)}$ with~$d(z_{f(j)})=d(z_{f(j+1)})=2$ that are consecutive along the boundary of~$G_B$. By construction, one of them is not in~$\mathcal I_B$, hence it is the desired neighbor of~$w_i$. Furthermore,~$w_1$ is adjacent to~$z_1=\ell_B$, which is not in~$\mathcal I_B\cup \{c_B,r_B\}$. Finally,~$w_s$ is adjacent to~$z_{k-1}$, since~$d(z_k)=2$, and~$z_{k-1}$ is not in~$\mathcal I_B\cup \{c_B,r_B\}$. Indeed, if~$k=2$, then~$z_{k-1}=z_1$ is not in~$\mathcal I_B\cup \{c_B,r_B\}$ by construction. If~$k\geq 3$ and~$d(z_{k-1})\geq 3$, then~$z_{k-1}$ is not in~$\mathcal I_B$ by construction. That~$d(z_{k-1})=2$ can only happen in cases (v) or (vii), where we have~$m=2$ or~$m=4$. Then~$z_{k-1}$ is not in~$\mathcal I_B$ since a vertex~$z_{f(j)}$ with~$j$ even and with~$d(z_{f(j)})=2$ is not in~$\mathcal I_B$.

We now prove that~$|\mathcal I_B|\geq 2(n_B-2)/3$. 
\begin{enumerate}[(i)]
    \item If~$d(z_1)=d(z_2)=d(z_3)=\dots=d(z_{k-1})=3$, then~$|\mathcal I_B|=s=2k-2$ and~$n_B=3k-1$. Hence, we have that~$|\mathcal I_B|=2k-2=2((3k-1)-2)/3=2(n_B-2)/3$. If~$d(z_1)=4$ and~$d(z_2)=d(z_3)=\dots=d(z_{k-1})=3$, then~$|\mathcal I_B|=s=2k-1$ and~$n_B=3k$. Hence, we have that~$|\mathcal I_B|=2k-1>2(3k-2)/3=2(n_B-2)/3$.    
    \item If~$d(z_1)=2$,~$d(z_i)=4$, and~$d(z_j)=3$, for some~$i\in \{2,\dots,k-1\}$ and for~$j=2,\dots,k-1$ with~$j\neq i$, then~$|\mathcal I_B|=s=2k-2$ and~$n_B=3k-1$. Hence, we have that~$|\mathcal I_B|=2k-2=2((3k-1)-2)/3=2(n_B-2)/3$. If~$d(z_1)=3$,~$d(z_i)=4$, and~$d(z_j)=3$, for some~$i\in \{2,\dots,k-1\}$ and for~$j=2,\dots,k-1$ with~$j\neq i$, then~$|\mathcal I_B|=s=2k-1$ and~$n_B=3k$. Hence, we have that~$|\mathcal I_B|=2k-1>2(3k-2)/3=2(n_B-2)/3$. 
    \item If~$d(z_1)=2$,~$d(z_i)=5$, and~$d(z_j)=3$, for some~$i\in \{2,\dots,k-1\}$ and for~$j=2,\dots,k-1$ with~$j\neq i$, then~$|\mathcal I_B|=s=2k-1$ and~$n_B=3k$. Hence, we have that~$|\mathcal I_B|=2k-1>2(3k-2)/3=2(n_B-2)/3$.     
    \item If~$d(z_1)=2$,~$d(z_i)=4$,~$d(z_{i'})=4$, and~$d(z_j)=3$, for some distinct~$i,i'\in \{2,\dots,k-1\}$ and for~$j=2,\dots,k-1$ with~$j\neq i$ and~$j\neq i'$, then~$|\mathcal I_B|=s=2k-1$ and~$n_B=3k$. Hence, we have that~$|\mathcal I_B|=2k-1>2(3k-2)/3=2(n_B-2)/3$.     
    \item If~$m=2$,~$d(z_1)=2$, and~$d(z_j)\leq 3$ for~$j=2\dots,k-1$, then~$|\mathcal I_B|=s+1=2k-4$ and~$n_B=3k-4$. Hence, we have that~$|\mathcal I_B|=2k-4=2((3k-4)-2)/3=2(n_B-2)/3$. If~$m=2$,~$d(z_1)=3$, and~$d(z_j)\leq 3$ for~$j=2,\dots,k-1$, then~$|\mathcal I_B|=s+1=2k-3$ and~$n_B=3k-3$. Hence, we have that~$|\mathcal I_B|=2k-3>2((3k-3)-2)/3=2(n_B-2)/3$. If~$m=4$,~$d(z_1)=2$, and~$d(z_j)\leq 3$ for~$j=2,\dots,k-1$, then~$|\mathcal I_B|=s+2=2k-5$ and~$n_B=3k-6$. Hence, we have that~$|\mathcal I_B|=2k-5>2((3k-6)-2)/3=2(n_B-2)/3$.
    \item If~$m=1$,~$d(z_1)=2$,~$d(z_{k-1})=3$, and~$d(z_j)\leq 3$ for~$j=2,\dots,k-2$, then~$|\mathcal I_B|=s+1=2k-3$ and~$n_B=3k-3$. Hence, we have that~$|\mathcal I_B|=2k-3>2((3k-3)-2)/3=2(n_B-2)/3$. 
    \item If~$d(z_1)=2$,~$m=2$,~$d(z_i)=4$, and~$d(z_j)\leq 3$, for some~$i\in \{2,\dots,k-1\}$ and for~$j=2,\dots,k-1$ with~$j\neq i$, then~$|\mathcal I_B|=s+1=2k-3$ and~$n_B=3k-3$. Hence, we have that~$|\mathcal I_B|=2k-3>2((3k-3)-2)/3=2(n_B-2)/3$.
\end{enumerate}

In {\bf Case~5.2.2}, we have~$d(z_1)=2$ and at least one of the following conditions is satisfied: 
\begin{enumerate}[(i)] 
\item~$3\leq d(z_k)\leq 4$ and~$d(z_2)=d(z_3)=\dots=d(z_{k-1})=3$; 
\item~$d(z_k)\leq 3$,~$d(z_i)=4$, and~$d(z_j)=3$, for some~$i\in \{2,\dots,k-1\}$ and for~$j=2,\dots,k-1$ with~$j\neq i$; \item~$d(z_k)=2$,~$d(z_i)=5$, and~$d(z_j)=3$, for some~$i\in \{2,\dots,k-1\}$ and for~$j=2,\dots,k-1$ with~$j\neq i$; 
\item~$d(z_k)=2$,~$d(z_i)=4$,~$d(z_{i'})=4$, and~$d(z_j)=3$, for some distinct~$i,i'\in \{2,\dots,k-1\}$ and for~$j=2,\dots,k-1$ with~$j\neq i$ and~$j\neq i'$; 
\item~$m=2$ and~$d(z_j)\leq 3$ for~$j=2,\dots,k$, or~$m=4$,~$d(z_k)=2$, and~$d(z_j)\leq 3$ for~$j=2,\dots,k-1$; 
\item~$m=1$,~$d(z_k)=2$,~$d(z_2)=3$, and~$d(z_j)\leq 3$ for~$j=3,\dots,k-1$; or 
\item~$d(z_k)=2$,~$m=2$,~$d(z_i)=4$, and~$d(z_j)\leq 3$, for some~$i\in \{2,\dots,k-1\}$ and for~$j=2,\dots,k-1$ with~$j\neq i$. 
\end{enumerate}
In this case, we have that~$\mathcal I_B$ satisfies Property~B. The proof of this statement is symmetric to Case~5.2.1. Note that Cases 5.2.1 and 5.2.2 might hold both true for~$G$. 

If we are neither in Case~5.2.1 nor in Case~5.2.2, then we are in {\bf Case~5.2.3} and we have that~$\mathcal I_B$ satisfies Property~A. The proof that~$G[\mathcal I_B\cup\{c_B\}]$ is an outerplane graph is similar, and actually simpler, to Case~5.2.1, since the neighbors~$z_1$ and~$z_k$ of~$w_1$ and~$w_s$, respectively, are now not in~$\mathcal I_B\cup \{c_B\}$. 

The proof that~$|\mathcal I_B|\geq 2(n_B-1)/3$ is similar in spirit to, however more complex than, the proof of~\cref{le:biconnectedK}. We charge each vertex of~$G_B$ not in~$\mathcal I_B$ and different from~$c_B$ to two vertices in~$\mathcal I_B$, so that each vertex in~$\mathcal I_B$ is charged with at most one vertex. 

The charging scheme is defined iteratively on the vertices~$z_1,\dots,z_k$ that do not belong to~$\mathcal I_B$. Initially, we charge~$z_1$ as follows. If~$d(z_1)\geq 3$, then~$z_1$ is charged to~$w_1,w_2$, while if~$d(z_1)=2$, then~$z_1$ is charged to~$w_1$ (and one more charged vertex will be defined later). For any~$j=1,\dots,k$, let~$\mathcal I_B^j$ denote the subset of~$\mathcal I_B$ which comprises all the vertices~$z_i\in \mathcal I_B$ with~$i\leq j$ and all the vertices~$w_1,w_2,\dots,w_{g(j)}$, where~$w_{g(j)}$ is the last neighbor of~$z_j$ in the sequence~$w_1,w_2,\dots,w_s$. 
After having charged the vertices among~$z_1,z_2,\dots,z_j$ that are not in~$\mathcal I_B^j$ to vertices in~$\mathcal I^j_B$, we have a \emph{stock of charge}~$\gamma_j$, which is an integer that represents the number of vertices in~$\mathcal I^j_B$ that have not been charged with any vertex among~$z_1,z_2,\dots,z_j$. Initially, this is defined as~$\gamma_1=d(z_1)-3$. Indeed,~$\mathcal I_B^1=\{w_1,\dots,w_{d(z_1)-1}\}$ (note that~$c_B$ contributes to~$d(z_1)$, however it is not among~$w_1,w_2,\dots,w_s$). Hence, if~$d(z_1)\geq 3$, then~$\gamma_j$ correctly represents the fact that~$z_1$ has been charged to~$w_1,w_2$ and the vertices~$w_3,w_4,\dots,w_{g(j)}$ have not been charged with~$z_1$. If~$d(z_1)=2$, then~$\gamma_j=-1$, which represents the fact that~$z_1$ needs to be charged to one additional vertex.

We now prove that, for any~$j\in\{1,\dots,k-1\}$, we can design a charging scheme that satisfies the following properties. First, the vertices~$z_1,\dots,z_j$ have been charged to two vertices in~$\mathcal I^j_B$, with the possible exception of~$z_1$ which might have been charged to~$w_1$ only, so that each vertex in~$\mathcal I^j_B$ has been charged with at most one vertex. Second, if~$d(z_1)\geq 3$ or if~$d(z_i)\neq 3$, for some~$2\leq i\leq j$, then~$z_1$ has also been charged to two vertices in~$\mathcal I^j_B$. Third, let~$m_j$ be equal to the number of vertices~$z_i$ among~$z_2,z_3,\dots,z_j$ such that~$d(z_i)=2$. Then the stock of charge~$\gamma_j$ is equal to the sum~$S_j$ of the following terms:
\begin{itemize}
\item~$d(z_1)-3$;
\item~$d(z_i)-3$, for each vertex~$z_i$ with~$d(z_i)\geq 3$ and with~$2\leq i\leq j$;
\item~$\lfloor m_j/2 \rfloor$; and
\item~$2$, if~$m_j$ is odd. 
\end{itemize}
Observe that~$S_j\geq -1$, and that~$S_j=-1$ if and only if~$d(z_1)=2$ and~$d(z_i)=3$, for~$i=2,\dots,j$. 

The initial charge of~$z_1$ trivially satisfies these properties; note that~$m_1=0$. Suppose now that the charging scheme for~$z_1,\dots,z_j$  satisfies these properties, for some~$j\in \{1,\dots,k-2\}$. We prove that the same properties are satisfied after considering~$z_{j+1}$.
\begin{itemize}
\item Suppose first that~$d(z_{j+1})=3$. Then~$|\mathcal I^{j+1}_B|=|\mathcal I^{j}_B|+2$, since~$w_{g(j+1)-1}$ and~$w_{g(j+1)}$ belong to~$\mathcal I^{j+1}_B$ and not to~$\mathcal I^{j}_B$; note that~$w_{g(j+1)-2}$ is also a neighbor of~$z_j$, hence it is already in~$\mathcal I^{j}_B$. Hence, we can charge~$z_{j+1}$ to two uncharged vertices in~$\mathcal I^{j+1}_B$ and all the required properties remain satisfied. In particular,~$z_1$ has been charged to two vertices in~$\mathcal I^{j+1}_B$ if and only if it has been charged to two vertices in~$\mathcal I^j_B$, given that~$d(z_i)\neq 3$ for some~$2\leq i\leq j+1$ if and only if~$d(z_i)\neq 3$ for some~$2\leq i\leq j$. Also,~$\gamma_{j+1}=\gamma_j$, given that~$\mathcal I^{j+1}_B$ contains two vertices more than~$\mathcal I^{j}_B$ and that~$z_{j+1}$ is charged to two vertices in~$\mathcal I^{j+1}_B$. Note that~$S_{j+1}=S_j$, given that~$m_{j+1}=m_j$ and that~$d(z_{j+1})-3=0$. Since~$\gamma_{j}=S_{j}$, by hypothesis, we have that~$\gamma_{j+1}=S_{j+1}$.
\item Suppose next that~$d(z_{j+1})>3$. Then~$|\mathcal I^{j+1}_B|=|\mathcal I^{j}_B|+d(z_{j+1})-1>|\mathcal I^{j}_B|+2$. Hence, after charging~$z_{j+1}$ to two uncharged vertices in~$\mathcal I^{j+1}_B$, there is at least one more uncharged vertex in~$\mathcal I^{j+1}_B$ than in~$\mathcal I^{j}_B$. If~$\gamma_{j}=-1$, then we also charge~$z_1$ to an uncharged vertex in~$\mathcal I^{j+1}_B$, and all the required properties remain satisfied. In particular,~$z_1$ has now been charged to two vertices in~$\mathcal I^{j+1}_B$. Also,~$\gamma_{j+1}=\gamma_j+d(z_{j+1})-3$, given that~$\mathcal I^{j+1}_B$ contains~$d(z_{j+1})-1$ vertices more than~$\mathcal I^{j}_B$ and that~$z_{j+1}$ is charged to two vertices in~$\mathcal I^{j+1}_B$; the fact that~$\gamma_{j}=-1$ if~$z_1$ has only been charged to one vertex of~$\mathcal I^{j}_B$ makes the equation~$\gamma_{j+1}=\gamma_j+d(z_{j+1})-3$ valid even if we require a third vertex of~$\mathcal I^{j+1}_B$ to be charged with~$z_1$. Note that~$S_{j+1}=S_j+d(z_{j+1})-3$, given that~$m_{j+1}=m_j$. Since~$\gamma_{j}=S_{j}$, by hypothesis, we have that~$\gamma_{j+1}=S_{j+1}$. 
\item Suppose next that~$d(z_{j+1})=2$ and that~$m_{j+1}$ is odd. Then~$|\mathcal I^{j+1}_B|=|\mathcal I^{j}_B|+2$, since~$z_{j+1}$ and~$w_{g(j+1)}$ belong to~$\mathcal I^{j+1}_B$ and not to~$\mathcal I^{j}_B$. Note that, in this case, we do not need to charge~$z_{j+1}$ to any vertex, since~$z_{j+1}$ belongs to~$\mathcal I_B$. However, if~$\gamma_{j}=-1$, then we also charge~$z_1$ to an uncharged vertex in~$\mathcal I^{j+1}_B$, and all the required properties remain satisfied. In particular,~$z_1$ has now been charged to two vertices in~$\mathcal I^{j+1}_B$. Also,~$\gamma_{j+1}=\gamma_j+2$, given that~$\mathcal I^{j+1}_B$ contains~$2$ vertices more than~$\mathcal I^{j}_B$; the fact that~$\gamma_{j}=-1$ if~$z_1$ has only been charged to one vertex of~$\mathcal I^{j}_B$ makes the equation~$\gamma_{j+1}=\gamma_j+2$ valid even if we require a vertex of~$\mathcal I^{j+1}_B$ to be charged with~$z_1$. Note that~$S_{j+1}=S_j+2$, given that~$\lfloor m_{j+1}/2 \rfloor=\lfloor m_j/2 \rfloor$ and the~$2$ additive term is present for~$S_{j+1}$, given that~$m_{j+1}$ is odd, but not for~$S_{j}$, given that~$m_j$ is even. Since~$\gamma_{j}=S_{j}$, by hypothesis, we have that~$\gamma_{j+1}=S_{j+1}$. 
\item Suppose finally that~$d(z_{j+1})=2$ and that~$m_{j+1}$ is even. Then~$|\mathcal I^{j+1}_B|=|\mathcal I^{j}_B|+1$, since~$w_{g(j+1)}+1$ belongs to~$\mathcal I^{j+1}_B$ and not to~$\mathcal I^{j}_B$. Since~$m_{j+1}$ is even,~$m_j$ is odd, hence the~$2$ additive term is present for~$S_j$, and thus~$\gamma_j\geq 1$. We can thus charge~$z_{j+1}$ to~$w_{g(j+1)}+1$ and to an uncharged vertex in~$\mathcal I^{j}_B$, and all the required properties remain satisfied. In particular,~$z_1$ was already charged to two vertices in~$\mathcal I^j_B$, given that~$\gamma_j\geq 1$. Also,~$\gamma_{j+1}=\gamma_j-1$, given that~$\mathcal I^{j+1}_B$ contains~$1$ vertex more than~$\mathcal I^{j}_B$ and that~$z_{j+1}$ is charged to two vertices in~$\mathcal I^{j+1}_B$. Also note that~$S_{j+1}=S_j-1$, given that~$\lfloor m_{j+1}/2 \rfloor=\lfloor m_j/2 \rfloor+1$ and the~$2$ additive term is present for~$S_{j}$, given that~$m_{j}$ is odd, but not for~$S_{j+1}$, given that~$m_{j+1}$ is even. Since~$\gamma_{j}=S_{j}$, by hypothesis, we have that~$\gamma_{j+1}=S_{j+1}$. 
\end{itemize}
This completes the proof that the described charging scheme satisfies the required properties. It remains to charge~$z_k$, though. In order to do that, it suffices to prove that at least one of the following three holds true: (a)~$\gamma_{k-1}\geq 2$, (b)~$\gamma_{k-1}=1$ and~$d(z_k)\geq 3$, or (c)~$d(z_k)\geq 4$. Indeed, if (a) holds then~$z_k$ can be charged to two uncharged vertices in~$\mathcal I^{k-1}_B$, if (b) holds then~$z_k$ can be charged to an uncharged vertex in~$\mathcal I^{k-1}_B$ and to~$w_s$, and if (c) holds then~$z_k$ can be charged to~$w_{s-1}$ and to~$w_s$. We distinguish three cases.
\begin{itemize}
    \item Suppose first that~$d(z_k)=2$. If~$d(z_1)\geq 5$, then~$\gamma_{k-1}\geq 2$, as all the other additive terms in the definition of~$S_{k-1}$ are non-negative, and thus (a) holds. We can hence assume that~$d(z_1)\leq 4$. We distinguish three further cases. In each case, we prove that~$\gamma_{k-1}\geq 2$, which implies that (a) holds.
    \begin{itemize}
        \item Suppose that~$d(z_1)=4$, hence~$\gamma_{k-1}$ gets a contribution of~$1$ from the first additive term in the definition of~$S_{k-1}$. If there exists a vertex~$z_i$ with~$i\in \{2,\dots,k-1\}$ and~$d(z_i)\geq 4$, then~$\gamma_{k-1}$ gets a contribution of at least~$1$ from the second additive term in the definition of~$S_{k-1}$; since the third and fourth terms are non-negative, we have~$\gamma_{k-1}\geq 2$. We can thus assume that~$d(z_i)\leq 3$, for~$i=2,\dots,k-1$. We cannot have~$d(z_2)=d(z_3)=\dots=d(z_{k-1})=3$, as otherwise we would be Case~5.2.1(i), thus we have~$m\geq 1$. If~$m$ is odd, then~$\gamma_{k-1}$ gets a contribution of~$2$ from the fourth additive term in the definition of~$S_{k-1}$, hence~$\gamma_{k-1}\geq 3$. If~$m\geq 2$, then~$\gamma_{k-1}$ gets a contribution of at least~$1$ from the third additive term in the definition of~$S_{k-1}$, hence~$\gamma_{k-1}\geq 2$. 
        \item Suppose next that~$d(z_1)=3$, hence the first additive term in the definition of~$S_{k-1}$ is equal to~$0$. If there exists a vertex~$z_i$ with~$i\in \{2,\dots,k-1\}$ and~$d(z_i)\geq 5$, then~$\gamma_{k-1}$ gets a contribution of at least~$2$ from the second additive term in the definition of~$S_{k-1}$; since the third and fourth terms are non-negative, we have~$\gamma_{k-1}\geq 2$. We can thus assume that~$d(z_i)\leq 4$, for~$i=2,\dots,k-1$. If we have two distinct vertices~$z_i$ and~$z_{i'}$ with~$i,i'\in \{2,\dots,k-1\}$ and~$d(z_i)=4$ and~$d(z_{i'})= 4$, then again~$\gamma_{k-1}$ gets a contribution of at least~$2$ from the second additive term in the definition of~$S_{k-1}$ and we have~$\gamma_{k-1}\geq 2$. We cannot have~$d(z_2)=d(z_3)=\dots=d(z_{k-1})=3$, as otherwise we would be in Case~5.2.1(i), and we cannot have~$d(z_i)=4$ and~$d(z_j)=3$, for some~$i\in \{2,\dots,k-1\}$ and for~$j=2,\dots,k-1$ with~$j\neq i$, as otherwise we would be in Case~5.2.1(ii), thus we have~$m\geq 1$. If~$m$ is odd, then~$\gamma_{k-1}$ gets a contribution of~$2$ from the fourth additive term in the definition of~$S_{k-1}$, hence~$\gamma_{k-1}\geq 2$. If~$m\geq 4$, then~$\gamma_{k-1}$ gets a contribution of at least~$2$ from the third additive term in the definition of~$S_{k-1}$, hence~$\gamma_{k-1}\geq 2$. We are left with the case~$m=2$, in which~$\gamma_{k-1}$ gets a contribution of~$1$ from the third additive term in the definition of~$S_{k-1}$ and of~$0$ from the fourth term. We cannot have~$d(z_j)\leq 3$ for~$j=2,\dots,k-1$, as otherwise we would be in Case~5.2.1(v), thus we have~$d(z_i)=4$, for some~$i\in \{2,\dots,k-1\}$, hence~$\gamma_{k-1}$ gets a contribution of~$1$ from the second additive term in the definition of~$S_{k-1}$, thus~$\gamma_{k-1}\geq 2$.
        \item Suppose next that~$d(z_1)=2$, hence the first additive term in the definition of~$S_{k-1}$ is equal to~$-1$. 
        \begin{itemize}
        \item If there exists a vertex~$z_i$ with~$i\in \{2,\dots,k-1\}$ and~$d(z_i)\geq 6$, then~$\gamma_{k-1}$ gets a contribution of at least~$3$ from the second additive term in the definition of~$S_{k-1}$; since the third and fourth terms are non-negative, we have~$\gamma_{k-1}\geq 2$. We can thus assume that~$d(z_i)\leq 5$, for~$i=2,\dots,k-1$. 
        \item Suppose that there exists a vertex~$z_i$ with~$i\in \{2,\dots,k-1\}$ and~$d(z_i)=5$, hence~$\gamma_{k-1}$ gets a contribution of at least~$2$ from the second additive term in the definition of~$S_{k-1}$. If~$m\geq 1$, then~$\gamma_{k-1}$ gets a contribution of at least~$1$ from the third or fourth additive term in the definition of~$S_{k-1}$, hence~$\gamma_{k-1}\geq 2$. We cannot have~$d(z_j)=3$, for~$j=2,\dots,k-1$ with~$j\neq i$, as otherwise we would be in Case~5.2.1(iii). If there exists a vertex~$z_{i'}$ with~$i'\in \{2,\dots,k-1\}$,~$i'\neq i$, and~$d(z_{i'})\geq 4$, then~$\gamma_{k-1}$ gets a contribution of at least~$3$ from the second additive term in the definition of~$S_{k-1}$, thus we have~$\gamma_{k-1}\geq 2$. We can thus assume that~$d(z_i)\leq 4$, for~$i=2,\dots,k-1$. 
        \item Suppose that there exists a vertex~$z_i$ with~$i\in \{2,\dots,k-1\}$ and~$d(z_i)=4$, hence~$\gamma_{k-1}$ gets a contribution of at least~$1$ from the second additive term in the definition of~$S_{k-1}$. If~$m$ is odd, then~$\gamma_{k-1}$ gets a contribution of~$2$ from the fourth additive term in the definition of~$S_{k-1}$, hence~$\gamma_{k-1}\geq 2$. If~$m\geq 4$, then~$\gamma_{k-1}$ gets a contribution of at least~$2$ from the third additive term in the definition of~$S_{k-1}$, hence~$\gamma_{k-1}\geq 2$. We cannot have~$m=2$ and~$d(z_j)\leq 3$, for~$j=2,\dots,k-1$ with~$j\neq i$, as otherwise we would be in Case~5.2.1(vii). If~$m=2$ and~$d(z_{i'})=4$, for some~$i'\in \{2,\dots,k-1\}$ with~$i'\neq i$, then~$\gamma_{k-1}$ gets a contribution of at least~$2$ from the second additive term in the definition of~$S_{k-1}$ and a contribution of~$1$ from the third additive term in the definition of~$S_{k-1}$, hence~$\gamma_{k-1}\geq 2$. This leaves us with the case~$m=0$. We cannot have~$d(z_j)=3$, for~$j=2,\dots,k-1$ with~$j\neq i$, as otherwise we would be in Case~5.2.1(ii). Also, we cannot have~$d(z_{i'})=4$ and~$d(z_j)=3$, for some~$i'\in \{2,\dots,k-1\}$ with~$i'\neq i$ and for~$j=2,\dots,k-1$ with~$j\neq i$ and~$j\neq i'$, as otherwise we would be in Case~5.2.1(iv). It follows that we have~$d(z_{i'})=4$ and~$d(z_{i''})=4$, for some~$i',i''\in \{2,\dots,k-1\}$ with~$i'\neq i$,~$i''\neq i$ and~$i''\neq i'$. Hence,~$\gamma_{k-1}$ gets a contribution of at least~$3$ from the second additive term in the definition of~$S_{k-1}$, thus we have~$\gamma_{k-1}\geq 2$. We can thus assume that~$d(z_i)\leq 3$, for~$i=2,\dots,k-1$. 
        \item We cannot have~$m=0$, as otherwise we would be in Case~5.1 (note that we would have~$k\geq 3$, since~$G$ is internally-triangulated). Also, we cannot have~$m=2$ or~$m=4$, as otherwise we would be in Case~5.2.1(v). Furthermore, we cannot have~$m=1$. Indeed, if~$k=3$, we would be in Case~5.1, if~$k\geq 4$ and~$d(z_{k-1})=3$ we would be in Case~5.2.1(vi), and if~$k\geq 4$ and~$d(z_2)=3$ we would be in Case~5.2.2(vi); note that~$d(z_{k-1})=2$ and~$d(z_2)=2$ cannot hold simultaneously, given that~$m=1$. If~$m\geq 3$ and~$m$ is odd, then~$\gamma_{k-1}$ gets a contribution of~$2$ from the fourth additive term in the definition of~$S_{k-1}$, and of at least~$1$ from the third additive term in the definition of~$S_{k-1}$, hence~$\gamma_{k-1}\geq 2$. Finally, if~$m\geq 6$, then~$\gamma_{k-1}$ gets a contribution of at least~$3$ from the third additive term in the definition of~$S_{k-1}$, hence~$\gamma_{k-1}\geq 1$. 
        \end{itemize}
    \end{itemize}
 
   \item Suppose next that~$d(z_k)=3$. If~$d(z_1)\geq 4$, then~$\gamma_{k-1}\geq 1$, as all the other additive terms in the definition of~$S_{k-1}$ are non-negative, and thus (b) holds. We can hence assume that~$d(z_1)\leq 3$. We distinguish two further cases. In each case, we prove that~$\gamma_{k-1}\geq 1$, which implies that (b) holds.
   \begin{itemize}
    \item Suppose first that~$d(z_1)=3$, hence the first additive term in the definition of~$S_{k-1}$ is equal to~$0$. If there exists a vertex~$z_i$ with~$i\in \{2,\dots,k-1\}$ and~$d(z_i)\geq 4$, then~$\gamma_{k-1}$ gets a contribution of at least~$1$ from the second additive term in the definition of~$S_{k-1}$; since the third and fourth terms are non-negative, we have~$\gamma_{k-1}\geq 1$. We can thus assume that~$d(z_i)\leq 3$, for~$i=2,\dots,k-1$. We cannot have~$d(z_2)=d(z_3)=\dots=d(z_{k-1})=3$, as otherwise~$B$ would be a pesky component, thus we have~$d(z_i)=2$, for some~$i\in \{2,\dots,k-1\}$. However, this implies that~$\gamma_{k-1}$ gets a contribution of at least~$1$ from the third or fourth additive term in the definition of~$S_{k-1}$, thus we have~$\gamma_{k-1}\geq 1$.
   \item  Suppose finally that~$d(z_1)=2$, hence the first additive term in the definition of~$S_{k-1}$ is equal to~$-1$. 
    \begin{itemize}
    \item If there exists a vertex~$z_i$ with~$i\in \{2,\dots,k-1\}$ and~$d(z_i)\geq 5$, then~$\gamma_{k-1}$ gets a contribution of at least~$2$ from the second additive term in the definition of~$S_{k-1}$; since the third and fourth terms are non-negative, we have~$\gamma_{k-1}\geq 1$. We can thus assume that~$d(z_i)\leq 4$, for~$i=2,\dots,k-1$. 
    \item Suppose that there exists a vertex~$z_i$ with~$i\in \{2,\dots,k-1\}$ and~$d(z_i)=4$, hence~$\gamma_{k-1}$ gets a contribution of at least~$1$ from the second additive term in the definition of~$S_{k-1}$. If~$m\geq 1$, then~$\gamma_{k-1}$ gets a contribution of at least~$1$ from the third or fourth additive term in the definition of~$S_{k-1}$, hence~$\gamma_{k-1}\geq 1$. We cannot have~$d(z_j)=3$, for~$j=2,\dots,k-1$ with~$j\neq i$, as otherwise we would be in Case~5.2.2(ii). If there exists a vertex~$z_{i'}$ with~$i'\in \{2,\dots,k-1\}$,~$i'\neq i$, and~$d(z_{i'})\geq 4$, then~$\gamma_{k-1}$ gets a contribution of at least~$2$ from the second additive term in the definition of~$S_{k-1}$, thus we have~$\gamma_{k-1}\geq 1$. We can thus assume that~$d(z_i)\leq 3$, for~$i=2,\dots,k-1$. 
    \item We cannot have~$m=0$, as otherwise we would be in Case~5.2.2(i). Also, we cannot have~$m=2$, as otherwise we would be in Case~5.2.2(v). If~$m$ is odd, then~$\gamma_{k-1}$ gets a contribution of~$2$ from the fourth additive term in the definition of~$S_{k-1}$, hence~$\gamma_{k-1}\geq 1$. Finally, if~$m\geq 4$, then~$\gamma_{k-1}$ gets a contribution of at least~$2$ from the third additive term in the definition of~$S_{k-1}$, hence~$\gamma_{k-1}\geq 2$. 
    \end{itemize}
   \end{itemize} 
    \item Suppose finally that~$d(z_k)\geq 4$. Then (c) holds.
\end{itemize}

This concludes the proof that~$\mathcal I_B$ satisfies Property~A and thus the proof of the lemma.
\end{proof}

We can now apply induction as follows. Let~$\mathcal I_B$ be a set as in~\cref{le:cushy}. If~$\mathcal I_B$ satisfies Property~A, let~$H$ be the subgraph of~$G$ induced by~$V(G)-(V(G_B)-\{c_B\})$; otherwise, if it satisfies Property~B, let~$H$ be the subgraph of~$G$ induced by~$V(G)-(V(G_B)-\{c_B,\ell_B\})$; otherwise, it satisfies Property~C and then let~$H$ be the subgraph of~$G$ induced by~$V(G)-(V(G_B)-\{c_B,r_B\})$. We apply induction on~$H$, so to find a good set~$\mathcal I_H$ for~$H$. We define~$\mathcal I=\mathcal I_H\cup \mathcal I_B$. 
We prove that~$\mathcal I$ is a good set. Note that~$|\mathcal I|=|\mathcal I_H|+|\mathcal I_B|$. If~$\mathcal I_B$ satisfies Property~A, then~$|V(H)|=n-n_B+1$, hence, by induction, we have~$|\mathcal I_H|\geq 2(n-n_B+1)/3$ and, by~\cref{le:cushy}, we have~$|\mathcal I_B|\geq 2(n_B-1)/3$, thus~$|\mathcal I|\geq 2n/3$. If~$\mathcal I_B$ satisfies Property~B or Property~C, then~$|V(H)|=n-n_B+2$, hence, by induction, we have~$|\mathcal I_H|\geq 2(n-n_B+2)/3$ and, by~\cref{le:cushy}, we have~$|\mathcal I_B|\geq 2(n_B-2)/3$, thus~$|\mathcal I|\geq 2n/3$.   
In order to prove that~$G[\mathcal I]$ is outerplane, note that~$G[\mathcal I_H]$ is outerplane by induction. If~$\mathcal I_B$ satisfies Property~A, then~$G[\mathcal I_B \cup (\{c_B\}\cap \mathcal I_H)]$ is outerplane by~\cref{le:cushy}. Since~$G[\mathcal I_H]$ and~$G[\mathcal I_B \cup (\{c_B\}\cap \mathcal I_H)]$ are each in the outer face of the other one and share either no vertex or one vertex, by~\cref{obs:union} we have that~$G[\mathcal I]$ is an outerplane graph. If~$\mathcal I_B$ satisfies Property~B, then~$G[\mathcal I_B \cup (\{c_B,\ell_B\}\cap \mathcal I_H)]$ is outerplane by~\cref{le:cushy}. Since~$G[\mathcal I_H]$ and~$G[\mathcal I_B \cup (\{c_B,\ell_B\}\cap \mathcal I_H)]$ are each in the outer face of the other one and share either no vertex, or a single vertex, or a single edge, by~\cref{obs:union} we have that~$G[\mathcal I]$ is an outerplane graph. Similarly, if~$\mathcal I_B$ satisfies Property~C, then~$G[\mathcal I_B \cup (\{c_B,r_B\}\cap \mathcal I_H)]$ is outerplane by~\cref{le:cushy}. Since~$G[\mathcal I_H]$ and~$G[\mathcal I_B \cup (\{c_B,r_B\}\cap \mathcal I_H)]$ are each in the outer face of the other one and share either no vertex, or a single vertex, or a single edge, by~\cref{obs:union} we have that~$G[\mathcal I]$ is an outerplane graph. 

As long as~$B$ is a non-trivial block, we can now assume that~$B$ is pesky.

\begin{figure}[tb]
\centering
    \begin{subfigure}{0.48\textwidth}
		\centering
		\includegraphics[scale=1.2, page=7]{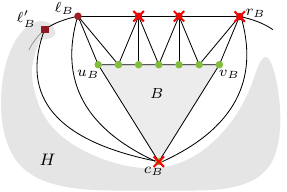}
		\subcaption{}
        \label{fig:algorithm-case6}
	\end{subfigure}
    \hfill
    \begin{subfigure}{0.48\textwidth}
		\centering
		\includegraphics[scale=1.2, page=8]{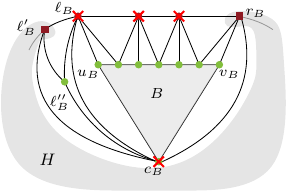}
		\subcaption{}
        \label{fig:algorithm-case7}
	\end{subfigure}

\caption{Illustrations for (a) Case~6 and (b) Case~7.}
\label{fig:algorithm-cases6-7}
\end{figure}

\paragraph*{Case~6:~$B$ is non-trivial and pesky, the edge~$c_B \ell'_B$ exists, and the~cycle~$c_B \ell_B \ell'_B$ bounds an internal face of~$G$}

Refer to \cref{fig:algorithm-case6}. Let~$H$ be the subgraph of~$G$ induced by~$V(G)-V(G_B)$. We apply induction on~$H$, so to find a good set~$\mathcal I_H$ for~$H$. Let~$\mathcal I_B = \{\ell_B\} \cup (V(B)-\{c_B\})$. We define~$\mathcal I=\mathcal I_H\cup \mathcal I_B$. Note that~$c_B$ and~$r_B$ do not belong to~$\mathcal I$, that~$\ell_B$ belongs to~$\mathcal I$, and that~$\ell'_B$ might or might not belong to~$\mathcal I$. 
We prove that~$\mathcal I$ is a good set. First,  we have~$|\mathcal I|=|\mathcal I_H|+|V(B)|$. By induction,~$|\mathcal I_H|\geq 2|V(H)|/3$; also, by~\cref{obs:pesky},~$|V(B)|=2|V(G_B)|/3$. Hence,~$|\mathcal I|\geq 2(|V(H)|+|V(G_B)|)/3=2n/3$. Also,~$G[\mathcal I_H]$ is outerplane by induction and~$G[\mathcal I_B \cup (\{\ell'_B\}\cap \mathcal I_H)]$ is outerplane, as well, given that~$B$ is an outerplane graph, given that~$\ell_B$ is outside~$B$, where~$\ell_B u_B u'_B$ bounds a face of~$G$ and~$\ell_B$ is not adjacent to any vertex of~$\mathcal I_B$ different from~$u_B$ and~$u'_B$, and given that~$\ell'_B$ is outside~$G[\mathcal I_B]$ and not adjacent to any vertex of~$\mathcal I_B$ different from~$\ell_B$. Since~$G[\mathcal I_H]$ and~$G[\mathcal I_B \cup (\{\ell'_B\}\cap \mathcal I_H)]$ are each in the outer face of the other one, and share either no vertex or a single vertex, by~\cref{obs:union}, we have that~$G[\mathcal I]$ is an outerplane graph.

\paragraph*{Case~6':~$B$ is non-trivial and pesky, the edge~$c_B r'_B$ exists, and the~cycle~$c_B r_B r'_B$ bounds an internal face of~$G$}

This case can be discussed symmetrically to Case~6. 

\paragraph*{Case~7:~$B$ is non-trivial and pesky, the edge~$c_B \ell'_B$ exists, and the~cycle~$c_B \ell_B \ell'_B$ contains in its interior a unique vertex}

Refer to \cref{fig:algorithm-case7}.
Let~$\ell''_B$ be the vertex contained in the cycle~$c_B \ell_B \ell'_B$. Let~$H$ be the subgraph of~$G$ induced by~$V(G)-(V(G_B)-\{r_B\} \cup \{\ell''_B\})$. We apply induction on~$H$, so to find a good set~$\mathcal I_H$ for~$H$. Let~$\mathcal I_B = \{\ell''_B\} \cup (V(B)-\{c_B\})$. We define~$\mathcal I=\mathcal I_H\cup \mathcal I_B$. Note that~$c_B$ and~$\ell_B$ do not belong to~$\mathcal I$, while each of~$\ell'_B$ and~$r_B$  might or might not belong to~$\mathcal I$. We prove that~$\mathcal I$ is a good set. First,  we have~$|\mathcal I|=|\mathcal I_H|+|V(B)|$, hence~$|\mathcal I|\geq 2n/3$ follows as in Case~6. Also,~$G[\mathcal I_H]$ is outerplane by induction and~$G[V(B)-\{c_B\} \cup (\{r_B\}\cap \mathcal I_H)]$ is outerplane, as well, given that~$B$ is an outerplane graph and given that~$r_B$ is outside~$B$, where~$r_B v_B v'_B$ bounds a face of~$G$ and~$r_B$ is not adjacent to any vertex of~$V(B)-\{c_B\}$ different from~$v_B$ and~$v'_B$. Since~$G[\mathcal I_H]$ and~$G[V(B)-\{c_B\} \cup (\{r_B\}\cap \mathcal I_H)]$ are each in the outer face of the other one, and share either no vertex or a single vertex, by~\cref{obs:union}, we have that~$G[\mathcal I - \{\ell''_B\}]$ is an outerplane graph. Furthermore,~$\ell'_B$ and~$\ell''_B$ induce an edge, hence an outerplane graph. Since~$G[\mathcal I - \{\ell''_B\}]$ and~$G[\{\ell''_B\} \cup (\mathcal I_H\cap \{\ell'_B\})]$ are each in the outer face of the other one and share either no vertex or a single vertex, by~\cref{obs:union}, we have that~$G[\mathcal I]$ is an outerplane graph.



\paragraph*{Case~7':~$B$ is non-trivial and pesky, the edge~$c_B r'_B$ exists, and the~cycle~$c_B r_B r'_B$ contains in its interior a unique vertex~$r''_B$} 

This case can be discussed symmetrically to Case~7. 

\begin{figure}[tb]
\centering
    \begin{subfigure}{0.48\textwidth}
		\centering
		\includegraphics[scale=1.1, page=9]{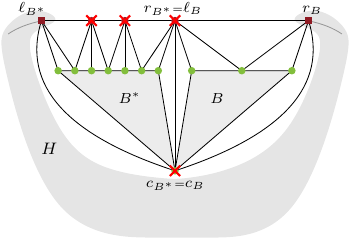}
		\subcaption{}
        \label{fig:algorithm-case8}
	\end{subfigure}
    \hfill
    \begin{subfigure}{0.48\textwidth}
		\centering
		\includegraphics[scale=1.1, page=10]{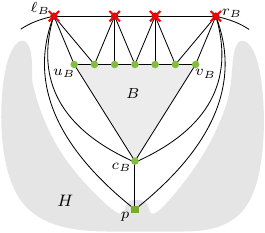}
		\subcaption{}
        \label{fig:algorithm-case11}
	\end{subfigure}

\caption{Illustrations for (a) Case~8 and (b) Case~11.}
\label{fig:algorithm-cases8-11}
\end{figure}

\paragraph*{Case~8:~$B$ is non-trivial and pesky, and there exists a non-trivial pesky extremal leaf~$B^*\neq B$ in~$K$ such that the link-vertex~$c_{B^*}$ of~$B^*$ is~$c_B$ and the right cage vertex~$r_{B^*}$ of~$B^*$ is~$\ell_B$}

Refer to \cref{fig:algorithm-case8}. Let~$H$ be the subgraph of~$G$ induced by~$V(G)-(V(G_B)-\{r_B\})-(V(G_{B^*})-\{\ell_{B^*}\})$. We apply induction on~$H$, so to find a good set~$\mathcal I_H$ for~$H$. Let~$\mathcal I_B = (V(B)\cup V(B^*))-\{c_B\}$. We define~$\mathcal I=\mathcal I_H\cup \mathcal I_B$. Note that~$c_B=c_{B^*}$ and~$r_{B^*}=\ell_B$ do not belong to~$\mathcal I$, while each of~$\ell_{B^*}$ and~$r_B$ might or might not belong to~$\mathcal I$. 
We prove that~$\mathcal I$ is a good set. First,  we have~$|\mathcal I|=|\mathcal I_H|+|V(B)|$. By induction,~$|\mathcal I_H|\geq 2|V(H)|/3$; also, by~\cref{obs:pesky},~$|V(B)|=2|V(G_B)|/3$ and~$|V(B^*)|=2|V(G_{B^*})|/3$. The number of vertices of 
$H$ is~$n-(|V(G_B)| + |V(G_{B^*})|-4)$, since the vertices of~$G$ that do not belong to~$H$ are those of~$G_B$ and~$G_{B^*}$, except for~$\ell_{B^*}$ and~$r_B$, and since~$G_B$ and~$G_{B^*}$ share the vertices~$r_{B^*}=\ell_B$ and~$c_B=c_{B^*}$. Also note that~$|\mathcal I_B|=|V(B)|+|V(B^*)|-2$, as~$c_B$ belongs both to~$B$ and~$B^*$ and is not part of~$\mathcal I_B$. It follows that~$|\mathcal I|\geq 2(n-|V(G_B)|-|V(G_{B^*})|+4)/3+|V(B)|+|V(B^*)|-2=2(n-|V(G_B)|-|V(G_{B^*})|)/3+2(|V(G_B)|+|V(G_{B^*})|)/3+8/3-2=2n/3+2/3>2n/3$. Also,~$G[\mathcal I_H]$ is outerplane by induction and~$G[(V(B)-\{c_B\}) \cup (\{r_B\} \cap \mathcal I_H)]$ is outerplane, as well, given that~$r_B$ is outside~$B$, that~$r_B v_B v'_B$ bounds a face of~$G$, and that~$r_B$ is not adjacent to any vertex of~$V(B)-\{c_B\}$ different from~$v_B$ and~$v'_B$. Since~$G[\mathcal I_H]$ and~$G[(V(B)-\{c_B\}) \cup (\{r_B\} \cap \mathcal I_H)]$ are each in the outer face of the other one, and share either no vertex or a single vertex, by~\cref{obs:union}, we have that~$G[\mathcal I_H \cup (V(B)-\{c_B\})]$ is an outerplane graph. Similarly,~$G[(V(B^*)-\{c_B\})\cup (\{\ell_{B^*}\} \cap \mathcal I_H)]$ is outerplane, and since~$G[\mathcal I_H \cup (V(B)-\{c_B\})]$ and~$G[(V(B^*)-\{c_B\})\cup (\{\ell_{B^*}\} \cap \mathcal I_H)]$ are each in the outer face of the other one, and share either no vertex or a single vertex, by~\cref{obs:union}, we have that~$G[\mathcal I]$ is an outerplane graph.


\paragraph*{Case~8':~$B$ is non-trivial and pesky, and there exists a non-trivial pesky extremal leaf~$B^*\neq B$ in~$K$ such that the link-vertex~$c_{B^*}$ of~$B^*$ is~$c_B$ and the left cage vertex~$\ell_{B^*}$ of~$B^*$ is~$r_B$} 

This case can be discussed symmetrically to Case~8. 

{\em Suppose next that~$B$ is a trivial block}; recall that, in this case,~$B$ is an edge~$c_Bd_B$, where~$c_B$ is a cutvertex of~$K$. 

\begin{figure}[tb]
\centering
    \begin{subfigure}{0.32\textwidth}
		\centering
		\includegraphics[scale=1.2, page=11]{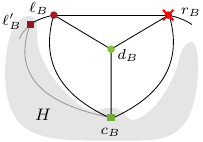}
		\subcaption{}
        \label{fig:algorithm-case9}
	\end{subfigure}
    \hfill
    \begin{subfigure}{0.32\textwidth}
		\centering
		\includegraphics[scale=1.2, page=12]{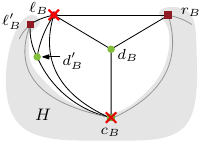}
		\subcaption{}
        \label{fig:algorithm-case10}
	\end{subfigure}
    \hfill
    \begin{subfigure}{0.32\textwidth}
		\centering
		\includegraphics[scale=1.2, page=13]{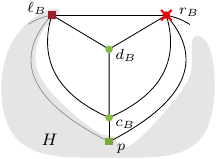}
		\subcaption{}
        \label{fig:algorithm-case12}
	\end{subfigure}

\caption{Illustrations for (a) Case~9, (b) Case~10, and (c) Case~12.}
\label{fig:algorithm-cases9-10-12}
\end{figure}

\paragraph*{Case~9:~$B$ is trivial, the edge~$c_B \ell'_B$ exists, and the cycle~$c_B \ell_B \ell'_B$ bounds an internal face of~$G$}

Refer to \cref{fig:algorithm-case9}. Let~$H$ be the subgraph of~$G$ induced by~$V(G)-\{\ell_B,r_B,d_B\}$. We apply induction on~$H$, so to find a good set~$\mathcal I_H$ for~$H$. We define~$\mathcal I=\mathcal I_H\cup \{\ell_B,d_B\}$. Note that~$r_B$ does not belong to~$\mathcal I$, while~$c_B$ and~$\ell'_B$ might or might not belong to~$\mathcal I$. 
We prove that~$\mathcal I$ is a good set. First,  we have~$|\mathcal I|=|\mathcal I_H|+2$.  Note that~$|V(H)|=n-3$, hence, by induction, we have~$|\mathcal I_H|\geq 2(n-3)/3$. It follows that~$|\mathcal I|\geq 2(n-3)/3+2=2n/3$. Also,~$G[\mathcal I_H]$ is outerplane by induction and the subgraph of~$G$ induced by~$c_B$,~$\ell'_B$,~$\ell_B$, and~$d_B$ is outerplane, as well, since it is composed of the cycles~$c_B \ell_B \ell'_B$ and~$c_B d_B \ell_B$, with~$d_B$ outside the former cycle. Since~$G[\mathcal I_H]$ and~$G[\{c_B,\ell'_B,\ell_B,d_B\} \cap \mathcal I_H]$ are each in the outer face of the other one, and share either no vertex, or a single vertex, or a single edge, by~\cref{obs:union}, we have that~$G[\mathcal I]$ is an outerplane graph. 

\paragraph*{Case~9':~$B$ is trivial, the edge~$c_B r'_B$ exists and the~cycle~$c_B r_B r'_B$ bounds an internal face of~$G$}

This case can be discussed symmetrically to Case~9. 

\paragraph*{Case~10:~$B$ is trivial, the edge~$c_B \ell'_B$ exists and the~cycle~$c_B \ell_B \ell'_B$ contains in its interior a unique vertex}

Refer to \cref{fig:algorithm-case10}.
Let~$d'_B$ be the vertex contained in the cycle~$c_B \ell_B \ell'_B$. Let~$H$ be the subgraph of~$G$ induced by~$V(G)-\{\ell_B,d_B,d'_B\}$. We apply induction on~$H$, so to find a good set~$\mathcal I_H$ for~$H$. We define~$\mathcal I=\mathcal I_H\cup \{d_B,d'_B\}$. Note that~$\ell_B$ does not belong to~$\mathcal I$, while~$c_B$,~$r_B$, and~$\ell'_B$ might or might not belong to~$\mathcal I$. 
We prove that~$\mathcal I$ is a good set. First,  we have~$|\mathcal I|=|\mathcal I_H|+2$ and~$|V(H)|=n-3$, hence~$|\mathcal I|\geq 2n/3$ follows as in Case~6. Also,~$G[\mathcal I_H]$ is outerplane by induction, and~$G[\{c_B,\ell'_B,d'_B\}]$ and~$G[\{c_B,r_B,d_B\}]$ are outerplane graphs, as well, since they are cycles. Since~$G[\mathcal I_H]$ and~$G[\{c_B,\ell'_B,d'_B\} \cap \mathcal I_H]$ are each in the outer face of the other one, and share either no vertex, or a single vertex, or a single edge, by~\cref{obs:union}, we have that~$G[\mathcal I_H \cup \{d'_B\}]$ is an outerplane graph. Also, since~$G[\mathcal I_H \cup \{d'_B\}]$ and~$G[\{c_B,r_B,d_B\} \cap \mathcal I_H]$ are each in the outer face of the other one, and share either no vertex, or a single vertex, or a single edge, by~\cref{obs:union}, we have that~$G[\mathcal I]$ is an outerplane graph.

\paragraph*{Case~10':~$B$ is trivial, the edge~$c_B r'_B$ exists, and the~cycle~$c_B r_B r'_B$ contains in its interior a unique vertex}

This case can be discussed symmetrically to Case~10.

By~\cref{le:case-distinction-blocks}, we can now assume that~$B$ is the only child of~$c_B$ in~$T_K$. The next two cases deal with the situation in which the parent~$B_P$ of~$c_B$ in~$T_K$ is a trivial block and~$B$ is non-trivial or trivial, respectively. 

\paragraph*{Case~11:~$B$ is non-trivial and pesky, and the parent~$B_P$ of~$c_B$ in~$T_K$ is a trivial block}

Refer to \cref{fig:algorithm-case11}.
Let~$pc_B$ be the edge corresponding to~$B_P$. By~\cref{le:case11}, the edges~$p \ell_B$ and~$p r_B$ belong to~$G$ and the cycles~$p \ell_B c_B$ and~$p r_B c_B$ delimit faces of~$G$. Let~$H$ be the subgraph of~$G$ induced by~$V(G)-V(G_B)$. We apply induction on~$H$, so to find a good set~$\mathcal I_H$ for~$H$. Let~$\mathcal I_B = V(B)$. We define~$\mathcal I=\mathcal I_H\cup \mathcal I_B$. Note that~$\ell_B$ and~$r_B$ do not belong to~$\mathcal I$, that~$c_B$ belongs to~$\mathcal I$, and that~$p$ might or might not belong to~$\mathcal I$. 
We prove that~$\mathcal I$ is a good set. First,  we have~$|\mathcal I|=|\mathcal I_H|+|V(B)|$, hence~$|\mathcal I|\geq 2n/3$ follows as in Case~6.  Also,~$G[\mathcal I_H]$ is outerplane by induction and~$G[\mathcal I_B \cup (\{p\}\cap \mathcal I_H)]$ is outerplane, as well, given that~$B$ is an outerplane graph and given that the edge~$pc_B$ is outside~$B$. Since~$G[\mathcal I_H]$ and~$G[\mathcal I_B \cup (\{p\}\cap \mathcal I_H)]$ are each in the outer face of the other one, and share either no vertex or a single vertex, by~\cref{obs:union}, we have that~$G[\mathcal I]$ is an outerplane graph.


\paragraph*{Case~12:~$B$ is trivial and the parent~$B_P$ of~$c_B$ in~$T_K$ is a trivial block}

Refer to \cref{fig:algorithm-case12}.
Let~$pc_B$ be the edge corresponding to the parent of~$c_B$ in~$T_K$. By~\cref{le:case11}, the edges~$p \ell_B$ and~$p r_B$ belong to~$G$ and the cycles~$p \ell_B c_B$ and~$p r_B c_B$ delimit faces of~$G$.  Let~$H$ be the subgraph of~$G$ induced by~$V(G)-\{c_B,d_B,r_B\}$. We apply induction on~$H$, so to find a good set~$\mathcal I_H$ for~$H$. We define~$\mathcal I=\mathcal I_H\cup \{c_B,d_B\}$. Note that~$r_B$ does not belong to~$\mathcal I$, while~$p$ and~$\ell_B$ might or might not belong to~$\mathcal I$. 
We prove that~$\mathcal I$ is a good set. First,  we have~$|\mathcal I|=|\mathcal I_H|+2$ and~$|V(H)|=n-3$, hence~$|\mathcal I|\geq 2n/3$ follows as in Case~6. Also,~$G[\mathcal I_H]$ is outerplane by induction and~$G[\{p,c_B,d_B,\ell_B\}]$ is outerplane, as well, since it is composed of the cycles~$p c_B \ell_B$ and~$c_B d_B \ell_B$, with~$d_B$ outside the former cycle. Since~$G[\mathcal I_H]$ and~$G[\{p,c_B,\ell_B,d_B\} \cap \mathcal I_H]$ are each in the outer face of the other one, and share either no vertex, or a single vertex, or a single edge, by~\cref{obs:union}, we have that~$G[\mathcal I]$ is an outerplane graph. 

\paragraph*{Case~13: Neither of Cases 1--12 applies}

Let~$K$ be a terminal component of~$G[L_2]$ and~$B_P$ be a biconnected component of~$K$ as in~\cref{le:case13}. If~$B_P$ is not the root of~$T_K$, then let~$c_P$ be the parent of~$B_P$ in~$T_K$. Otherwise, let~$l$ be the leaf of~$T^*$ corresponding to the face of~$G[L_1]$ in which~$K$ lies, and let~$xy$ be the edge of~$G[L_1]$ dual to the edge of~$T^*$ from~$l$ to its parent (or let~$e^*=xy$ if~$l$ is the root of~$T^*$). Then~$c_P$ is the vertex of~$K$ such that the~$3$-cycle~$xyc_P$ bounds an internal face of~$G$. 

Let~$u_P$ and~$v_P$ be the neighbors of~$c_P$ on the boundary of the outer face of~$B_P$, where~$u_P$,~$c_P$, and~$v_P$ appear in this counter-clockwise order along the outer face of~$B_P$. Let~$c_P\ell_P$ be the edge of~$G$ that precedes~$c_Pu_P$ in clockwise direction around~$c_P$ and note that~$\ell_P\in L_1$. Symmetrically, let~$c_Pr_P$ be the edge of~$G$ that follows~$c_Pv_P$ in clockwise direction around~$c_P$ and note that~$r_P\in L_1$. Note that, since~$G$ is internally-triangulated, the edges~$\ell_Pu_P$ and~$r_Pv_P$ exist and the cycles~$c_P\ell_Pu_P$ and~$c_Pr_Pv_P$ bound internal faces of~$G$. 

Let~$G_P$ be the subgraph of~$G$ whose vertices and edges are inside or on the boundary of the cycle composed of the path~$\ell_Pc_Pr_P$ and of the path obtained by traversing in clockwise direction the outer face of~$G$ from~$\ell_P$ to~$r_P$. Let~$n_P=|V(G_P)|$. Although the setting is different, the general strategy we employ to handle this final case is similar to the one of Cases 4 and 5.2 and supported by the following lemma.

\begin{lemma} \label{le:final-case}
There exists a set~$\mathcal I_P\subset V(G_P)$ such that~$\{c_P,\ell_P,r_P\} \cap \mathcal I_P=\emptyset$ and such that (at least) one of the following properties is satisfied:
\begin{itemize}
    \item {\em Property~A:}~$|\mathcal I_P|\geq 2(n_P-1)/3$ and~$G[\mathcal I_P\cup\{c_P\}]$ is an outerplane graph;
    \item {\em Property~B:}~$|\mathcal I_P|\geq 2(n_P-2)/3$ and~$G[\mathcal I_P\cup\{c_P,\ell_P\}]$ is an outerplane graph; and
    \item {\em Property~C:}~$|\mathcal I_P|\geq 2(n_P-2)/3$ and~$G[\mathcal I_P\cup\{c_P,r_P\}]$ is an outerplane graph.
\end{itemize}
\end{lemma}

\begin{proof}
Refer to \cref{fig:algorithm-case13-general}. Let~$c_P,w_1=u_P,w_2,\dots,w_{s-1},w_s=v_P$ be the clockwise order of the vertices along the outer face of~$B_P$, and let~$c_P,z_1=\ell_P,z_2,\dots,z_{k-1},z_k=r_P$ be the clockwise order of the vertices along the outer face of~$G_P$. Let~$K_P$ be the subgraph of~$K$ composed of~$B_P$ and of all the blocks that are descendants of~$B_P$; that is,~$K_P$ is obtained from~$G_P$ by removing the vertices~$z_1,z_2,\dots,z_k$ and their incident edges. Note that every vertex among~$z_1,z_2,\dots,z_k$ has at least two neighbors in~$K_P$. Indeed,~$z_1$ is adjacent to~$c_P$ and~$u_P$,~$z_k$ is adjacent to~$c_P$ and~$v_P$, and, for any~$i\in \{2,\dots,k-1\}$, vertex~$z_i$ has at least two neighbors in~$K_P$, as otherwise Case~1 or Case~2 would apply. 

\begin{figure}[tbh]
\centering
\includegraphics[scale=1.2, page=26]{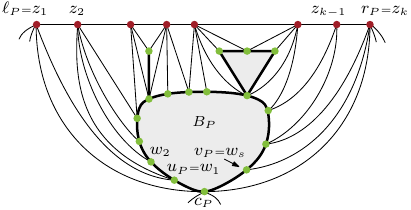}
\caption{Illustration for Case~13. The figure depicts~$G_P$, plus part of edges incident to~$\ell_P$,~$c_P$, and~$r_P$ outside~$G_P$. The boundary of~$K_P$ is represented by thick lines.}
\label{fig:algorithm-case13-general}
\end{figure}

We say that a vertex~$z_j$ is \emph{transparent} if either: (i)~$z_j$ has exactly two neighbors in~$K_P$; or (ii)~$z_j$ has exactly three neighbors in~$K_P$, one of which is a vertex~$d_B$ that has degree~$1$ in~$K_P$. Note that, if (i) holds true, as for~$z_1$ and~$z_{k-1}$ in \cref{fig:algorithm-case13-general}, then the two neighbors of~$z_j$ in~$K_P$ belong to~$B_P$. Indeed, every non-trivial block~$B$ that is a descendant of~$B_P$ in~$T_K$ is pesky, hence if~$z_j$ has a neighbor in~$B$ it actually has~three neighbors in~$B$. Also, if~$z_j$ has a neighbor~$d_B$ that does not belong to~$B_P$ and belongs to a trivial block~$B$, then~$z_j=\ell_B$ or~$z_j=r_B$, hence~$z_j$ is also a neighbor of the parent~$c_B$ of~$B$ in~$T_K$ and of~$c^l_B$ or~$c^r_B$, respectively, by~\cref{le:case13}. If (ii) holds true, as for~$z_3$ in \cref{fig:algorithm-case13-general}, then~$d_B$ is a vertex of a trivial block~$B$ that is child of a child~$c_B$ of~$B_P$, with~$d_B\neq c_B$, and~$z_j$ is neighbor of~$d_B$,~$c_B$, and of either~$c^l_B$ or~$c^r_B$. We say that~$z_j$ is \emph{opaque} if it is not transparent.  


Observe that~$k\geq 2$. Namely, by~\cref{le:case13}, we have that~$B_P$ has at least one child~$c_B$ in~$T_K$, and~$c_B$ has one child~$B$; then, if~$k=1$, the edges~$c_B\ell_B$ and~$c_Br_B$, which exist since~$G$ is internally-triangulated, would connect the same pair of vertices. We start by considering the case~$k=2$. In this case, we can define~$\mathcal I_P=V(K_P)-\{c_P\}$ and observe that~$\{c_P,\ell_P,r_P\} \cap \mathcal I_P=\emptyset$. If~$|V(K_P)|\geq 5$ (see \cref{fig:algorithm-case13-k-equal-2-big}), then~$\mathcal I_P$ satisfies Property~A. Namely,~$G[\mathcal I_P\cup\{c_P\}]$ is an outerplane graph, because it coincides with~$K_P$. Since~$n_P=|V(K_P)|+2\geq 7$, it follows that~$|\mathcal I_P|=|V(K_P)|-1=n_P-3\geq 2(n_P-1)/3$, as required. If~$|V(K_P)|\leq 4$ (see \cref{fig:algorithm-case13-k-equal-2-small}), then, since~$B_P$ is a non-trivial block that is not a leaf in~$T_K$, we have that~$|V(K_P)|=4$, that~$|V(B_P)|=3$, and that~$B_P$ has exactly one descendant block~$B$ in~$T_K$, where~$B$ is a trivial block. Suppose that the cutvertex parent of~$B$ is~$u_P$, as the case in which it is~$v_P$ is symmetric. By~\cref{le:case13}, we have that~$G_P$ contains the edges~$\ell_Pu_P$,~$r_Pu_P$, and~$r_Pv_P$. Then~$\mathcal I_P$ satisfies Property~B (if the cutvertex parent of~$B$ is~$v_P$, then~$\mathcal I_P$ satisfies Property~C). Namely, we have that~$G[\mathcal I_P\cup\{c_P,\ell_P\}]$ is an outerplane graph, since all its vertices are adjacent to~$r_P$, which is in~$L_1$ and not in~$\mathcal I_P\cup\{c_P,\ell_P\}$. Also,~$|\mathcal I_P|=3>2(n_P-2)/3=8/3$, as required. 

\begin{figure}[tb]
\centering
    \begin{subfigure}{0.48\textwidth}
		\centering
		\includegraphics[scale=1.2, page=2]{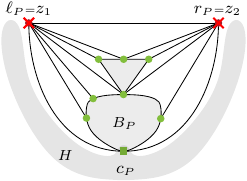}
		\subcaption{}
        \label{fig:algorithm-case13-k-equal-2-big}
	\end{subfigure}
    \hfill
    \begin{subfigure}{0.48\textwidth}
		\centering
		\includegraphics[scale=1.2, page=3]{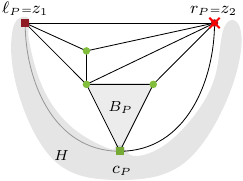}
		\subcaption{}
        \label{fig:algorithm-case13-k-equal-2-small}
	\end{subfigure}
\caption{Illustrations for Case~13 when~$k=2$ (a) if~$|V(K_P)| \geq 5$ and (b) if~$|V(K_P)| \leq 4$.}
\label{fig:algorithm-case13-k-equal-2}
\end{figure}

We can now assume that~$k\geq 3$. In this case, the construction of~$\mathcal I_P$ is similar to Cases 4 and 5.2. First, all the vertices of~$B_P$, except for~$c_P$, are in~$\mathcal I_P$. Second,~$\ell_P$ and~$r_P$ are not in~$\mathcal I_P$, hence~$\{c_P,\ell_P,r_P\} \cap \mathcal I_P=\emptyset$. Third, each opaque vertex~$z_j$ is not in~$\mathcal I_P$. Finally, whether a transparent vertex is in~$\mathcal I_P$ or not is decided according to five different cases.


In {\bf Case~13.1},~$z_1$ is transparent and~$z_2$ is opaque. Then, for~$j=3,\dots,k-1$, if~$z_j$ is a transparent vertex, it is in~$\mathcal I_P$ if~$z_{j-1}\notin \mathcal I_P$, and it is not in~$\mathcal I_P$ if~$z_{j-1}\in \mathcal I_P$. 

We prove that~$\mathcal I_P$ satisfies Property~B. In order to prove that~$G[\mathcal I_P\cup\{c_P,\ell_P\}]$ is an outerplane graph, it suffices to show that, for each vertex in~$V(K_P)-\{c_P\}$, there exists a neighbor among~$z_1,\dots,z_k$ which is not in~$\mathcal I_P\cup\{c_P,\ell_P\}$. This is done similarly to the proofs of~\cref{le:biconnectedK,le:cushy}. Namely, consider a vertex~$w$ in~$V(K_P)-\{c_P\}$. If~$w$ is a neighbor of an opaque vertex~$z_j$, then~$z_j$ is not in~$\mathcal I_P\cup\{c_P,\ell_P\}$, hence it is the desired neighbor of~$w$; this includes all vertices~$w$ that belong to non-trivial~(pesky) blocks which are descendants of~$B_P$ in~$T_K$. Each remaining vertex~$w$ is only neighbor of transparent vertices. Note that, for any two transparent vertices~$z_j$ and~$z_{j+1}$ that are consecutive in the sequence~$z_1,\dots,z_k$, we have that at most one of them belongs to~$\mathcal I_P\cup\{c_P,\ell_P\}$. This comes from the algorithm's construction if~$3\leq j\leq k-2$, from the fact that~$r_P$ is not in~$\mathcal I_P\cup\{c_P,\ell_P\}$ if~$j=k-1$ and from the assumption that~$z_2$ is opaque if~$j=1$ or~$j=2$. If~$w$ belongs to a trivial block~$B$, then its neighbors~$\ell_B$ and~$r_B$ are consecutive in the sequence~$z_1,\dots,z_k$, hence one of them does not belong to~$\mathcal I_P\cup\{c_P,\ell_P\}$. Also, if~$w$ is a vertex~$w_i$ of~$B_P$ and does not belong to any block that is a descendant of~$B_P$ in~$T_K$, then again its neighbors are consecutive in the sequence~$z_1,\dots,z_k$, hence one of them does not belong to~$\mathcal I_P\cup\{c_P,\ell_P\}$.

We prove that~$|\mathcal I_P|\geq 2(n_P-2)/3$. We charge each vertex of~$G_P$ not in~$\mathcal I_P$ and different from~$c_P$ and~$\ell_P$ to two vertices in~$\mathcal I_P$, so that each vertex in~$\mathcal I_P$ is charged with at most one vertex. The charging scheme is defined iteratively on the vertices~$z_2,\dots,z_k$ that do not belong to~$\mathcal I_P$. Throughout the algorithm that defines the charging scheme, we maintain an \emph{active edge}~$w_iz_j$, with the following meaning. Let~$\mathcal C^{i,j}$ be the cycle~$c_P w_1 w_2 \dots w_i z_j z_{j-1} \dots z_1$ and let~$\mathcal I_P^{i,j}$ be the subset of~$\mathcal I_P$ composed of those vertices that are inside or on the boundary of~$\mathcal C^{i,j}$. Then the following properties are satisfied:
\begin{enumerate}[(X1)]
    \item All the vertices among~$z_2,\dots,z_j$ that are not in~$\mathcal I_P^{i,j}$ have been charged to two vertices in~$\mathcal I_P^{i,j}$, so that each vertex in~$\mathcal I_P^{i,j}$ has been charged with at most one vertex.
    \item If~$j<k$, then the set~$\mathcal I_P^{i,j}$ contains a \emph{first free vertex}~$\phi_1$, that is, a vertex that has not been charged with any vertex among~$z_2,\dots,z_j$.
    \item If~$j<k$ and if all the vertices of a block~$B$ that is a descendant of~$B_P$ are inside or on the boundary of~$\mathcal C^{i,j}$, then~$\mathcal I_P^{i,j}$ contains a \emph{second free vertex}~$\phi_2$, that is, a vertex that has not been charged with any vertex among~$z_2,\dots,z_j$ and is different from~$\phi_1$. 
    \item We have that~$j=1$ or that~$z_j$ is not in~$\mathcal I_P^{i,j}$.
\end{enumerate}

Observe that, once~$j=k$, Property~(X1) implies that~$|\mathcal I_P|\geq 2(n_P-2)/3$. 

The first active edge is~$w_1z_1$. Recall that the edge~$w_1z_1=u_P\ell_P$ exists and that the cycle~$c_Pw_1z_1$ bounds an internal face of~$G$. Hence, Property~(X3) is vacuously satisfied, since no block~$B$ that is a descendant of~$B_P$ has all its vertices inside or on the boundary of~$\mathcal C^{1,1}$. Properties~(X1) and~(X4) are also trivially satisfied, since~$j=1$. Note that~$z_1=\ell_P$ does not need to be charged to any vertex. Finally, Property~(X2) is satisfied by setting~$\phi_1=w_1$; observe that~$\mathcal I_P^{1,1}=\{w_1\}$.  

Suppose now that our charging scheme has defined an active edge~$w_iz_j$, for some~$1\leq i\leq s$ and~$1\leq j\leq k-1$, so that Properties~(X1)--(X4) are satisfied. Let~$w_iz_jv^*$ be the cycle delimiting an internal face of~$G$ that does not lie inside~$\mathcal C^{i,j}$. We cannot have~$v^*=c_P$, since this would imply that~$i=s$ and~$j=k$, given that the edge~$c_Pz_k$ follows~$c_Pw_s$ in clockwise order around~$c_P$ in~$G$, by construction. We distinguish some cases.

\begin{figure}[tb]
\centering
    \begin{subfigure}{0.48\textwidth}
		\centering
		\includegraphics[scale=1.4, page=4]{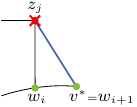}
		\subcaption{}
        \label{fig:algorithm-case13.1-case1}
	\end{subfigure}
    \hfill
    \begin{subfigure}{0.48\textwidth}
		\centering
		\includegraphics[scale=1.4, page=5]{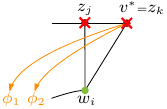}
		\subcaption{}
        \label{fig:algorithm-case13.1-case2}
	\end{subfigure}
\caption{Illustrations for Case~13.1 (a) when~$v^*=w_{i+1}$ and (b) when~$v^*=z_k$. In this and the following figures, the thick blue edge denotes the new active edge.}
\label{fig:algorithm-case13.1-part-I}
\end{figure}

\begin{figure}[tb]
\centering
    \begin{subfigure}{0.48\textwidth}
		\centering
		\includegraphics[scale=1.4, page=7]{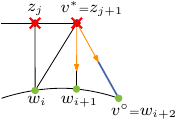}
		\subcaption{}
        \label{fig:algorithm-case13.1-case3.1}
	\end{subfigure}
    \hfill
    \begin{subfigure}{0.48\textwidth}
		\centering
		\includegraphics[scale=1.4, page=8]{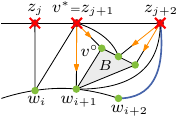}
		\subcaption{}
        \label{fig:algorithm-case13.1-case3.2}
	\end{subfigure}
\caption{Illustrations for Case~13.1 when~$v^*=z_{j+1}$, with~$j+1<k$, and~$z_{j+1}\notin \mathcal I_P$ (a) if~$v^{\circ}=w_{i+2}$ and (b) if~$v^{\circ}$ belongs to a non-trivial pesky block~$B$ that is a descendant of~$B_P$ and whose parent is~$w_{i+1}$.}
\label{fig:algorithm-case13.1-part-II}
\end{figure}

\begin{enumerate}[(I)]
    \item First, if~$v^*=w_{i+1}$ (see \cref{fig:algorithm-case13.1-case1}), we proceed with~$w_{i+1}z_j$ as an active edge. Properties~(X1)--(X4) are satisfied by~$w_{i+1}z_j$ since they are satisfied by~$w_iz_j$.
    \item Second, suppose that~$v^*=z_k$ (see \cref{fig:algorithm-case13.1-case2}). Then we conclude the algorithm by charging~$z_k$ to~$\phi_1$ and~$\phi_2$. Observe that~$\mathcal C^{i,j}$ contains inside or on its boundary all the vertices of a block~$B$ that is a descendant of~$B_P$, hence~$\mathcal I_P^{i,j}$ contains a second free vertex~$\phi_2$. Indeed, by~\cref{le:case13}, we have that a block~$B$ that is a descendant of~$B_P$ exists. Furthermore,~$B$ cannot be contained inside the cycle~$w_i w_{i+1} \dots w_s z_k$, as otherwise we would have~$\ell_B=r_B=z_k$.    
    \item Third, suppose that~$v^*=z_{j+1}$, with~$j+1<k$, and that~$z_{j+1}\notin \mathcal I_P$ (see \cref{fig:algorithm-case13.1-part-II}). Then~$z_{j+1}$ is opaque. Indeed, we have that~$j=1$ or that~$z_j$ is not in~$\mathcal I_P^{i,j}$. In the former case,~$z_{j+1}$ is opaque by assumption. In the latter case, if~$z_{j+1}$ were transparent, then it would belong to~$\mathcal I_P$, by the algorithm's construction, whereas it does not. Note that there is no block~$B$ that is descendant of~$B_P$, whose parent is~$w_i$, and such that~$\ell_B=z_{j+1}$; indeed, if such a block existed, then the edge~$\ell_Bc_B^l$ would not belong to~$G$, contradicting~\cref{le:case13}. It follows that~$i<s$, that the edge~$w_{i+1}z_{j+1}$ exists, and that the cycle~$w_iw_{i+1}z_{j+1}$ delimits an internal face of~$G$. Let~$v^{\circ}$ be the vertex which is different from~$w_i$ and such that the cycle~$w_{i+1}z_{j+1}v^{\circ}$ delimits an internal face of~$G$. We have~$v^{\circ}\neq z_{j+2}$, as otherwise~$z_{j+1}$ would be transparent, while it is opaque. Also,~$v^{\circ}$ does not belong to a trivial block~$B$ that is descendant of~$B_P$, whose parent is~$w_{i+1}$, and such that~$\ell_B=z_{j+1}$, as otherwise~$z_{j+1}$ would be transparent, while it is opaque. Also,~$v^{\circ}\neq c_P$, since~$v^{\circ}=c_P$ would imply~$j+1=k$, given that the edge~$c_Pz_k$ follows~$c_Pw_s$ in clockwise order around~$c_P$ in~$G$, by construction. Hence, two cases are possible.
    \begin{itemize}
        \item If~$v^{\circ}=w_{i+2}$ (see \cref{fig:algorithm-case13.1-case3.1}), then we charge~$z_{j+1}$ to~$w_{i+1}$ and~$w_{i+2}$ and we proceed with~$w_{i+2}z_{j+1}$ as an active edge. Property~(X1) is satisfied since it is satisfied by~$w_iz_j$ and since~$z_{j+1}$ has been charged to~$w_{i+1}$ and~$w_{i+2}$, which belong to~$\mathcal I_P^{i+2,j+1}$ and do not belong to~$\mathcal I_P^{i,j}$. Properties~(X2) and~(X3) are satisfied since they are satisfied by~$w_iz_j$ and since~$\mathcal C^{i+2,j+1}$ contains the same blocks inside or on its boundary as~$\mathcal C^{i,j}$. Finally, Property~(X4) is satisfied since~$z_{j+1}$ is not in~$\mathcal I_P^{i+2,j+1}$.  
        \item Otherwise,~$v^{\circ}$ belongs to a non-trivial pesky block~$B$ that is a descendant of~$B_P$, whose parent is~$w_{i+1}$, and such that~$\ell_B=z_{j+1}$ (see \cref{fig:algorithm-case13.1-case3.2}). Let~$\ell_B=z_{j+1},z_{j+2},\dots,z_{j+h}=r_B$ be the cage path of~$B$. By \cref{obs:pesky}, we can charge~$z_{j+1},z_{j+2},\dots,z_{j+h}$ to the~$2h$ vertices of~$B$ (including~$w_{i+1}$). By~\cref{le:case13}, the edge~$w_{i+2}z_{j+h}$ exists and the cycle~$w_{i+1}w_{i+2}z_{j+h}$ delimits an internal face of~$G$. If~$j+h=k$, we are done. Otherwise, we proceed with~$w_{i+2}z_{j+h}$ as an active edge. Property~(X1) is satisfied since it is satisfied by~$w_iz_j$ and since each of~$z_{j+1},z_{j+2},\dots,z_{j+h}$ has been charged to two distinct vertices in~$B$, which belong to~$\mathcal I_P^{i+2,j+h}$ and do not belong to~$\mathcal I_P^{i,j}$. Property~(X2) is satisfied since it is satisfied by~$w_iz_j$. Property (X3) is satisfied by setting~$\phi_2=w_{i+2}$. Finally, Property~(X4) is satisfied since~$z_{j+h}$ is not in~$\mathcal I_P^{i+2,j+h}$.    
    \end{itemize}
    
\begin{figure}[tb]
\centering
\begin{subfigure}{0.32\textwidth}
    \centering
    \includegraphics[scale=1.4, page=9]{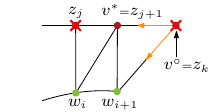}
    \subcaption{}
    \label{fig:algorithm-case13.1-case4.1}
\end{subfigure}
\hfill
\begin{subfigure}{0.32\textwidth}
    \centering
    \includegraphics[scale=1.4, page=10]{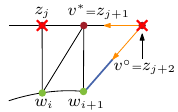}
    \subcaption{}
    \label{fig:algorithm-case13.1-case4.2}
\end{subfigure}
\hfill
\begin{subfigure}{0.32\textwidth}
    \centering
    \includegraphics[scale=1.4, page=11]{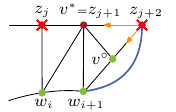}
    \subcaption{}
    \label{fig:algorithm-case13.1-case4.3}
\end{subfigure}
\caption{Illustrations for Case~13.1 when~$v^*=z_{j+1}$, with~$j+1<k$, and~$z_{j+1}\in \mathcal I_P$ (a) if~$v^{\circ}=z_{j+2}$ and~$j+2=k$, (b) if~$v^{\circ}=z_{j+2}$ and~$j+2<k$, and (c) if~$v^{\circ}$ belongs to a trivial block~$B$ that is a descendant of~$B_P$ and whose parent is~$w_{i+1}$.}
\label{fig:algorithm-case13.1-part-III}
\end{figure}

\item Fourth, suppose that~$v^*=z_{j+1}$, with~$j+1<k$, and that~$z_{j+1}\in \mathcal I_P$ (see \cref{fig:algorithm-case13.1-part-III}). Then~$z_{j+1}$ is transparent and the cycle~$w_iw_{i+1}z_{j+1}$ delimits an internal face of~$G$. Let~$v^{\circ}$ be the vertex which is different from~$w_i$ and such that the cycle~$w_{i+1}z_{j+1}v^{\circ}$ delimits an internal face of~$G$. Since~$z_{j+1}$ is transparent, we have that~$z_{j+2}\notin \mathcal I_P$ and three cases are possible.  
\begin{itemize}
    \item If~$v^{\circ}=z_{j+2}$ and~$j+2=k$ (see \cref{fig:algorithm-case13.1-case4.1}), we conclude the algorithm by charging~$z_k$ to~$w_{i+1}$ and~$z_{j+1}$ (note that~$\phi_1$ and~$\phi_2$ remain uncharged). 
    \item If~$v^{\circ}=z_{j+2}$ and~$j+2<k$ (see \cref{fig:algorithm-case13.1-case4.2}), then we charge~$z_{j+2}$ to~$w_{i+1}$ and~$z_{j+1}$, and we proceed with~$w_{i+1}z_{j+2}$ as an active edge. Property~(X1) is satisfied since it is satisfied by~$w_iz_j$ and since~$z_{j+2}$ has been charged to~$w_{i+1}$ and~$z_{j+1}$, which belong to~$\mathcal I_P^{i+1,j+2}$ and not~$\mathcal I_P^{i,j}$. Properties~(X2) and~(X3) are satisfied since they are satisfied by~$w_iz_j$ and since~$\mathcal C^{i+1,j+2}$ contains the same blocks inside or on its boundary as~$\mathcal C^{i,j}$. Finally, Property~(X4) is satisfied since~$z_{j+2}$ is not in~$\mathcal I_P^{i+1,j+2}$.  
    \item Otherwise,~$v^{\circ}$ belongs to a trivial block~$B$ that is a descendant of~$B_P$, whose parent is~$w_{i+1}$, and such that~$\ell_B=z_{j+1}$ and~$r_B=z_{j+2}$ (see \cref{fig:algorithm-case13.1-case4.3}).  We charge~$z_{j+2}$ to~$v^{\circ}$ and~$z_{j+1}$ and we proceed with~$w_{i+1}z_{j+2}$ as an active edge. Property~(X1) is satisfied since it is satisfied by~$w_iz_j$ and since~$z_{j+2}$ has been charged to~$v^{\circ}$ and~$z_{j+1}$, which belong to~$\mathcal I_P^{i+1,j+2}$ and not to~$\mathcal I_P^{i,j}$. Property~(X2) is satisfied since it is satisfied by~$w_iz_j$. Property (X3) is satisfied by setting~$\phi_2=w_{i+1}$. Finally, Property~(X4) is satisfied since~$z_{j+2}$ is not in~$\mathcal I_P^{i+1,j+2}$.   
\end{itemize}

\item Fifth, suppose that~$v^*$ belongs to a non-trivial pesky block~$B$ that is a descendant of~$B_P$, whose parent is~$w_i$, and such that~$\ell_B=z_j$ (see \cref{fig:algorithm-case13.1-part-IV}). Let~$\ell_B=z_j,z_{j+1},\dots,z_{j+h}=r_B$ be the cage path of~$B$. By \cref{obs:pesky}, we can charge~$z_{j+1},z_{j+2},\dots,z_{j+h}$ to~$2h$ vertices among the~$2h+1$ vertices of~$B$ different from~$w_i$. By~\cref{le:case13}, the edge~$w_{i+1}z_{j+h}$ exists and the cycle~$w_{i}w_{i+1}z_{j+h}$ delimits an internal face of~$G$. If~$j+h=k$, we are done; note that this is always the case if~$w_{i+1}$ is actually~$c_P$, given that the edge~$c_Pz_k$ follows~$c_Pw_s$ in clockwise order around~$c_P$ in~$G$, by construction. Otherwise, we proceed with~$w_{i+1}z_{j+h}$ as an active edge. Property~(X1) is satisfied since it is satisfied by~$w_iz_j$ and since each of~$z_{j+1},z_{j+2},\dots,z_{j+h}$ has been charged to two distinct vertices in~$B$ different from~$w_i$; these vertices belong to~$\mathcal I_P^{i+1,j+h}$ and do not belong to~$\mathcal I_P^{i,j}$. Property~(X2) is satisfied since it is satisfied by~$w_iz_j$. Also, Property (X3) is satisfied by setting~$\phi_2=w_{i+1}$; note that there is even one more vertex of~$B$ which has not been charged with any vertex. Finally, Property~(X4) is satisfied since~$z_{j+h}$ is not in~$\mathcal I_P^{i+1,j+h}$. 

\begin{figure}[tb]
\centering
\includegraphics[scale=1.4, page=12]{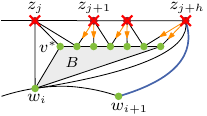}
\caption{Illustration for Case~13.1 when~$v^*$ belongs to a non-trivial pesky block~$B$ that is a descendant of~$B_P$ and whose parent is~$w_i$.}
\label{fig:algorithm-case13.1-part-IV}
\end{figure}
\begin{figure}[b]
\centering
\includegraphics[scale=1.4, page=13]{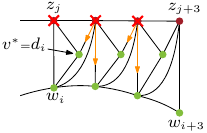}
\caption{Illustration for Case~13.1 when~$v^*$ belongs to a trivial block~$B$ that is a descendant of~$B_P$, whose parent is~$w_i$, and assuming~$h=2$. The figure only shows the charging of~$z_{j+1}$ and~$z_{j+2}$.}
\label{fig:algorithm-case13.1-part-V}
\end{figure}

\item Sixth, suppose that~$v^*$ belongs to a trivial block~$B$ that is a descendant of~$B_P$, whose parent is~$w_i$, and such that~$\ell_B=z_j$. Let~$h$ be the maximal index such that, for~$l=0,1\dots,h$, there exists a trivial block~$B_{i+l}$ that is a descendant of~$B_P$, whose parent is~$w_{i+l}$, and such that~$\ell_{B_{i+l}}=z_{j+l}$. Note that~$h\geq 0$, since~$B_i=B$ satisfies the required properties. Also, for~$l=0,1\dots,h-1$, we have~$r_{B_{i+l}}=z_{j+l+1}=\ell_{B_{i+l+1}}$ and such a vertex is opaque. Moreover, we have that~$r_{B_{i+h}}=z_{j+h+1}$; such a vertex might be opaque or transparent (see \cref{fig:algorithm-case13.1-part-V} for an example with~$h=2$). For~$l=0,\dots,h$, let~$d_{i+l}$ be the vertex of~$B_{i+l}$ different from~$w_{i+l}$. For~$l=1,\dots,h$, we charge~$z_{j+l}$ to~$w_{i+l}$ and to~$d_{i+l-1}$. If~$j+h+1=k$ (see \cref{fig:algorithm-case13.1-case6.1}), then we conclude the algorithm by charging~$z_k$ to~$\phi_1$ and to~$d_{i+h}$. Thus we can assume that~$j+h+1<k$. This implies that~$i+h<s$, given that the edge~$c_Pz_k$ follows~$c_Pw_s$ in clockwise order around~$c_P$ in~$G$, by construction.  By \cref{le:case13}, the edge~$w_{i+h+1}z_{j+h+1}$ exists and the cycle~$w_{i+h}w_{i+h+1}z_{j+h+1}$ delimits an internal face of~$G$.  Note that~$z_{j+h+1}$ still needs to be charged, if it is opaque, and that~$d_{i+h}$ and~$w_{i+h+1}$ have not yet been charged with any vertex. We cannot simply proceed with~$w_{i+h+1}z_{j+h+1}$ as the active edge, even if~$z_{j+h+1}$ is opaque, as Property (X3) would not be satisfied. Indeed, we encountered a block~$B$ descendant of~$B_P$ but we could not set~$\phi_2$ to any uncharged vertex. For this reason, we need to perform a case distinction very similar to the one we just presented, but with one more guarantee: Let~$w_{i+h+1}z_{j+h+1}u^*$ be the cycle delimiting an internal face of~$G$ with~$u^*\neq w_{i+h}$; then, by the maximality of~$h$, we have that~$u^*$ does not belong to a trivial block that is a descendant of~$B_P$, whose parent is~$w_{i+h+1}$, and such that~$\ell_B=z_{j+h+1}$. Note that~$u^*\neq c_P$, as~$u^*=c_P$ would imply that~$j+h+1=k$, given that the edge~$c_Pz_k$ follows~$c_Pw_s$ in clockwise order around~$c_P$ in~$G$, by construction. Thus, the following cases arise.

\begin{figure}[tb]
\centering
\begin{subfigure}{0.32\textwidth}
    \centering
    \includegraphics[scale=1.4, page=14]{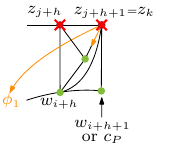}
    \subcaption{} 
    \label{fig:algorithm-case13.1-case6.1}
\end{subfigure}
\hfill
\begin{subfigure}{0.32\textwidth}
    \centering
    \includegraphics[scale=1.4, page=15]{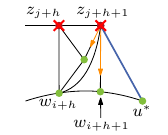}
    \subcaption{}
    \label{fig:algorithm-case13.1-case6.2}
\end{subfigure}
\hfill
\begin{subfigure}{0.32\textwidth}
    \centering
    \includegraphics[scale=1.4, page=16]{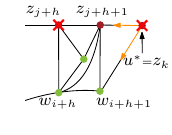}
    \subcaption{}
    \label{fig:algorithm-case13.1-case6.3}
\end{subfigure}
\caption{Illustrations for Case~13.1 when~$v^*$ belongs to a trivial block~$B$ that is a descendant of~$B_P$ and whose parent is~$w_i$ (a) if~$j+h+1=k$, (b) if~$j+h+1<k$ and~$u^*=w_{i+h+2}$, and (c) if~$u^*=z_{j+h+2}=z_k$.}
\label{fig:algorithm-case13.1-part-VI}
\end{figure}

\begin{enumerate}[1.]
\item First, if~$u^*=w_{i+h+2}$ (see \cref{fig:algorithm-case13.1-case6.2}), then~$z_{j+h+1}$ is opaque. We charge~$z_{j+h+1}$ to~$w_{i+h+1}$ and to~$d_{i+h}$, and we proceed with~$w_{i+h+2}z_{j+h+1}$ as an active edge. Property~(X1) is satisfied since it is satisfied by~$w_iz_j$ and since each of~$z_{j+1},z_{j+2},\dots,z_{j+h+1}$ has been charged to two distinct vertices that belong to~$\mathcal I_P^{i+h+2,j+h+1}$ and do not belong to~$\mathcal I_P^{i,j}$.  Property~(X2) is satisfied since it is satisfied by~$w_iz_j$. Also, Property (X3) is satisfied by setting~$\phi_2=w_{i+h+2}$. Finally, Property~(X4) is satisfied since~$z_{j+h+1}$ is opaque and hence it is not in~$\mathcal I_P^{i+h+2,j+h+1}$. 
\item Second, if~$u^*=z_{j+h+2}=z_k$ (see \cref{fig:algorithm-case13.1-case6.3}), then~$z_{j+h+1}$ is transparent and belongs to~$\mathcal I_P$. We conclude the algorithm by charging~$z_k$ to~$w_{i+h+1}$ and~$z_{j+h+1}$ (note that~$\phi_1$ and~$d_{i+h}$ remain uncharged).   
\item Third, suppose that~$u^*=z_{j+h+2}$, with~$j+h+2<k$ (see \cref{fig:algorithm-case13.1-case6.4}). Then~$z_{j+h+1}$ is transparent and belongs to~$\mathcal I_P$, while~$z_{j+h+2}$ does not belong to~$\mathcal I_P$. We charge~$z_{j+h+2}$ to~$d_{i+h}$ and to~$z_{j+h+1}$, and we proceed with~$w_{i+h+1}z_{j+h+2}$ as an active edge. Property~(X1) is satisfied since it is satisfied by~$w_iz_j$ and since each of~$z_{j+1},z_{j+2},\dots,z_{j+h},z_{j+h+2}$ has been charged to two distinct vertices that belong to~$\mathcal I_P^{i+h+1,j+h+2}$ and do not belong to~$\mathcal I_P^{i,j}$. Property~(X2) is satisfied since it is satisfied by~$w_iz_j$. Also, Property (X3) is satisfied by setting~$\phi_2=w_{i+h+1}$. Finally, Property~(X4) is satisfied since~$z_{j+h+2}$ is not in~$\mathcal I_P^{i+h+1,j+h+2}$. 
\item Finally, suppose that~$u^*$ belongs to a non-trivial pesky block~$B'$ that is a descendant of~$B_P$, whose parent is~$w_{i+h+1}$, and such that~$\ell_{B'}=z_{j+h+1}$ (see \cref{fig:algorithm-case13.1-case6.5}). Let~$\ell_{B'}=z_{j+h+1},z_{j+h+2},\dots,z_{j+h+h'}=r_{B'}$ be the cage path of~$B'$. By \cref{obs:pesky}, we can charge~$z_{j+h+1},z_{j+h+2},\dots,z_{j+h+h'}$ to the~$2h'$ vertices of~$B'$, including~$w_{i+h+1}$. If~$j+h+h'=k$, we are done. Otherwise, we proceed with~$w_{i+h+1}z_{j+h+h'}$ as an active edge. Property~(X1) is satisfied since it is satisfied by~$w_iz_j$ and since each of~$z_{j+1},z_{j+2},\dots,z_{j+h+h'}$ has been charged to two distinct vertices that belong to~$\mathcal I_P^{i+h+1,j+h+h'}$ and do not belong to~$\mathcal I_P^{i,j}$. Property~(X2) is satisfied since it is satisfied by~$w_iz_j$. Also, Property (X3) is satisfied by setting~$\phi_2=d_{i+h}$. Finally, Property~(X4) is satisfied since~$z_{j+h+h'}$ is not in~$\mathcal I_P^{i+h+1,j+h+h'}$. 
\end{enumerate}
\end{enumerate}

This concludes the discussion of Case~13.1.

\begin{figure}[tb]
\centering
    \begin{subfigure}{0.48\textwidth}
		\centering
		\includegraphics[scale=1.4, page=17]{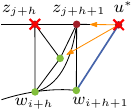}
		\subcaption{}
        \label{fig:algorithm-case13.1-case6.4}
	\end{subfigure}
    \hfill
    \begin{subfigure}{0.48\textwidth}
		\centering
		\includegraphics[scale=1.4, page=18]{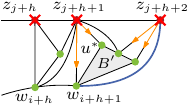}
		\subcaption{}
        \label{fig:algorithm-case13.1-case6.5}
	\end{subfigure}
\caption{Illustrations for Case~13.1 when~$v^*$ belongs to a trivial block~$B$ that is a descendant of~$B_P$ and whose parent is~$w_i$ (a) if~$u^*=z_{j+h+2}$, with~$j+h+2<k$ and (b) if~$u^*$ belongs to a non-trivial pesky block~$B'$ that is a descendant of~$B_P$ and whose parent is~$w_{i+h+1}$.}
\label{fig:algorithm-case13.1-part-VII}
\end{figure}


In {\bf Case~13.2},~$z_k$ is transparent and~$z_{k-1}$ is opaque. Then, for~$j=k-2,k-1,\dots,2$, if~$z_j$ is a transparent vertex, it is in~$\mathcal I_P$ if~$z_{j+1}\notin \mathcal I_P$, and it is not in~$\mathcal I_P$ if~$z_{j+1}\in \mathcal I_P$. We have that~$\mathcal I_P$ satisfies Property~C. The proof is symmetric to the one of Case~13.1.


In {\bf Case~13.3}, we have that Cases 13.1 and 13.2 do not apply, and that~$z_k$ is opaque. Then, for~$j=2,\dots,k-1$, if~$z_j$ is a transparent vertex, it is in~$\mathcal I_P$ if~$z_{j-1}\notin \mathcal I_P$, and it is not in~$\mathcal I_P$ if~$z_{j-1}\in \mathcal I_P$. 

We prove that~$\mathcal I_P$ satisfies Property~A. The proof that~$G[\mathcal I_P\cup\{c_P\}]$ is an outerplane graph is very similar to the one of Case~13.1, with the only difference in the argument which proves that, for any two transparent vertices~$z_j$ and~$z_{j+1}$, at most one of them belongs to~$\mathcal I_P\cup\{c_P\}$. This comes from the algorithm's construction if~$2\leq j\leq k-2$ and from the fact that~$\ell_P$ and~$r_P$ are not in~$\mathcal I_P\cup\{c_P\}$ if~$j=1$ or~$j=k-1$, respectively.

We prove that~$|\mathcal I_P|\geq 2(n_P-1)/3$. As in Case~13.1, we iteratively charge the vertices~$z_1,\dots,z_k$ that do not belong to~$\mathcal I_P$ to two vertices in~$\mathcal I_P$, so that each vertex in~$\mathcal I_P$ is charged with at most one vertex. We again maintain an active edge~$w_iz_j$, for which~$\mathcal C^{i,j}$ and~$\mathcal I_P^{i,j}$ are defined as in Case~13.1, so that the following properties are satisfied:
\begin{enumerate}[(Y1)]
    \item All the vertices among~$z_1,\dots,z_j$ that are not in~$\mathcal I_P^{i,j}$ have been charged to two vertices in~$\mathcal I_P^{i,j}$, so that each vertex in~$\mathcal I_P^{i,j}$ has been charged with at most one vertex.
    \item If~$j<k$ and if all the vertices of a block~$B$ that is a descendant of~$B_P$ are inside or on the boundary of~$\mathcal C^{i,j}$, then~$\mathcal I_P^{i,j}$ contains a \emph{first free vertex}~$\phi_1$, that is, a vertex that has not been charged with any vertex among~$z_1,\dots,z_j$. 
    \item We have that~$z_j$ is not in~$\mathcal I_P^{i,j}$.
\end{enumerate}

Observe that, once~$j=k$, Property~(Y1) implies that~$|\mathcal I_P|\geq 2(n_P-1)/3$. 

While the iteration in the charging scheme is extremely similar to the one in Case~13.1, the initialization is much more complicated and distinguishes several cases. Let~$w_1z_1v^*_1$ be the cycle delimiting an internal face of~$G$ with~$v^*\neq c_P$. 

\begin{itemize}
    \item First, suppose that~$v^*_1=w_{2}$ (see \cref{fig:algorithm-case13.3-init-1}). Then~$z_1$ is opaque. We charge~$z_1$ to~$w_1$ and~$w_2$ and proceed with~$w_2z_1$ as an active edge. Property (Y1) is satisfied since~$z_1$ has been charged to two vertices in~$\mathcal I_P^{2,1}$ and each of them has been charged with~$z_1$ only. Also, Property (Y2) is vacuously satisfied since~$\mathcal C^{2,1}$ does not contain inside or on its boundary any block~$B$ that is a descendant of~$B_P$. Finally, Property (Y3) is satisfied since~$z_1$ is not in~$\mathcal I_P^{2,1}$, by construction. 
    \item Second, suppose that~$v^*_1$ belongs to a non-trivial pesky block~$B$ that is a descendant of~$B_P$, whose parent is~$w_1$, and such that~$\ell_B=z_1$ (see \cref{fig:algorithm-case13.3-init-2}). Let~$\ell_B=z_1,z_2,\dots,z_h=r_B$ be the cage path of~$B$. Note that all of~$z_1,z_2,\dots,z_h$ are opaque. By \cref{obs:pesky}, we can charge~$z_1,z_2,\dots,z_h$ to the~$2h$ vertices of~$B$, including~$w_1$. By~\cref{le:case13}, the edge~$w_2z_h$ exists and the cycle~$w_1w_2z_h$ delimits an internal face of~$G$. If~$h=k$, we are done. Otherwise, we proceed with~$w_2z_h$ as an active edge. Property~(Y1) is satisfied since each of~$z_1,z_2,\dots,z_h$ has been charged to two distinct vertices which belong to~$\mathcal I_P^{2,h}$. Property~(Y2) is satisfied by setting~$\phi_1=w_2$. Also, Property (Y3) is satisfied since~$z_h$ is not in~$\mathcal I_P^{2,h}$.

\begin{figure}[tb]
\centering
\begin{subfigure}{0.32\textwidth}
    \centering
    \includegraphics[scale=1.4, page=20]{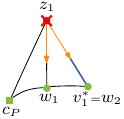}
    \subcaption{}
    \label{fig:algorithm-case13.3-init-1}
\end{subfigure}
\hfill
\begin{subfigure}{0.32\textwidth}
    \centering
    \includegraphics[scale=1.4, page=21]{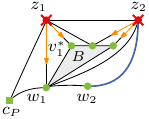}
    \subcaption{}
    \label{fig:algorithm-case13.3-init-2}
\end{subfigure}
\hfill
\begin{subfigure}{0.32\textwidth}
    \centering
    \includegraphics[scale=1.4, page=25]{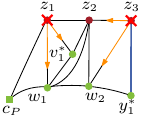}
    \subcaption{}
    \label{fig:algorithm-case13.3-init-4}
\end{subfigure}
\caption{Illustrations for Case~13.3 (a) if~$v^*_1=w_{2}$, (b) if~$v^*_1$ belongs to a non-trivial pesky block~$B$ that is a descendant of~$B_P$ and whose parent is~$w_1$, and (c) if~$v^*_1$ belongs to a trivial block~$B$ that is a descendant of~$B_P$ and whose parent is~$w_1$.}
\label{fig:algorithm-case13.3-part-I}
\end{figure}
    \item Third, suppose that~$v^*_1=z_{2}$ (see \cref{fig:algorithm-case13.3-part-II}). It follows that~$z_1$ is transparent. Since we are not in Case~13.1, we have that~$z_2$ is transparent, as well. This implies that the edge~$w_2z_2$ exists and that the cycle~$w_1w_2z_2$ delimits an internal face of~$G$. Note that~$z_2\in \mathcal I_P$, by construction, and hence~$z_3\notin \mathcal I_P$. Let~$w_2z_2u^*_1$ be the cycle delimiting an internal face of~$G$ with~$u^*_1\neq w_1$. Since~$z_2$ is transparent, we have that either~$u^*_1=z_3$ or that~$u^*_1$ belongs to a trivial block~$B$ that is a descendant of~$B_P$, whose parent is~$w_2$, and such that~$\ell_B=z_2$. We discuss these two cases.

    Suppose first that~$u^*_1=z_3$ and let~$w_2z_3y^*_1$ be the cycle delimiting an internal face of~$G$ with~$y^*_1\neq z_2$ (see \cref{fig:algorithm-case13.3-init-3.1.1}). By \cref{le:case13}, we have that~$y^*_1$ does not belong to a block~$B'$ that is a descendant of~$B_P$, whose parent is~$w_2$, and such that~$\ell_{B'}=z_3$. Also,~$y^*_1\neq z_4$, as otherwise~$z_3$ would have degree~$3$ in~$G$. Hence, we have that~$y^*_1=w_3$, then we charge~$z_1$ to~$w_1$ and~$z_2$, we charge~$z_3$ to~$w_2$ and~$w_3$, and we proceed with~$w_3z_3$ as an active edge. Note that~$k>3$, since~$\mathcal C^{3,3}$ does not contain inside or on its boundary all the vertices of any block that is a descendant of~$B_P$. For the same reason,  (Y2) is vacuously satisfied. Property (Y1) is satisfied since~$z_1$ and~$z_3$ have each been charged to two vertices in~$\mathcal I_P^{3,3}$ and each vertex in~$\mathcal I_P^{3,3}$ has been charged with one vertex only. Also, Property (Y3) is satisfied since~$z_3$ is not in~$\mathcal I_P^{3,3}$, by construction.

    Suppose next that~$u^*_1$ belongs to a trivial block~$B$ that is a descendant of~$B_P$, whose parent is~$w_2$, and such that~$\ell_B=z_2$ (see \cref{fig:algorithm-case13.3-init-3.2}). Then~$r_B=z_3$ and, by \cref{le:case13}, the edge~$w_3z_3$ exists and the cycle~$w_2w_3z_3$ delimits an internal face of~$G$. We charge~$z_1$ to~$w_1$ and~$z_2$ and we charge~$z_3$ to~$w_2$ and to~$u^*_1$. If~$k=3$, we are done.  Otherwise, we proceed with~$w_3z_3$ as an active edge. Property~(Y1) is satisfied since~$z_1$ and~$z_3$ have each been charged to two vertices in~$\mathcal I_P^{3,3}$ and each vertex in~$\mathcal I_P^{3,3}$ has been charged with one vertex only. Property~(Y2) is satisfied by setting~$\phi_1=w_3$. Also, Property (Y3) is satisfied since~$z_3$ is not in~$\mathcal I_P^{3,3}$.

\begin{figure}[tb]
\centering
\begin{subfigure}{0.48\textwidth}
    \centering
    \includegraphics[scale=1.4, page=22]{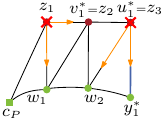}
    \subcaption{}
    \label{fig:algorithm-case13.3-init-3.1.1}
\end{subfigure}
\hfill
\begin{subfigure}{0.48\textwidth}
    \centering
    \includegraphics[scale=1.4, page=24]{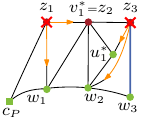}
    \subcaption{}
    \label{fig:algorithm-case13.3-init-3.2}
\end{subfigure}
\caption{Illustrations for Case~13.3 when~$v^*_1=z_{2}$ (a) if~$u^*_1=z_3$, and (b) if~$u^*_1$ belongs to a trivial block~$B$ that is a descendant of~$B_P$ and whose parent is~$w_2$.}
\label{fig:algorithm-case13.3-part-II}
\end{figure}


    \item Finally, suppose that~$v^*_1$ belongs to a trivial block~$B$ that is a descendant of~$B_P$, whose parent is~$w_1$, and such that~$\ell_B=z_1$ (see \cref{fig:algorithm-case13.3-init-4}). Note that~$r_B=z_2$ and that~$z_1$ is transparent. Since we are not in Case~13.1,~$z_2$ is transparent, as well. By~\cref{le:case13}, the edge~$w_2z_2$ exists and the cycle~$w_1w_2z_2$ delimits an internal face of~$G$. Since~$z_2$ is transparent, the edge~$w_2z_3$ exists and the cycle~$w_2w_3z_3$ delimits an internal face of~$G$. Let~$w_2z_3y^*_1$ be the cycle delimiting an internal face of~$G$ with~$y^*_1\neq z_2$. By \cref{le:case13}, we have that~$y^*_1$ does not belong to a block~$B$ that is a descendant of~$B_P$, whose parent is~$w_2$, and such that~$\ell_{B}=z_3$. Also,~$y^*_1\neq z_4$, as otherwise~$z_3$ would have degree~$3$ in~$G$. Hence,~$y^*_1=w_3$. We charge~$z_1$ to~$w_1$ and~$v^*_1$, and we charge~$z_3$ to~$z_2$ and~$w_2$. If~$k=3$, we are done. Otherwise, we proceed with~$w_3z_3$ as an active edge. Property (Y1) is satisfied since~$z_1$ and~$z_3$ have each been charged to two vertices in~$\mathcal I_P^{3,3}$ and each vertex in~$\mathcal I_P^{3,3}$ has been charged with one vertex only. Property~(Y2) is satisfied by setting~$\phi_1=w_3$. Also, Property (Y3) is satisfied since~$z_3$ is not in~$\mathcal I_P^{3,3}$.

\end{itemize}


This concludes the initialization. The algorithm now proceeds iteratively, by means of a case distinction which is the same as in Case~13.1, except for three differences. First, whenever in Case~13.1 we set the second free vertex~$\phi_2$ after encountering a block~$B$ descendant of~$B_P$, here we set the first free vertex~$\phi_1$ (which in Case~13.1 was already set in the initialization). This guarantees the satisfaction of Property~(Y2) throughout the algorithm. Second, in Case~13.1(II), in which~$v^*=z_k$, we conclude the algorithm by charging~$z_k$ to~$\phi_1$ and~$\phi_2$. Here, we do not have the second free vertex~$\phi_2$, however we can exploit the fact that~$z_k$ is opaque. Indeed, this implies that~$i<s$ and thus we can conclude the algorithm by charging~$z_k$ to~$\phi_1$ and~$w_s$. Third, in Case~13.1(VI), if~$j+h+1=k$ we conclude the algorithm by charging~$z_k$ to~$\phi_1$ and to~$d_{i+h}$. Here,  we have not necessarily set the first free vertex~$\phi_1$, however we can exploit the fact that~$z_k$ is opaque. Indeed, this implies that~$i+h<s$ and thus we can conclude the algorithm by charging~$z_k$ to~$w_s$ and to~$d_{i+h}$.

In {\bf Case~13.4}, we have that Cases 13.1 and 13.2 do not apply, and that~$z_1$ is opaque. Then, for~$j=k-1,\dots,2$, if~$z_j$ is a transparent vertex, it is in~$\mathcal I_P$ if~$z_{j+1}\notin \mathcal I_P$, and it is not in~$\mathcal I_P$ if~$z_{j+1}\in \mathcal I_P$. We have that~$\mathcal I_P$ satisfies Property~A. The proof is symmetric to the one of Case~13.3.

If we are in neither of Cases 13.1, 13.2, 13.3, and 13.4, then we are in {\bf Case~13.5}, and~$z_1$,~$z_2$,~$z_{k-1}$, and~$z_k$ are all transparent. Let~$a$ and~$b$ be the maximum and minimum indices, respectively, such that~$z_1,\dots,z_a$ and~$z_b,\dots,z_k$ are all transparent. Note that~$a\geq 2$ and~$b\leq k-1$; also,~$a=k$ (and~$b=1$) is possible. 

We first get rid of the special case~$a=k$ (and~$b=1$). If~$k$ is odd, we insert in~$\mathcal I_P$ all the vertices~$z_j$ with~$j$ even. Then~$\mathcal I_P$ satisfies Property~A. The proof that~$G[\mathcal I_P\cup\{c_P\}]$ is an outerplane graph is very similar to the one of Case~13.1. In particular, that at most one of any two consecutive transparent vertices~$z_j$ and~$z_{j+1}$ belongs to~$\mathcal I_P\cup\{c_P\}$ is true by construction. Also, the number of vertices in~$\mathcal I_P$ is equal to~$\frac{k-1}{2}$ (the number of vertices~$z_2,z_4,\dots,z_{k-1}$), plus~$k-1$ (the number of vertices in~$B_P$ different from~$c_P$), plus~$n_P-2k$ (the number of vertices not in~$B_P$ in blocks descendant of~$B_P$); note that~$n_P-2k\geq 1$, since a (trivial) block that is descendant of~$B_P$ exists. Hence, we have~$|\mathcal I_P|=\frac{k-1}{2}+(k-1)+(n_P-2k)\geq \frac{2}{3}(n_P-1)$, where the inequality is satisfied because~$n_P\geq 2k+1$ and~$k\geq 3$. If~$k$ is even, we insert in~$\mathcal I_P$ all the vertices~$z_j$ with~$j$ even, except for~$z_k$. Then~$\mathcal I_P$ satisfies Property~B. The proof that~$G[\mathcal I_P\cup\{c_P,\ell_P\}]$ is an outerplane graph is again very similar to the one of Case~13.1. Also,~$|\mathcal I_P|=(\frac{k}{2}-1)+(k-1)+n_P-2k$, where the first term accounts for the number of vertices~$z_2,z_4,\dots,z_{k-2}$. Hence, we have~$|\mathcal I_P|\geq \frac{2}{3}(n_P-2)$, given that~$n_P\geq 2k+1$ and~$k\geq 2$.

We can now assume~$a\neq k$, which implies that~$a<b-1$. We distinguish three cases.

In {\em Case 13.5.1}, the length of the sequence~$z_b,\dots,z_k$ is even. Then we decide which transparent vertices are in~$\mathcal I_P$ as follows:~$z_2$ is not in~$\mathcal I_P$ and, for~$j=3,\dots,k-1$, if~$z_j$ is a transparent vertex, it is in~$\mathcal I_P$ if~$z_{j-1}\notin \mathcal I_P$, and it is not in~$\mathcal I_P$ if~$z_{j-1}\in \mathcal I_P$. We have that~$\mathcal I_P$ satisfies Property~B. The proof that~$G[\mathcal I_P\cup\{c_P,\ell_P\}]$ is an outerplane graph is very similar to the one of Case~13.1; while there we have~$z_2\notin \mathcal I_P$ because~$z_2$ is opaque, here we have~$z_2\notin \mathcal I_P$ by construction. In order to prove that~$|\mathcal I_P|\geq 2(n_P-2)/3$ we again charge each vertex of~$G_P$ not in~$\mathcal I_P$ and different from~$c_P$ and~$\ell_P$ to two vertices in~$\mathcal I_P$, so that each vertex in~$\mathcal I_P$ is charged with at most one vertex, and we again maintain an active edge~$w_iz_j$, for which~$\mathcal C^{i,j}$ and~$\mathcal I_P^{i,j}$ are defined as in Case~13.1, which satisfies the following properties:
\begin{enumerate}[(Z1)]
    \item All the vertices among~$z_2,\dots,z_j$ that are not in~$\mathcal I_P^{i,j}$ have been charged to two vertices in~$\mathcal I_P^{i,j}$, so that each vertex in~$\mathcal I_P^{i,j}$ has been charged with at most one vertex.
    \item We have that~$z_j$ is not in~$\mathcal I_P^{i,j}$.
\end{enumerate}

Note that here we do not even need to maintain a first free vertex~$\phi_1$.

Since~$z_1$ is transparent, the edge~$w_1z_2$ exists and the cycle~$w_1z_1z_2$ either delimits an internal face of~$G$, or contains inside a trivial block~$B$ that is a descendant of~$B_P$, whose parent is~$w_1$, and such that~$\ell_{B}=z_1$ and~$r_{B}=z_2$. Also, the edge~$w_2z_2$ exists and the cycle~$w_1w_2z_2$ delimits an internal face of~$G$; this comes from \cref{le:case13} if~$w_1z_1z_2$ contains inside a trivial block~$B$ as above, and it comes from \cref{le:case13} and from the fact that we are not in Case~2 if~$w_1z_1z_2$ delimits an internal face of~$G$. Indeed, if~$w_1z_1z_2$ delimits an internal face of~$G$, let~$w_1z_2v_1^*$ denote the cycle delimiting an internal face of~$G$ with~$v_1^*\neq z_1$. Then \cref{le:case13} excludes that~$v_1^*$ belongs to a block descendant of~$B_P$, and the fact that we are not in Case~2 excludes that~$v_1^*=z_3$, as~$z_2$ would have degree~$3$ in~$G$. Thus, we can charge~$z_2$ to~$w_1$ and~$w_2$ and initialize the algorithm for the charging scheme with~$w_2z_2$ as an active edge.  Properties (Z1) and (Z2) are obviously satisfied. Note that, if a trivial block~$B$ inside the cycle~$w_1z_1z_2$ exists, then its vertex different from~$w_1$ remains uncharged. 

The main iteration of the charging algorithm then proceeds as in Case~13.1, with cases that we denote as 13.5(I)--13.5(VI) and that correspond to Cases 13.1(I)--13.1(VI), respectively. However, when the algorithm encounters~$z_{b-1}$, it stops and concludes the charging as described in the following. Let~$w_m$ and~$w_{m+1}$ be the two neighbors in~$B_P$ of~$z_b$. Since~$b$ is the smallest index such that~$z_b,\dots,z_k$ are all transparent and since~$a<b-1$, we have that~$z_{b-1}$ exists and is opaque. Consider the step of the algorithm's execution in which~$z_{b-1}$ is first encountered. Let~$w_iz_j$ be the active edge before that step. Hence,~$z_{b-1}$ is not inside or on the boundary of~$\mathcal C^{i,j}$, but it is inside or on the boundary of~$\mathcal C^{i',j'}$ if~$w_{i'}z_{j'}$ denotes the active edge after~$w_iz_j$. Let~$w_iz_jv^*$ be the cycle delimiting an internal face of~$G$ that does not lie inside~$\mathcal C^{i,j}$.
\begin{itemize}
    \item The algorithm cannot be in Case 13.5(I), as in this case a vertex~$v^*=w_{i+1}$ is encountered. 
    \item The algorithm cannot be in Case 13.5(II), as in this case the vertex~$v^*=z_k$ is encountered, and we have~$b-1<k$.
    \item The algorithm might be in Case 13.5(III). Indeed, we might have~$v^*=z_{j+1}=z_{b-1}$, where~$z_{j+1}w_iw_{i+1}$ and~$z_{j+1}w_{i+1}w_{i+2}$ are cycles delimiting internal faces of~$G$, as in \cref{fig:algorithm-case13.1-case3.1}, and~$i+2\leq m$. Then we conclude the algorithm by charging~$z_{b-1}$ to~$w_{i+1}$ and~$w_{i+2}$, and by charging~$z_{b+j}$ to~$z_{b+j-1}$ and~$w_{m+j}$, for~$j=1,3,5,\dots,k-b$. It is also possible that the algorithm first encounters~$z_{b-1}$ in Case 13.5(III) if~$z_{j+1}w_iw_{i+1}$ is a cycle delimiting an internal face of~$G$, and if the vertex~$v^\circ\neq w_i$ such that the cycle~$w_{i+1}z_{j+1}v^\circ$ delimits an internal face of~$G$ belongs to a non-trivial pesky block~$B$, as in \cref{fig:algorithm-case13.1-case3.2}. In this case, since~$z_b$ is transparent and all the vertices of the cage path of~$B$ are opaque,~$z_{b-1}$ might be the right cage vertex~$r_B$ of~$B$. Then the vertices of the cage path of~$B$, including~$z_{b-1}$, can be charged to the vertices of~$B$ and we conclude the algorithm by charging~$z_{b+j}$ to~$z_{b+j-1}$ and~$w_{m+j}$, for~$j=1,3,5,\dots,k-b$. 
    \item The algorithm might be in Case 13.5(IV). Indeed, we might have~$v^*=z_{j+1}$ and~$z_{j+2}=z_{b-1}$, where~$z_{j+1}\in \mathcal I_P$, where~$z_{j+1}w_iw_{i+1}$ and~$z_{j+1}z_{j+2}w_{i+1}$ are cycles of~$G$, as in \cref{fig:algorithm-case13.1-case4.2,fig:algorithm-case13.1-case4.3}, with~$i+1\leq m$. Then we conclude the algorithm by charging~$z_{j+2}$ to~$w_{i+1}$ and~$z_{j+1}$, and by charging~$z_{b+j}$ to~$z_{b+j-1}$ and~$w_{m+j}$, for~$j=1,3,5,\dots,k-b$.
    \item The algorithm might be in Case 13.5(V). Indeed,~$v^*$ might belong to a non-trivial pesky block~$B$, as in \cref{fig:algorithm-case13.1-part-IV}, with~$z_{b-1}=r_B$. Then the vertices of the cage path of~$B$, including~$z_{b-1}$ and excluding~$\ell_B$, can be charged to the vertices of~$B$ different from~$w_i$ and we conclude the algorithm by charging~$z_{b+j}$ to~$z_{b+j-1}$ and~$w_{m+j}$, for~$j=1,3,5,\dots,k-b$. 
    \item Finally, the algorithm might be in Case 13.5(VI). Recall that, in this case,~$v^*$ belongs to a trivial block~$B$ that is a descendant of~$B_P$, whose parent is~$w_i$, and such that~$\ell_B=z_j$, as in \cref{fig:algorithm-case13.1-part-V}. This results in a sequence~$B_i,\dots,B_{i+h}$ of trivial blocks descendants of~$B_P$ such that, for~$l=0,1\dots,h$, the block~$B_{i+l}$ has~$w_{i+l}$ as parent, has~$d_{i+l}$ as vertex different from~$w_{i+l}$, has~$\ell_{B_{i+l}}=z_{j+l}$, and has~$r_{B_{i+l}}=z_{j+l+1}$.  Since the vertices~$z_{j+1},\dots,z_{j+h}$ are opaque, we might have~$z_{b-1}=z_{j+h}$ or~$z_{b-1}=z_{j+h+1}$. In the former (latter) case, for~$l=1,\dots,h$ (resp.\ for~$l=1,\dots,h+1$), we charge~$z_{j+l}$ to~$w_{i+l}$ and to~$d_{i+l-1}$, and we conclude the algorithm by charging~$z_{b+j}$ to~$z_{b+j-1}$ and~$w_{m+j}$, for~$j=1,3,5,\dots,k-b$. The algorithm might also first encounter~$z_{b-1}$ in Case 13.5(VI) if the cycle~$w_{i+h}w_{i+h+1}z_{j+h+1}$ delimits an internal face of~$G$ and the cycle~$w_{i+h+1}z_{j+h+1}u^*$ delimiting an internal face of~$G$ with~$u^*\neq w_{i+h}$ has~$u^*$ in a non-trivial pesky block~$B'$ that is a descendant of~$B_P$, whose parent is~$w_{i+h+1}$, such that~$\ell_{B'}=z_{j+h+1}$, as in \cref{fig:algorithm-case13.1-case6.5}, and such that~$r_{B'}=z_{b-1}$. Then the vertices of the cage path of~$B'$ can be charged to the vertices of~$B'$, and we conclude the algorithm by charging~$z_{b+j}$ to~$z_{b+j-1}$ and~$w_{m+j}$, for~$j=1,3,5,\dots,k-b$. 
\end{itemize}


In {\em Case 13.5.2}, the length of the sequence~$z_1,\dots,z_a$ is even. This case can be handled symmetrically to Case 13.5.1. 

Finally, in {\em Case 13.5.3}, we have that~$a$ and~$k-b+1$ are both odd. We let~$z_2,z_4,\dots,z_{a-1}$ be in~$\mathcal I_P$ (and~$z_3,z_5,\dots,z_a$ not be in~$\mathcal I_P$). Then, for~$j=a+2,\dots,k-1$, if~$z_j$ is a transparent vertex, it is in~$\mathcal I_P$ if~$z_{j-1}\notin \mathcal I_P$, and it is not in~$\mathcal I_P$ if~$z_{j-1}\in \mathcal I_P$. We have that~$\mathcal I_P$ satisfies Property~C. The proof that~$G[\mathcal I_P\cup\{c_P,r_P\}]$ is an outerplane graph is again very similar to the one of the previous cases. Note that~$z_{k-1}\notin \mathcal I_P$, given that~$k-b+1$ is odd. In order to prove that~$|\mathcal I_P|\geq 2(n_P-2)/3$ we again charge each vertex of~$G_P$ not in~$\mathcal I_P$ and different from~$c_P$ and~$r_P$ to two vertices in~$\mathcal I_P$, so that each vertex in~$\mathcal I_P$ is charged with at most one vertex, and we maintain an active edge~$w_iz_j$, which satisfies Properties (Z1) and (Z2). Note that, since~$a\geq 2$ and~$a$ is odd, we have~$a\geq 3$. Since~$z_1$,~$z_2$, and~$z_3$ are transparent, the cycles~$w_1z_1z_2$,~$w_1w_2z_2$,~$w_2z_2z_3$, and~$w_2w_3z_3$ exist; the second and fourth cycles delimit internal faces of~$G$, while each of the first and third cycles either delimits an internal face of~$G$ or contains inside a trivial block descendant of~$B_P$. Regardless, we can charge~$z_1$ to~$w_1$ and~$z_2$, and~$z_3$ to~$w_2$ and~$w_3$, and initialize the algorithm for the charging scheme with~$w_3z_3$ as an active edge. Properties (Z1) and (Z2) are obviously satisfied. 

The algorithm then proceeds exactly as in Case~13.5.1, except for the fact that, when it encounters and charges $z_{b-1}$ to two vertices in $\mathcal I_P$, it terminates by charging~$z_{b+j}$ to~$z_{b+j-1}$ and~$w_{i+j}$, for~$j=1,3,5,\dots,k-b-1$, rather than for~$j=1,3,5,\dots,k-b$; note that $z_k=r_P$ does not need to be charged.
\end{proof}


By employing~\cref{le:final-case}, we can now apply induction in the same way as described after the proof of~\cref{le:cushy}, with~$\mathcal I_P$,~$c_P$,~$\ell_P$,~$r_P$,~$G_P$ in place of~$\mathcal I_B$,~$c_B$,~$\ell_B$,~$r_B$,~$G_B$, respectively, so to find a good set for~$G$. This concludes the proof of~\cref{th:main}.  

\section{Conclusions} \label{se:conclusions}

In this paper, we proved that every~$n$-vertex~$2$-outerplane graph has a set with at least~$2n/3$ vertices that induces an outerplane graph. A natural goal for future research is to prove analogous results for graph classes broader than the one of the~$2$-outerplane graphs. In particular, our techniques could be useful for proving that every~$n$-vertex~$k$-outerplane graph has a set with at least~$cn$ vertices that induces an outerplane graph, for some (small)~$k\geq 3$ and some constant~$c>11/21$. We recall that proving that  every~$n$-vertex plane graph has a set with at least~$cn$ vertices that induces an outerplane graph, for some~$c>3/5$, would improve the best known lower bound for the notorious Albertson and Berman's conjecture on the size of an induced forest that one is guaranteed to find in a planar graph~\cite{albertson1979conjecture}.

\bibliography{bibliography}

\end{document}